\documentclass[a4paper,11pt]{amsart}

\newfont{\cyr}{wncyr10 scaled 1100}

\usepackage[left=2.7cm,right=2.7cm,top=3.5cm,bottom=3cm]{geometry}
\usepackage{amsthm,amssymb,amsmath,amsfonts,mathrsfs,amscd,yfonts}
\usepackage{turnstile}
\usepackage[latin1]{inputenc}
\usepackage[all]{xy}
\usepackage{latexsym}
\usepackage{longtable}
\usepackage[usenames,dvipsnames]{color}
\usepackage{listings}
\usepackage{mathtools}
\usepackage{tikz}
\usetikzlibrary{matrix,arrows,decorations.pathmorphing}
\usepackage{bm}
\usepackage[shortlabels]{enumitem}
\usepackage{hyperref}
\usepackage{bm}
\hypersetup{colorlinks, linkcolor=blue, citecolor=black, urlcolor=black}

\theoremstyle{plain}

\newtheorem*{hyp}{Hypothesis}
\newtheorem{theorem}{Theorem}[section]
\newtheorem{corollary}[theorem]{Corollary}
\newtheorem{lemma}[theorem]{Lemma}
\newtheorem{proposition}[theorem]{Proposition}
\newtheorem{propo}[theorem]{Proposition}

\newtheorem{conj}[theorem]{Conjecture}

\theoremstyle{definition}
\newtheorem{definition}[theorem]{Definition}
\newtheorem{defi}[theorem]{Definition}

\newtheorem{examplewr}[theorem]{Example}

\theoremstyle{remark}
\newtheorem{obswr}[theorem]{Observation}
\newtheorem{remarkwr}[theorem]{Remark}
\newtheorem*{remark-intro}{Remark}
\newtheorem{hypotheses}[theorem]{Hypotheses}

\newenvironment{remark}{\begin{remarkwr}\begin{upshape}}{\end{upshape}\end{remarkwr}}

\newcommand{\bb}{\mathbb}
\newcommand{\frk}{\mathfrak}
\newcommand{\cl}{\mathcal}

\newcommand{\univ}{\kappa}

\newcommand{\kap}{{\kappa_{\psi,g,h}}}
\newcommand{\kapone}{{\kappa_{\psi,g,h,1}}}
\newcommand{\kapinfty}{{\kappa_{\psi,g,h,\infty}}}
\newcommand{\kapinftyone}{{\kappa_{\psi,g,h,1,\infty}}}
\newcommand{\kapinftyn}{{\kappa_{\psi,g,h,n,\infty}}}
\newcommand{\kapinftynq}{{\kappa_{\psi,g,h,nq,\infty}}}
\newcommand{\kapn}{{\kappa_{\psi,g,h,n}}}
\newcommand{\kapnq}{{\kappa_{\psi,g,h,nq}}}
\newcommand{\tkapn}{{\tilde{\kappa}_{\psi,g,h,n}}}
\newcommand{\tkapnq}{{\tilde{\kappa}_{\psi,g,h,nq}}}

\newcommand{\Vrep}{V^\psi_{g,h}}
\newcommand{\Trep}{T^\psi_{g,h}}
\newcommand{\Arep}{A^\psi_{g,h}}

\DeclareMathOperator{\et}{et}
\DeclareMathOperator{\id}{Id}

\DeclareMathOperator{\Ind}{Ind}

\DeclareMathOperator{\lcm}{lcm}
\DeclareMathOperator{\bal}{bal}

\DeclareMathOperator{\Gr}{Gr}

\DeclareMathOperator{\im}{im}
\DeclareMathOperator{\Char}{Char}
\DeclareMathOperator{\tors}{tors}

\DeclareMathOperator{\ac}{ac}

\DeclareMathOperator{\cor}{cor}

\DeclareMathOperator{\Tsym}{Tsym}
\DeclareMathOperator{\Symm}{Symm}

\newcommand{\cF}{\mathcal F}

\newcommand{\cW}{\mathcal W}

\newcommand{\Q}{\mathbb{Q}}
\newcommand{\Z}{\mathbb{Z}}

\newcommand{\C}{\mathbb{C}}

\newcommand{\Sel}{\mathrm{Sel}}

\newcommand{\Gal}{\mathrm{Gal\,}}
\newcommand{\GL}{\mathrm{GL}}

\newcommand{\Tr}{\mathrm{Tr}}

\newcommand{\cond}{\mathrm{cond}}

\newcommand{\Fr}{\mathrm{Fr}}

\newcommand{\ord}{{\mathrm{ord}}}
\newfont{\gotip}{eufb10 at 12pt}

\newcommand{\cO}{{\mathcal O}}

\newcommand{\hf}{{\mathbf{f}}}
\newcommand{\hg}{{\mathbf{g}}}
\newcommand{\hh}{{\mathbf h}}

\DeclareMathOperator{\Hom}{Hom}

\newcommand{\res}{\mathrm{res}}

\numberwithin{equation}{section}

\include{thebibliography}

\begin{document}



\title{The diagonal cycle Euler system for ${\rm GL}_2\times{\rm GL}_2$}

\author{Ra\'ul Alonso, Francesc Castella, and \'Oscar Rivero}

\begin{abstract}
We construct an anticyclotomic Euler system for the Rankin--Selberg convolutions of two modular forms, using $p$-adic families of generalized Gross--Kudla--Schoen diagonal cycles. As applications of this construction, we prove new results on the Bloch--Kato conjecture in analytic ranks zero and one, and a divisibility towards an Iwasawa main conjecture.
\end{abstract}


\date{\today}

\address{R. A.: Department of Mathematics, Princeton University, Fine Hall, Princeton, NJ 08544-1000, USA}
\email{raular@math.princeton.edu}

\address{F. C.: Department of Mathematics, University of California, Santa Barbara, CA 93106, USA}
\email{castella@ucsb.edu}

\address{O. R.: Mathematics Institute, Zeeman Building, University of Warwick, Coventry CV4 7AL, UK}
\email{Oscar.Rivero-Salgado@warwick.ac.uk}


\maketitle

\setcounter{tocdepth}{1}
\tableofcontents

\section{Introduction}
\renewcommand{\thetheorem}{\Alph{theorem}}

In this paper we study the anticyclotomic Iwasawa theory of Rankin--Selberg convolutions of two modular forms using a new Euler system arising from $p$-adic families of diagonal cycles. By an application of Kolyvagin's methods, our construction yields results towards the Bloch--Kato conjecture and the Iwasawa main conjecture in this setting.

\subsection{Statement of the main results}\label{subsec:statements}

Let $g\in S_l(N_g,\chi_g)$ and $h\in S_m(N_h,\chi_h)$ be newforms of weights $l\geq m\geq 2$ of the same parity and nebentypus $\chi_g$ and $\chi_h$. Let $K/\Q$ be an imaginary quadratic field of discriminant $-D<0$. Let $k>0$ be an even integer, and let $\psi$ be a Hecke character of $K$ of infinity type $(1-k,0)$, conductor $\frk{f}$, and central character 
\[
\varepsilon_\psi=\bar{\chi}_g\bar{\chi}_h.
\] 

Fix an odd prime $p\nmid N_gN_h$ such that $(\frk{f},p)=1$ and an embedding $\iota_p:\overline{\Q}\hookrightarrow\overline{\Q}_p$, and let $E=L_{\mathfrak{P}}$ be a finite extension of $\Q_p$ containing the image under $\iota_p$ of the values of $\psi$ and the Fourier coefficients of $g$ and $h$. We consider the $E$-valued $G_K$-representation
\[
\Vrep:=V_g\otimes V_h(\psi_{\frk{P}}^{-1})(1-c),
\]
where $c=(k+l+m-2)/2$, $V_g$ and $V_h$ are the (dual of Deligne's) $p$-adic Galois representations associated to $g$ and $h$, respectively, and $\psi_{\frk{P}}$ is a $p$-adic Galois character attached to $\psi$.

The cyclotomic Iwasawa theory of $V_g\otimes V_h$ has been extensively studied in a series of works of Lei--Loeffler--Zerbes \cite{LLZ0, LLZ} and Kings--Loeffler--Zerbes \cite{KLZ,KLZ0}, among others (\cite{BKLV}, \cite{BL-IMRN}, etc.). The key tool exploited in these works is the Euler system of Beilinson--Flach classes, a system of cohomology classes 
arising from certain special elements (introduced by Beilinson \cite{beilinson}, and further studied by Flach \cite{flach-sym} and  Bertolini--Darmon--Rotger \cite{BDR1,BDR2}) in the $K_1$ of products of two modular curves.

In contrast, the anticyclotomic Iwasawa theory of $V_{g}\otimes V_h$ (or rather of its conjugate self-dual twists, such as $\Vrep$) appears to not have been studied before. The principal contribution of this paper is the construction of an anticyclotomic Euler system for $V_{g,h}^\psi$. As stated in Theorem~\ref{thmA} below (which corresponds to Theorem~\ref{principal:wild} in the body of the paper), for general weights $(k,l,m)$ our construction requires the additional assumptions that $p$ splits in $K$ and $p\nmid h_K$ (the class number of $K$), and that both $g$ and $h$ are ordinary at $p$, but note that for $(k,l,m)=(2,l,l)$ Theorem~\ref{principal:tame} contains a version of our main result without these additional hypotheses.

\begin{theorem}
\label{thmA}
Suppose that $p$ splits in $K$ and $p\nmid h_K$, and that both $g$ and $h$ are ordinary at $p$. 
Let $\mathcal{S}$ be the set of all squarefree products of primes split $q$ in $K$ and coprime to $DN_gN_h\frk{f}$, and denote by $K[n]$ the maximal $p$-extension of $K$ inside the ring class field of conductor $n$. Then there exists a family of cohomology classes
\[
\kappa_{\psi,g,h,np^r}\in H^1(K[np^r],\Trep)
\]
for all $n\in\mathcal S$ and $r\geq 0$, where $\Trep$ is a fixed $G_K$-stable lattice inside $\Vrep$, such that for all $nq\in\mathcal{S}$ with $q$ prime, we have
\[
{\rm cor}_{K[nqp^r]/K[np^r]}(\kappa_{\psi,g,h,nqp^r})=\begin{cases}
P_{\frk{q}}({\rm Fr}_{\frk{q}}^{-1})\,\kappa_{\psi,g,h,np^r}&\textrm{if $q\neq p$,}\\
\kappa_{\psi,g,h,np^r}&\textrm{if $q=p$},\\
\end{cases}
\]
where $\frk{q}$ is any of the primes of $K$ above $q$, and $P_{\frk{q}}({\rm Fr}_{\frk{q}}^{-1})=\det(1-{\rm Fr}_{\frk{q}}^{-1}X\vert(V_{g,h}^\psi)^\vee(1))$.
\end{theorem}

The construction of this Euler system, which is taken up in the first part of the paper, is based on the diagonal classes studied by Darmon--Rotger \cite{DR1,DR2,DR3} and Bertolini--Seveso--Venerucci \cite{BSV}, extending earlier constructions due to Gross--Kudla \cite{gross-kudla} and Gross--Schoen \cite{gross-schoen}. Roughly speaking, our classes $\kappa_{\psi,g,h,np^r}$ are suitable modifications of diagonal classes for the triples $(\tilde{\theta}_{\psi,np^r},g,h)$, where $\tilde{\theta}_{\psi,np^r}$ is a certain deformation of the theta series associated to $\psi$, and the main difficulty in the proof of Theorem~\ref{thmA} is in establishing the Euler system norm relations.

The main results in the second part of the paper are the proof of new cases of the Bloch--Kato conjecture for $\Vrep$ in analytic rank zero and a divisibility towards the Iwasawa main conjecture for $\Vrep$. These are obtained by applying Kolyvagin's methods (in the form recently developed by Jetchev--Nekov{\'a}{\v{r}}--Skinner \cite{JNS} in the anticyclotomic setting) to our Euler system. In the results that follow, we use ``big image'' to refer to Hypothesis~(HS) in $\S\ref{sec:ES}$, for which sufficient conditions are given in $\S\ref{subsec:verify}$.

\begin{theorem}
\label{thmB}
Suppose that:
\begin{itemize}
\item[{\rm (a)}] $g$ and $h$ are ordinary at $p$, non-Eisenstein and $p$-distinguished,
\item[{\rm (b)}] $p$ splits in $K$, 
\item[{\rm (c)}] $p$ does not divide the class number of $K$,
\item[{\rm (d)}] $V_{g,h}^\psi$ has big image.
\end{itemize}
Let
\[
\kap:=\kapone.
\]
If $l-m<k<l+m$, then the following implication holds:
\[
\kap\neq 0\quad\Longrightarrow\quad{\rm dim}_E\,\Sel(K,\Vrep)=1,
\]
where $\Sel(K,\Vrep)\subset H^1(G_K,\Vrep)$ is the Bloch--Kato Selmer group.
\end{theorem}

\begin{remark-intro}
\begin{itemize}
\item[(1)]
For $k=l=m=2$, together with the Gross--Zagier formula for diagonal cycles by Yuan--Zhang--Zhang \cite{YZZ}, Theorem~\ref{thmB} supports the Bloch--Kato conjecture for $V^\psi_{g,h}$ in analytic rank one, reducing it to the expected injectivity of the $p$-adic \'{e}tale  Abel--Jacobi map. 
\item[(2)]
Still in the case $k=l=m=2$, combined with the $p$-adic Gross--Zagier formula for diagonal cycles in forthcoming work of Hsieh--Yamana \cite{HY}, Theorem~\ref{thmB} establishes some cases of Perrin-Riou's $p$-adic Beilinson conjecture in analytic rank one.
\item[(3)] In general, by the main result of \cite{DR1}, the nonvanishing of $\kappa_{\psi,g,h}$ also follows from the nonvanishing of a special value of the triple product $p$-adic $L$-function $\mathscr{L}_p(\hf,g,h)$ introduced below.
\end{itemize}
\end{remark-intro}


In analytic rank zero, we get unconditional applications to the Bloch--Kato conjecture. Let $f=\theta_\psi\in S_k(N_\psi,\varepsilon_\psi)$ be the theta series associated to $\psi$, let $\varepsilon_\ell(\Vrep)$ be the epsilon factor of the Weil--Deligne representation associated to the restriction of $V_{f}\otimes V_g\otimes V_h(1-c)$ to $G_{\Q_\ell}$, and put $N=\lcm(N_\psi,N_g,N_h)$.

\begin{theorem}
\label{thmC}
Let the hypotheses be as in Theorem~\ref{thmB}, and assume in addition that
\begin{itemize}
\item $\varepsilon_\ell(\Vrep)=+1$ for all primes $\ell\mid N$,
\item ${\rm gcd}(N_\psi,N_g,N_h)$ is squarefree.
\end{itemize}
If $k\geq l+m$, then 
\[
L(\Vrep,0)\neq 0\quad\Longrightarrow\quad\Sel(K,\Vrep)=0,
\]
and hence the Bloch--Kato conjecture for $\Vrep$ holds in analytic rank zero.
\end{theorem}

\begin{remark-intro}
Here $L(\Vrep,s)$ is the triple product $L$-function introduced by Garrett, Piatetski--Shapiro and Rallis, which satisfies a functional equation relating its values at $s$ and $-s$. When $k\geq l+m$, the local root number condition in Theorem~\ref{thmC} implies that the sign in this functional equation is $+1$, and so the central $L$-values $L(\Vrep,0)$ are expected to be generically nonzero.
\end{remark-intro}


A third application is to the anticyclotomic Iwasawa main conjectures for Rankin--Selberg convolutions. Let $(\hf,\hg,\hh)$ be a triple of $p$-adic Hida families. In \cite{Hs}, Hsieh has constructed a square-root triple product $p$-adic $L$-function $\mathscr{L}_p(\hf,\hg,\hh)$ whose square interpolates the central values of the triple product $L$-function attached to the classical specializations of $(\hf,\hg,\hh)$ to weights $(k_1,k_2,k_3)$ with $k_1\geq k_2+k_3$. 
Letting $\hg$ and $\hh$ be the Hida families passing through the ordinary $p$-stabilizations of $g$ and $h$, respectively, we obtain an element
\[
\mathscr{L}_p(\hf,g,h)\in\Lambda_{\hf}
\]
interpolating a square-root of the above central $L$-values for the specializations of $\hf$ to weights $k\geq l+m$, where $\Lambda_{\hf}$ is the finite flat extension of $\Lambda=\mathbb Z_p[[1+p\mathbb Z_p]]$ generated by the coefficients of $\hf$. Greenberg's generalization of the Iwasawa main conjectures \cite{Gr94} predicts that $\mathscr{L}_p(\hf,g,h)^2$ generates the $\Lambda_{\hf}$-characteristic ideal of a certain torsion Selmer group $X_{\mathcal F}(\mathbb A_{\hf gh}^\dag)$. 
We also show that our classes $\kapn$ are universal norms in the $p$-direction, therefore giving rise in particular to an Iwasawa cohomology class
\[
\kapinfty\in H^1_{\rm Iw}(K_\infty,T_{g,h}^\psi)
\]
for the anticyclotomic $\mathbb{Z}_p$-extension $K_\infty/K$. 

The class $\kapinfty$ is associated with the triple $(\hf,g,h)$, where $\hf=\hf_\psi$ is a CM Hida family attached to $\psi$
for which 
$\Lambda_{\hf}\cong\Lambda_{\ac}$, the anticyclotomic Iwasawa algebra.
Assuming the non-triviality of $\kapinfty$, we can prove the following result towards the Iwasawa main conjecture for $\mathscr{L}_p(\hf,g,h)^2$. 

\begin{theorem}
\label{thmD}
Let $\hf=\hf_\psi$, and suppose that:
\begin{itemize}
\item[{\rm (a)}] $g$ and $h$ are ordinary at $p$, non-Eisenstein and $p$-distinguished,
\item[{\rm (b)}] $p$ splits in $K$,
\item[{\rm (c)}] $p$ does not divide the class number of $K$,
\item[{\rm (d)}] $V_{g,h}^\psi$ has big image,
\item[{\rm (e)}] $\varepsilon_\ell(\Vrep)=+1$ for all primes $\ell\mid N$,
\item[{\rm (f)}] ${\rm gcd}(N_\psi,N_g,N_h)$ is squarefree.
\end{itemize}
If $\kapinfty$ is not $\Lambda_{\ac}$-torsion, then the module $X_{\mathcal F}(\mathbb A_{\hf gh}^\dag)$ is $\Lambda_{\ac}$-torsion, and
\[
\Char_{\Lambda_{\ac}}(X_{\mathcal F}(\mathbb A_{\hf gh}^{\dag})) \supset (\mathscr{L}_p(\hf,g,h)^2)
\]
in $\Lambda_{\ac}\otimes_{\mathbb Z_p}\mathbb Q_p$.
\end{theorem}


\begin{remark-intro}
The classes $\kapn$ may be viewed as a counterpart in the study of the arithmetic of $V_{g,h}^\psi$ to systems of Heegner points and Heegner cycles for individual modular forms. It would be interesting to see whether the methods of Cornut--Vatsal can be extended to establish the non-triviality of $\kapinfty$.
\end{remark-intro}

\begin{remark-intro} 
The ``big image'' hypothesis on $V_{g,h}^\psi$ excludes some cases of arithmetic interest; notably, the case in which $h=g^*$ is the dual of $g$ (assuming $\psi$ has trivial central character) 
is excluded from our applications in this paper. 
We study this case in \cite{ACR2}, where building on (a suitable projection of) the classes $\kappa_{\psi,g,g^*,n}$ constructed in this paper, we obtain a new anticyclotomic Euler systems for twists of the three-dimensional $G_K$-representation ${\rm ad}^0(V_g)$, with applications to the Bloch--Kato conjecture in rank zero and the Iwasawa main conjecture in this setting. 
\end{remark-intro}

\begin{remark-intro}
As already noted, the anticyclotomic Euler system classes constructed in this paper arise from diagonal classes attached to triples $(f,g,h)$ of modular forms in which $f$ varies over certain CM form by $K$. A modification of this construction with $g$ and $h$ varying among certain CM forms for the same imaginary quadratic field $K$ gives rise to a new anticyclotomic Euler system for twists of $V_f\vert_{G_K}$. This construction, and its arithmetic applications, is studied in \cite{Do-PhD,CaDo}.
\end{remark-intro}

\subsection{Acknowledgements}
It is a pleasure to thank Chris Skinner, as well as K\^azim B\"u\-y\"uk\-bo\-duk, Daniel Disegni, Ming-Lun Hsieh, Antonio Lei, Victor Rotger and Shou-Wu Zhang, for several helpful conversations related to this work. This project has received funding from the ERC under the European Union's Horizon 2020 research and innovation programme (grant agreement No 682152). During the preparation of this paper, F.C. was partially supported by the NSF grants DMS-1946136 and DMS-2101458; O.R. was supported by a Royal Society Newton International Fellowship and by ``la Caixa'' Foundation (grant LCF/BQ/ES17/11600010).

\part{The diagonal cycle Euler system}



\renewcommand{\thetheorem}{\arabic{section}.\arabic{theorem}}
\section{Preliminaries}\label{subsec:prelim}

In this section be begin by discussing our conventions regarding modular curves and Hecke operators, for which we shall largely follow \cite[\S2]{Kato} and \cite[\S2]{BSV}.

\subsection{Modular curves}

Given integers  $M\geq 1$, $N\geq 1$, $m\geq 1$ and $n\geq 1$ with
$M+N\geq 5$, we denote by $Y(M(m),N(n))$ the affine modular curve over $\bb{Z}[1/MNmn]$ representing the functor taking a $\bb{Z}[1/MNmn]$-scheme $S$ to the set of isomorphism classes of 5-tuples $(E,P,Q,C,D)$, where:
\begin{itemize}
\item $E$ is an elliptic curve over $S$,
\item $P$ is an $S$-point of $E$ of order $M$,
\item $Q$ is an $S$-point of $E$ of order $N$,
\item $C$ is a cyclic order-$Mm$ subgroup of $E$ defined over $S$ and containing $P$,
\item $D$ is a cyclic order-$Nn$ subgroup of $E$ defined over $S$ and containing $Q$,
\end{itemize}
and such that $C$ and $D$ have trivial intersection. If either $m=1$ or $n=1$ we omit it from the notation, and we will often write $Y_1(N)$ for $Y(1,N)$.

We will denote by
$$
E(M(m),N(n))\rightarrow Y(M(m),N(n))
$$
the universal elliptic curve over $Y(M(m),N(n))$.

Define the modular group
$$
\Gamma(M(m),N(n))=\left\{\begin{pmatrix} a & b \\ c & d\end{pmatrix}\in \text{SL}_2(\bb{Z}): a\equiv 1\,(M), b\equiv 0\,(Mm), c\equiv 0\,(Nn),d\equiv 1\,(N)\right\}.
$$
Then, letting $\mathcal{H}$ be the Poincar\'e upper half-plane, we have the complex uniformization
\begin{equation}\label{eq:unif}
Y(M(m),N(n))(\bb{C})\cong (\bb{Z}/M\bb{Z})^\times\times\Gamma(M(m),N(n))\backslash\mathcal{H},
\end{equation}
with a pair $(a,\tau)$ on the right-hand side corresponding
to the isomorphism class of the 5-tuple $(\bb{C}/\bb{Z}+\bb{Z}\tau, a\tau/M,1/N,\langle \tau/Mm\rangle, \langle 1/Nn\rangle)$.

If $r\geq 1$ is an integer, there is an isomorphism of $\bb{Z}[1/MNmnr]$-schemes
$$
\varphi_r: Y(M(m),N(nr))\xrightarrow{\simeq} Y(M(mr),N(n))
$$
defined in terms of moduli by
$$
(E,P,Q,C,D)\mapsto (E', P', Q', C',D'),
$$
where $E'=E/NnD$, $P'$ is the image of $P$ in $E'$, $Q'$ is the image of $r^{-1}(Q)\cap D$ in $E'$, $C'$ is the image of $r^{-1}(C)$ in $E'$, and $D'$ is the image of $D$ in $E'$. Under the complex uniformizations $(\ref{eq:unif})$, the isomorphism $\varphi_r$ sends $(a,\tau)\mapsto (a,r\cdot\tau)$. If
$$
\varphi_r^\ast(E(M(mr),N(n)))\rightarrow Y(M(m),N(nr))
$$
denotes the base change of $E(M(mr),N(n))\rightarrow Y(M(mr),N(n))$ under $\varphi_r$, there is a natural degree-$r$ isogeny
$$
\lambda_r:E(M(m),N(nr))\rightarrow \varphi_r^\ast(E(M(mr),N(n))).
$$

\subsection{Degeneracy maps}

With the same notations as above, we have natural degeneracy maps
\begin{align*}
& Y(M(m),Nr(n))\xrightarrow{\mu_r} Y(M(m),N(nr))\xrightarrow{\nu_r} Y(M(m),N(n)), \\
& Y(Mr(m),N(n))\xrightarrow{\check{\mu}_r} Y(M(mr),N(n))\xrightarrow{\check{\nu}_r} Y(M(m),N(n)),
\end{align*}
forgetting the extra level structure, e.g.
\begin{align*}
& \mu_r(E,P,Q,C,D)=(E,P,r\cdot Q, C,D), \\
& \nu_r(E,P,Q,C,D)=(E,P,Q,C,rD).
\end{align*}
We also define degeneracy maps
\begin{equation}\label{eq:pi12}
\begin{aligned}
& \pi_1: Y(M(m),Nrs(nt))\rightarrow Y(M(m),N(ns)), \\
& \pi_2: Y(M(m),Nrs(nt))\rightarrow Y(M(m),N(ns)),
\end{aligned}
\end{equation}
acting on the moduli space by
\begin{align*}
& \pi_1(E,P,Q,C,D)=(E,P,rs\cdot Q,C, rtD), \\
& \pi_2(E,P,Q,C,D)=(E',P',Q',C',D'),
\end{align*}
where $E'=E/NnsD$, $P'$ is the image of $P$ in $E'$, $Q'$ is the image of $t^{-1}(s\cdot Q)\cap D$ in $E'$, $C'$ is the image of $C$ in $E'$ and $D'$ is the image of $D$ in $E'$. Under the complex uniformizations (\ref{eq:unif}), the maps $\pi_1$ and $\pi_2$ correspond to the identity and to multiplication by $rt$, respectively, on $\mathcal{H}$. It is straightforward to check that the maps $\pi_1$ and $\pi_2$ are given by the compositions
\begin{align*}
& Y(M(m),Nrs(nt))\xrightarrow{\mu_{rs}} Y(M(m),N(nrst))\xrightarrow{\nu_{rt}} Y(M(m),N(ns)), \\
& Y(M(m),Nrs(nt))\xrightarrow{\mu_{rs}} Y(M(m),N(nrst))\xrightarrow{\varphi_{rt}} Y(M(mrt),N(ns))\xrightarrow{\check{\nu}_{rt}} Y(M(m),N(ns)),
\end{align*}
respectively.

\subsection{Relative Tate modules}\label{subsubsec:relT}

Fix a prime $p$. Let $S$ be a $\bb{Z}[1/MNmnp]$-scheme and let
$$
v:E(M(m),N(n))_S\rightarrow Y(M(m),N(n))_S
$$
be the structural morphism.
For every $\mathbb Z[1/MNmnp]$-scheme $X$, denote by $A=A_X$ either the locally constant constructible sheaf $\mathbb Z/p^t(j)$ or the locally constant $p$-adic sheaf $\mathbb Z_p(j)$ on $X_{\et}$, for fixed $t\geq 1$ and $j\in\Z$. Set
$$
\mathscr{T}_{M(m),N(n)}(A)=R^1v_\ast \bb{Z}_p(1)\otimes_{\bb{Z}_p} A\quad\text{and}\quad \mathscr{T}_{M(m),N(n)}^\ast(A)=\Hom(\mathscr{T}_{M(m),N(n)}(A),A).
$$
In particular, in the case $A=\bb{Z}_p$ this gives the relative Tate module of the universal elliptic curve and its dual, respectively; in this case, we will often drop $A$ from the notation.

From the proper base change theorem, both $\mathscr T_{M(m),N(n)}(A)$ and $\mathscr T_{M(m),N(n)}^*(A)$ are locally constant $p$-adic sheaves on $Y(M(m),N(n))_S$ of formation compatible with base changes along morphisms of $\bb{Z}[1/MNmnp]$-schemes $S'\rightarrow S$.

For every integer $r\geq 0$, define
\[
\mathscr L_{M(m),N(n),r}(A) = \Tsym_A^r \mathscr T_{M(m),N(n)}(A), \quad \mathscr S_{M(m),N(n),r}(A) = \Symm_A^r \mathscr T_{M(m),N(n)}^*(A),
\]
where, for any finite free module $M$ over a profinite $\mathbb Z_p$-algebra $R$, one denotes by $\Tsym_R^r M$ the $R$-submodule of symmetric tensors in $M^{\otimes r}$ and by $\Symm_R^r M$ the maximal symmetric quotient of $M^{\otimes r}$.

When the level of the modular curve $Y(M(m),N(n))_S$ is clear, we may use the simplified notations \[ \mathscr L_r(A) = \mathscr L_{M(m),N(n),r}(A), \quad \mathscr L_r = \mathscr L_r(\mathbb Z_p), \quad \mathscr S_r(A) = \mathscr S_{M(m),N(n),r}(A), \quad \mathscr S_r = \mathscr S_r(\mathbb Z_p). \]

\subsection{Hecke operators}\label{subsec:hecke}

Let $\mathscr{F}_{M(m),N(n)}^r$ denote either $\mathscr{L}_{M(m),N(n),r}(A)$ or $\mathscr{S}_{M(m),N(n),r}(A)$ and let $q$ be a rational prime. Then there are natural isomorphisms of sheaves
\begin{equation}\label{eq:iso-q}
\nu_q^\ast(\mathscr{F}_{M(m),N(n)}^r)\cong \mathscr{F}_{M(m),N(nq)}^r\quad\text{and}\quad \check{\nu}_q^\ast(\mathscr{F}_{M(m),N(n)}^r)\cong \mathscr{F}_{M(mq),N(n)}^r,
\end{equation}
and therefore pullback morphisms
\begin{align*}
& H^i_{\et}(Y(M(m),N(n))_S,\mathscr{F}_{M(m),N(n)}^r)\xrightarrow{\nu_q^\ast} H^i_{\et}(Y(M(m),N(nq))_S,\mathscr{F}_{M(m),N(nq)}^r), \\
& H^i_{\et}(Y(M(m),N(n))_S,\mathscr{F}_{M(m),N(n)}^r)\xrightarrow{\check{\nu}_q^\ast} H^i_{\et}(Y(M(mq),N(n))_S,\mathscr{F}_{M(mq),N(n)}^r),
\end{align*}
and traces
\begin{equation}\label{eq:trace}
\begin{aligned}
& H^i_{\et}(Y(M(m),N(nq))_S,\mathscr{F}_{M(m),N(nq)}^r)\xrightarrow{\nu_{q\ast}} H^i_{\et}(Y(M(m),N(n))_S,\mathscr{F}_{M(m),N(n)}^r), \\
& H^i_{\et}(Y(M(mq),N(n))_S,\mathscr{F}_{M(mq),N(n)}^r)\xrightarrow{\check{\nu}_{q\ast}} H^i_{\et}(Y(M(m),N(n))_S,\mathscr{F}_{M(m),N(n)}^r).
\end{aligned}
\end{equation}

Also, the isogeny $\lambda_q$ induces morphisms of sheaves
$$
\lambda_{q\ast} : \mathscr{F}_{M(m),N(nq)}^r\rightarrow \varphi_q^\ast(\mathscr{F}_{M(mq),N(n)}^r) \quad\text{and}\quad
\lambda_{q}^\ast : \varphi_q^\ast(\mathscr{F}_{M(mq),N(n)}^r)\rightarrow \mathscr{F}_{M(m),N(nq)}^r.
$$
These morphisms allow us to define
\begin{align*}
& \Phi_{q\ast}:H^i_{\et}(Y(M(m),N(nq))_S,\mathscr{F}_{M(m),N(nq)}^r)\rightarrow H^i_{\et}(Y(M(mq),N(n))_S,\mathscr{F}_{M(mq),N(n)}^r), \\
& \Phi_{q}^\ast:H^i_{\et}(Y(M(mq),N(n))_S,\mathscr{F}_{M(mq),N(n)}^r)\rightarrow H^i_{\et}(Y(M(m),N(nq))_S,\mathscr{F}_{M(m),N(nq)}^r),
\end{align*}
as the compositions
$$
\Phi_{q\ast}=\varphi_{q\ast}\circ\lambda_{q\ast}\quad\text{and}\quad \Phi_q^\ast=\lambda_q^\ast\circ\varphi_q^\ast.
$$

We define the Hecke operators $T_q$ and the adjoint Hecke operators $T_q'$ acting on the \'{e}tale cohomology groups
$$
H^i_{\et}(Y(M(m),N(nq))_S,\mathscr{F}_{M(m),N(nq)}^r)
$$
as the compositions
$$
T_q=\check{\nu}_{q\ast}\circ\Phi_{q\ast}\circ\nu_q^\ast\quad\text{and}\quad T_q'=\nu_{q\ast}\circ\Phi_q^\ast\circ\check{\nu}_q^\ast.
$$

If we define pullbacks
\begin{align*}
& H^i_{\et}(Y(M(m),N(n))_S,\mathscr{F}_{M(m),N(n)}^r)\xrightarrow{\pi_1^\ast} H^i_{\et}(Y(M(m),N(nq))_S,\mathscr{F}_{M(m),N(nq)}^r), \\
& H^i_{\et}(Y(M(m),N(n))_S,\mathscr{F}_{M(m),N(n)}^r)\xrightarrow{\pi_2^\ast} H^i_{\et}(Y(M(mq),N(n))_S,\mathscr{F}_{M(mq),N(n)}^r),
\end{align*}
and pushforwards
\begin{align*}
& H^i_{\et}(Y(M(m),N(nq))_S,\mathscr{F}_{M(m),N(nq)}^r)\xrightarrow{\pi_{1\ast}} H^i_{\et}(Y(M(m),N(n))_S,\mathscr{F}_{M(m),N(n)}^r), \\
& H^i_{\et}(Y(M(mq),N(n))_S,\mathscr{F}_{M(mq),N(n)}^r)\xrightarrow{\pi_{2\ast}} H^i_{\et}(Y(M(m),N(n))_S,\mathscr{F}_{M(m),N(n)}^r),
\end{align*}
as
$$
\pi_1^\ast=\nu_q^\ast,\quad \pi_2^\ast=\Phi_q^\ast\circ\check{\nu}_q^\ast,\quad \pi_{1\ast}=\nu_{q\ast}\quad\text{and}\quad \pi_{2\ast}=\check{\nu}_{q\ast}\circ\Phi_{q\ast},
$$
then we can write
$$
T_q= \pi_{2\ast}\circ\pi_1^\ast\quad\text{and}\quad T_q'=\pi_{1\ast}\circ\pi_2^\ast.
$$

Now we introduce diamond operators. For $d\in (\bb{Z}/MN\bb{Z})^\times$, these are defined on the curves
$Y(M(m),N(n))$ as the automorphisms $\langle d\rangle$ acting on the moduli space by
$$
(E,P,Q,C,D)\mapsto (E,d^{-1}\cdot P,d\cdot Q, C, D).
$$
We can also define the diamond operator $\langle d \rangle$ on the corresponding universal elliptic curve as the unique automorphism making the diagram
\begin{center}
\begin{tikzpicture}
\matrix(m) [matrix of math nodes, row sep=2.6em, column sep=2.8em, text height=1.5ex, text depth=0.25ex]
{E(M(m),N(n))_S & E(M(m),N(n))_S \\ Y(M(m),N(n))_S & Y(M(m),N(n))_S \\ };
\path[->,font=\scriptsize,>=angle 90]
(m-1-1) edge node [auto] {$\langle d\rangle$} (m-1-2)
(m-2-1) edge node [auto] {$\langle d\rangle$} (m-2-2)
(m-1-1) edge node [auto] {$v$} (m-2-1)
(m-1-2) edge node [auto] {$v$} (m-2-2);
\end{tikzpicture}
\end{center}
cartesian. This in turn induces automorphisms $\langle d\rangle = \langle d\rangle^\ast$ and $\langle d\rangle'=\langle d\rangle_\ast$ on the group $H^i_{\et}(Y(M(m),N(n))_S,\mathscr{F}_{M(m),N(n)}^r)$ which are inverses of each other.

In general, we will be interested in modular curves of the form $Y(1(m),N(n))$. In this case, the natural pairing $\mathscr{L}_r\otimes_{\bb{Z}_p}\mathscr{S}_r\rightarrow \bb{Z}_p$ together with cup-product yields a pairing
$$
H^1_{\et}(Y(1(m),N(n))_S,\mathscr{L}_r(1))\otimes_{\bb{Z}_p} H^1_{\et,c}(Y(1(m),N(n))_S,\mathscr{S}_r)\rightarrow \bb{Z}_p
$$
which becomes perfect after inverting $p$. The operators $T_q$, $T_q'$, $\langle d\rangle$, $\langle d\rangle'$ induce endomorphisms on compactly supported cohomology and
$$
(T_q,T_q'),\quad (T_q',T_q),\quad (\langle d\rangle,\langle d\rangle') \quad\text{and}\quad (\langle d\rangle',\langle d\rangle)
$$
are adjoint pairs under this pairing.

\subsection{Galois representations}

Let $f\in S_k(N_f,\chi_f)$ be a newform of weight $k=r+2\geq 2$, level $N_f$ and character $\chi_f$. Let $p$ be a prime and let $E$ be a finite extension of $\bb{Q}_p$ with ring of integers $\cl{O}$ containing the Fourier coefficients of $f$. By work of Eichler--Shimura and Deligne, there is a two-dimensional representation
$$
\rho_f\,:\,G_\bb{Q}\longrightarrow \GL_2(E)
$$
unramified outside $pN_f$ and characterized by the property that
$$
{\rm trace}\,\rho_f(\Fr_{q})=a_{q}(f)
$$
for all primes $q\nmid pN_f$, where $\Fr_{q}$ denotes an arithmetic Frobenius element at $q$. (In fact, this is the dual of the $p$-adic representation constructed by Deligne.)

It will be convenient for our purposes to work with the following geometric realization of $\rho_f$. Let
\[
H^1_{\et}(Y_{1}(N_f)_{\overline{\bb{Q}}},\mathscr{L}_{r}(1))\otimes_{\Z_p}E\twoheadrightarrow V_f
\]
be the maximal quotient on which $T_q'$ and $\langle d\rangle'$ act as multiplication by $a_q(f)$ and $\chi_f(d)$ for all primes $q\nmid N_f$ and all $d\in(\Z/N_f\Z)^\times$. Then $V_f$ is a two-dimensional $E$-vector space affording the $p$-adic representation $\rho_f$, and we let $T_f\subset V_f$ be the lattice defined by the image of
\begin{align*}
H^1_{\et}(Y_{1}(N_f)_{\overline{\bb{Q}}},\mathscr{L}_{r}(1))\otimes_{\Z_p}\cl{O}
\end{align*}
under the above quotient map.


\section{Hecke algebras and ring class fields}\label{subsec:llz}

In this section we extend the results of \cite[\S{5.2}]{LLZ}, including ring class field extensions of an imaginary quadratic field $K$. The resulting Corollary~\ref{cor:LLZ} will allow us to obtain classes over ring class field extensions of $K$ from diagonal cycles over $\Q$ on triple products of modular curves of varying levels.


Thus let $K$ be an imaginary quadratic field of discriminant $-D<0$, and let $\varepsilon_K$ be the corresponding quadratic character. Let $\psi$ be a Gr\"ossencharacter of $K$ of infinity type $(-1,0)$ and conductor $\mathfrak f$, taking values in a finite extension $L/K$, and let $\chi$ be the unique Dirichlet character modulo $N_{K/\mathbb Q}(\mathfrak f)$ such that $\psi((n))=n \chi(n)$ for  integers $n$ coprime to $N_{K/\mathbb Q}(\mathfrak f)$. Put $N_\psi=N_{K/\mathbb Q}(\mathfrak f)D$, and let $\theta_{\psi} \in S_2(N_{\psi}, \chi\varepsilon_K)$ be the newform attached to $\psi$, i.e.,
\[ \theta_{\psi} = \sum_{(\mathfrak a, \mathfrak f)=1} \psi(\mathfrak a) q^{ N_{K/\mathbb Q}(\mathfrak a)}.
\]

Fix a prime $p\geq 5$ unramified in $K$, a prime $\frk{p}$ of $K$ above $p$ and a prime $\mathfrak{P}$ of $L$ above $\frk{p}$. Let $E=L_{\mathfrak P}$ 
and let $\cl{O}\subset E$ be the ring of integers. Let $\psi_{\mathfrak P}$ be the continuous $E$-valued character of $K^{\times} \backslash \mathbb A_{K,\text{f}}^{\times}$ defined by \[ \psi_{\mathfrak P}(x)=x_{\mathfrak p}^{-1} \psi(x), \] where $x_{\mathfrak p}$ is the projection of the id\`ele $x$ to the component at $\mathfrak p$. We will also denote by $\psi_\frk{P}$ the corresponding character of $G_K$ obtained via the geometric Artin map. Then $\Ind_K^\bb{Q} E(\psi_\frk{P}^{-1})$ is the $p$-adic representation attached to $\theta_\psi$.

\begin{defi}
For an integral ideal $\frk{n}$ of $K$, we denote by $H_\frk{n}$ the maximal $p$-quotient of the corresponding ray class group, and by $K(\frk{n})$ the maximal $p$-extension in the corresponding ray class field. We similarly define $R_n$ and $K[n]$, for each integer $n>0$, as the maximal $p$-quotient in the corresponding ring class group and the maximal $p$-extension in the corresponding ring class field.
\end{defi}

Let $\mathfrak n$ be an integral ideal of $K$ divisible by $\mathfrak f$, and let $N=N_{K/\mathbb Q}(\mathfrak n)D$, which is of course a multiple of $N_{\psi}$. Let $\bb{T}_1'(N)$ be the algebra generated by all the Hecke operators $T_q'$, $\langle d\rangle'$ acting on $H^1(Y_1(N)(\C),\Z)$.

\begin{propo}\label{prop:LLZ}
With the previous definitions and notations, there exists a homomorphism $\phi_{\frk{n}}:\bb{T}_1'(N) \rightarrow\cl{O}[H_\frk{n}]$ defined on generators by
$$
\phi_\frk{n}(T_{q}') = \sum_{\frk{q}}\psi(\frk{q})[\frk{q}]
$$
for every rational prime $q$, where the sum runs over ideals coprime to $\frk{n}$ of norm $q$; and
$$
\phi_\frk{n}(\langle d\rangle')=\chi(d)\varepsilon_K(d)[(d)].
$$
\end{propo}
\begin{proof}
This follows immediately from \cite[Prop. 3.2.1]{LLZ}.
\end{proof}

Now let $\frk{n'}=\frk{nq}$ for some prime ideal $\frk{q}$ above a rational prime $q$. Assume that $\frk{n'}$ is coprime to $p$,
and let $N'=N_{K/\bb{Q}}(\frk{n'})D$. Following \cite[\S3.3]{LLZ}, we define norm maps
$$
\cl{N}_\frk{n}^\frk{n'}:\cl{O}[H_\frk{n'}]\otimes_{(\bb{T}_1'(N')\otimes\Z_p,\phi_\frk{n'})} H^1_{\et}(Y_1(N')_{\overline{\bb{Q}}},\bb{Z}_p(1))\longrightarrow \cl{O}[H_\frk{n}]\otimes_{(\bb{T}_1'(N)\otimes\Z_p,\phi_\frk{n})} H^1_{\et}(Y_1(N)_{\overline{\bb{Q}}},\bb{Z}_p(1))
$$
by the formulae:
\begin{itemize}

\item if $\frk{q}\mid\frk{n}$,
$$
\cl{N}_\frk{n}^\frk{n'}=1\otimes \pi_{1\ast};
$$
\item if $\frk{q}\nmid\frk{n}$ and $\frk{q}$ is ramified or split,
$$
\cl{N}_\frk{n}^\frk{n'}=1\otimes \pi_{1\ast}-\frac{\psi(\frk{q})[\frk{q}]}{q}\otimes \pi_{2\ast};
$$
\item if $\frk{q}\nmid\frk{n}$ and $\frk{q}$ is inert,
$$
\cl{N}_\frk{n}^\frk{n'}=1\otimes \pi_{1\ast}-\frac{\psi(\frk{q})[\frk{q}]}{q^2}\otimes \pi_{2\ast}.
$$

\end{itemize}


More generally, for $\frk{n'}=\frk{n}\frk{r}$ with $\frk{r}$ a product of (not necessarily distinct) prime ideals, we define the map $\cl{N}_\frk{n}^\frk{n'}$ by composing in the natural way the previously defined norm maps.

From now on, we assume that in the case where $(p)=\frk{p}\overline{\frk{p}}$ splits in $K$ the following holds: If $\frk{p}\mid\frk{f}$  then $\overline{\frk{p}}\nmid\frk{f}$ and $\psi\vert_{\cl{O}_{K,\frk{p}}^\times}$ is not congruent to the Teichm\"uller character modulo $\frk{P}$.

\begin{theorem}\label{thm:LLZ}
Let $A$ be the set of prime ideals of $K$ coprime to $\overline{\frk{p}}$ (resp. $p$) if $p$ splits (resp. is inert) in $K$ and divisible by $\frk{f}$. Then there is a family of $G_\bb{Q}$-equivariant isomorphisms of $\cl{O}[H_\frk{n}]$-modules
\begin{center}
\begin{tikzpicture}
\matrix(m) [matrix of math nodes, row sep=2.6em, column sep=2.8em, text height=1.5ex, text depth=0.25ex]
{\nu_\frk{n}:\cl{O}[H_\frk{n}]\otimes_{(\bb{T}_1'(N)\otimes\Z_p,\phi_\frk{n})} H^1_{\et}(Y_1(N)_{\overline{\bb{Q}}},\bb{Z}_p(1)) & \Ind_{K(\frk{n})}^\bb{Q}\cl{O}(\psi_\frk{P}^{-1}), \\};
\path[->,font=\scriptsize,>=angle 90]
(m-1-1) edge node [above] {$\simeq$} (m-1-2);
\end{tikzpicture}
\end{center}
for all $\frk{n}\in A$, such that for $\frk{n}\mid\frk{n'}$ the diagram
\begin{center}
\begin{tikzpicture}
\matrix(m) [matrix of math nodes, row sep=2.6em, column sep=2.8em, text height=1.5ex, text depth=0.25ex]
{\cl{O}[H_\frk{n'}]\otimes_{(\bb{T}_1'(N')\otimes\Z_p,\phi_\frk{n'})} H^1_{\et}(Y_1(N')_{\overline{\bb{Q}}},\bb{Z}_p(1)) & \Ind_{K(\frk{n'})}^\bb{Q}\cl{O}(\psi_\frk{P}^{-1}) \\
\cl{O}[H_\frk{n}]\otimes_{(\bb{T}_1'(N)\otimes\Z_p,\phi_\frk{n})} H^1_{\et}(Y_1(N)_{\overline{\bb{Q}}},\bb{Z}_p(1)) & \Ind_{K(\frk{n})}^\bb{Q}\cl{O}(\psi_\frk{P}^{-1}) \\};
\path[->,font=\scriptsize,>=angle 90]
(m-1-1) edge node [above] {$\nu_\frk{n'}$}
             node [below] {$\simeq$} (m-1-2)
(m-1-1) edge node [left] {$\cl{N}_\frk{n}^\frk{n'}$} (m-2-1)
(m-2-1) edge node [above] {$\nu_\frk{n}$}
             node [below] {$\simeq$} (m-2-2)
(m-1-2) edge node [auto] {} (m-2-2);
\end{tikzpicture}
\end{center}
commutes, where the right vertical arrow is the natural norm map.
\end{theorem}

\begin{proof}
This is \cite[Cor. 5.2.6]{LLZ}.
\end{proof}

\begin{defi}
For any positive integer $n$ with $(n,p\frk{f})=1$, we let $K(\frk{f})[n]$ be the compositum of $K(\frk{f})$ and $K[n]$, and put $R_{\frk{f},n}={\rm Gal}(K(\frk{f})[n]/K)$.
\end{defi}

Let $\bb{T}'(1,N_\psi(n^2))\subset{\rm End}_{\Z}(H^1(Y(1,N_\psi(n^2))(\C),\Z))$ be the subalgebra generated by all Hecke operators $T_q'$ and $\langle d\rangle'$.

\begin{lemma}\label{lemma:hecketorcg}
There exists a homomorphism
\[
\phi_n:\bb{T}'(1,N_\psi(n^2))\longrightarrow\cl{O}[R_{\frk{f},n}]
\]
defined on generators by the same formula as in Proposition~\ref{prop:LLZ}.
\end{lemma}

\begin{proof}
Take the modulus $\frk{n}=\frk{f}(n)$. By Proposition~5.1.2 and Remark~5.1.3 in \cite{LLZ}, the kernel $\cl{I}$ of the composition
\[
\bb{T}_{1}'(N_\psi n^2)\overset{\phi_{\frk{n}}}\longrightarrow\cl{O}[H_{\frk{n}}]\longrightarrow\cl{O}\longrightarrow\cl{O}/\frk{P},
\]
where $\phi_{\frk{n}}$ is as in Proposition~\ref{prop:LLZ}, is a non-Eisenstein maximal ideal of $\bb{T}_{1}'(N_\psi n^2)$ in the sense of [\emph{op.cit.}, Def.~4.1.2]. Therefore, denoting $\cl{I}$-adic completions with the subscript $\cl{I}$,
we have an isomorphism of $\bb{T}_1'(N_\psi n^2)_\cl{I}$-modules
$$
H^1(Y_1(N_\psi n^2)(\bb{C}),\bb{Z})_\cl{I}\cong H^1_c(Y_1(N_\psi n^2)(\bb{C}),\bb{Z})_\cl{I}.
$$

On the other hand, as in the proof of \cite[Lem.~4.2.4]{LLZ}, the natural pullback map yields an isomorphism
$$
H^1_c(Y(1,N_\psi (n^2))(\bb{C}),\bb{Z})\cong H^1_c(Y_1(N_\psi n^2)(\bb{C}),\bb{Z})^\Delta,
$$
where $\Delta$ is the set of diamond operators $\langle d\rangle'$ with $d\equiv 1\pmod{N_\psi}$. Since $\Delta$ maps to 1 under the composition
\[
\bb{T}_{1}'(N_\psi n^2)\overset{\phi_{\frk{n}}}\longrightarrow\cl{O}[H_{\frk{n}}]\longrightarrow\cl{O}[R_{\frk{f},n}],
\]
the result follows.
\end{proof}


\begin{corollary}\label{cor:LLZ}
Let $B$ be the set of positive integers $n$ coprime to $p\frk{f}$. Then there is a family of $G_\bb{Q}$-equivariant isomorphisms of $\cl{O}[R_{\frk{f},n}]$-modules
\begin{center}
\begin{tikzpicture}
\matrix(m) [matrix of math nodes, row sep=2.6em, column sep=2.8em, text height=1.5ex, text depth=0.25ex]
{\nu_{n}:\cl{O}[R_{\frk{f},n}]\otimes_{(\bb{T}'(1,N_\psi(n^2))\otimes\Z_p,\phi_{n})} H^1_{\et}(Y(1,N_\psi(n^2))_{\overline{\bb{Q}}},\bb{Z}_p(1)) & \Ind_{K(\frk{f})[n]}^\bb{Q}\cl{O}(\psi_\frk{P}^{-1}) \\};
\path[->,font=\scriptsize,>=angle 90]
(m-1-1) edge node [above] {$\simeq$} (m-1-2);
\end{tikzpicture}
\end{center}
for all $n\in B$, such that for $n\mid n'$ the diagram
\begin{center}
\begin{tikzpicture}
\matrix(m) [matrix of math nodes, row sep=2.6em, column sep=2.8em, text height=1.5ex, text depth=0.25ex]
{\cl{O}[R_{\frk{f},n'}]\otimes_{(\bb{T}'(1,N_\psi(n'^2))\otimes\Z_p,\phi_{n'})} H^1_{\et}(Y(1,N_\psi(n'^2))_{\overline{\bb{Q}}},\bb{Z}_p(1)) & \Ind_{K(\frk{f})[n']}^\bb{Q}\cl{O}(\psi_\frk{P}^{-1}) \\
\cl{O}[R_{\frk{f},n}]\otimes_{(\bb{T}'(1,N_\psi(n^2))\otimes\Z_p,\phi_{n})} H^1_{\et}(Y(1,N_\psi(n^2))_{\overline{\bb{Q}}},\bb{Z}_p(1)) & \Ind_{K(\frk{f})[n]}^\bb{Q}\cl{O}(\psi_\frk{P}^{-1}) \\};
\path[->,font=\scriptsize,>=angle 90]
(m-1-1) edge node [above] {$\nu_{n'}$}
             node [below] {$\simeq$} (m-1-2)
(m-1-1) edge node [left] {$\cl{N}_{\frk{f},n}^{\frk{f},n'}$} (m-2-1)
(m-2-1) edge node [above] {$\nu_{n}$}
             node [below] {$\simeq$} (m-2-2)
(m-1-2) edge node [auto] {} (m-2-2);
\end{tikzpicture}
\end{center}
commutes, where $\cl{N}_{\frk{f},n}^{\frk{f},n'}$ is induced by  $\cl{N}_{\frk{f}(n')}^{\frk{f}(n)}$ and
the right vertical arrow is the natural norm map. 
\end{corollary}
\begin{proof}
Let $\frk{n}=\frk{f}(n)$, $\cl{I}$ and $\Delta$ be as in the proof of Lemma~\ref{lemma:hecketorcg}.
Since $\cl{I}$ is non-Eisenstein, the natural trace map
\[
H^1_{\et}(Y_1(N_\psi n^2)_{\overline{\bb{Q}}},\bb{Z}_p(1))_{\Delta}\longrightarrow H^1_{\et}(Y(1,N_\psi(n^2))_{\overline{\bb{Q}}},\bb{Z}_p(1))
\]
becomes an isomorphism after taking $\cl{I}$-adic completions. Since the map $\phi_{n}$ of Lemma~\ref{lemma:hecketorcg} is induced by $\phi_{\frk{n}}$ (as shown in the proof of that result), it follows that the $\cl{O}[R_{\frk{f},n}]$-module
$$
\cl{O}[R_{\frk{f},n}]\otimes_{(\bb{T}'(1,N_\psi(n^2))\otimes\Z_p,\phi_{n})} H^1_{\et}(Y(1,N_\psi(n^2))_{\overline{\bb{Q}}},\bb{Z}_p(1))
$$
is naturally isomorphic to
$$
\cl{O}[R_{\frk{f},n}]\otimes_{\cl{O}[H_{\frk{n}}]}\left(\cl{O}[H_{\frk{n}}]\otimes_{(\bb{T}_1'(N_\psi n^2)\otimes\Z_p,\phi_{\frk{n}})} H^1_{\et}(Y_1(N_\psi n^2)_{\overline{\bb{Q}}},\bb{Z}_p(1))\right).
$$
The result now follows from Theorem~\ref{thm:LLZ}.
\end{proof}

\section{Proof of the tame norm relations}\label{subsec:tameeulersystem}

We keep the notations introduced in $\S\ref{subsec:llz}$. Fix two newforms $(g,h)$ of weights $(l,m)$ of the same parity, levels $(N_g,N_h)$, and characters $(\chi_g,\chi_h)$ such that $\chi\varepsilon_K\chi_g\chi_h=1$. Enlarging $L$ if necessary, assume that it contains the Fourier coefficients of $g$ and $h$.

Let $N={\rm lcm}(N_\psi,N_g,N_h)$, and (since $N$ will be fixed throughout) put $Y(m)= Y(1,N(m))$ for every positive integer $m$.

\begin{defi}\label{def:1}
Let $\mathbf{r}=(r_1, r_2, r_3)$ be a triple of non-negative integers such that
\[
r_1+r_2+r_3=2r
\]
with $r\in\Z_{\geq 0}$, and $r_i+r_j\geq r_k$ for every permutation $(i,j,k)$ of $(1,2,3)$. Put
\[
\mathscr{L}_{[\mathbf{r}]}=\mathscr{L}_{1,N(m),r_1}(\Z_p)\otimes_{\Z_p}\mathscr{L}_{1,N(m),r_2}(\Z_p)\otimes_{\Z_p}\mathscr{L}_{1,N(m),r_3}(\Z_p),
\]
and define
\begin{equation*}
\kappa_{m,\mathbf{r}}^{(1)}
\in H^1\left(\Q,H^3_{\et}(Y(m)^3_{\overline{\Q}},\mathscr{L}_{[\mathbf{r}]})\otimes_{\Z_p} \Q_p(2-r)\right)
\end{equation*}
to be the class $\kappa_{N(m),\mathbf{r}}=\mathtt{s}_{\mathbf{r}\ast}\circ \mathtt{HS}\circ d_\ast(\mathtt{Det}_{N(m)}^\mathbf{r})$ constructed as in \cite[\S3]{BSV} for the modular curve $Y(m)$.  \end{defi}

\begin{lemma}\label{lemma:normrelations1}
Let $m$ be a positive integer and let $q$ be a prime number. Assume that both $m$ and $q$ are coprime to $p$ and $N$. Then
\begin{align*}
&(\pi_2,\pi_1,\pi_1)_\ast \kappa_{mq,\mathbf{r}}^{(1)} = (T_q,1,1)\kappa_{m,\mathbf{r}}^{(1)};\quad (\pi_1,\pi_2,\pi_2)_\ast \kappa_{mq,\mathbf{r}}^{(1)} = q^{r-r_1}(T_q',1,1)\kappa_{m,\mathbf{r}}^{(1)}; \\
&(\pi_1,\pi_2,\pi_1)_\ast \kappa_{mq,\mathbf{r}}^{(1)} = (1,T_q,1)\kappa_{m,\mathbf{r}}^{(1)};\quad (\pi_2,\pi_1,\pi_2)_\ast \kappa_{mq,\mathbf{r}}^{(1)} = q^{r-r_2}(1,T_q',1)\kappa_{m,\mathbf{r}}^{(1)}; \\
&(\pi_1,\pi_1,\pi_2)_\ast \kappa_{mq,\mathbf{r}}^{(1)} = (1,1,T_q)\kappa_{m,\mathbf{r}}^{(1)};\quad (\pi_2,\pi_2,\pi_1)_\ast \kappa_{mq,\mathbf{r}}^{(1)} = q^{r-r_3}(1,1,T_q')\kappa_{m,\mathbf{r}}^{(1)}.
\end{align*}
If $q$ is coprime to $m$ we also have
\[
(\pi_1,\pi_1,\pi_1)_\ast \kappa_{mq,\mathbf{r}}^{(1)} = (q+1)\kappa_{m,\mathbf{r}}^{(1)};\quad (\pi_2,\pi_2,\pi_2)_\ast \kappa_{mq,\mathbf{r}}^{(1)} = (q+1)q^{r}\kappa_{m,\mathbf{r}}^{(1)}.
\]
\end{lemma}

\begin{proof}
The same argument proving equations (174) and (176) in \cite{BSV} yields these identities, adding the prime $q$ to the level rather than the prime $p$.
\end{proof}

We next consider the following `asymmetric' diagonal classes.

\begin{defi}
For each squarefree positive integer $n$ coprime to $p$ and $N$, let
\begin{equation*}
\kappa_{n,\mathbf{r}}^{(2)} = n^{r_2}(1,1,\langle n\rangle')(1,\pi_1,\pi_2)_* \kappa_{n^2,\mathbf{r}}^{(1)}\in H^1\left(\Q,H^3_{\et}(Y(n^2)_{\overline{\Q}}\times Y(1)^2_{\overline{\Q}},\mathscr{L}_{[\mathbf{r}]})\otimes_{\Z_p} \Q_p(2-r)\right),
\end{equation*}
where $\pi_1,\pi_2:Y(n^2)\rightarrow Y(1)$ are the degeneracy maps in (\ref{eq:pi12}).
\end{defi}

\begin{lemma}\label{lemma:normrelations2}
Let $n$ be as above and let $q$ be a rational prime coprime to $p$, $N$ and $n$. Then
\begin{align*}
&(\pi_{11},1,1)_\ast \kappa_{nq,\mathbf{r}}^{(2)}=\left\{q^{r_2}(1,1,T_q T_q')-(q+1)q^{r_2+r_3}(1,1,1)\right\}\kappa_{n,\mathbf{r}}^{(2)}, \\
&(\pi_{21},1,1)_\ast \kappa_{nq,\mathbf{r}}^{(2)}=\left\{q^{r}(1,T_q', T_q')-q^{r_2+r_3}(T_q',\langle q\rangle',\langle q\rangle')\right\}\kappa_{n,\mathbf{r}}^{(2)}, \\
&(\pi_{22},1,1)_\ast \kappa_{nq,\mathbf{r}}^{(2)}=\left\{q^{r_1+r_3}(1,T_q'^2, \langle q\rangle')-(q+1)q^{2r}(1,\langle q\rangle',\langle q\rangle')\right\}\kappa_{n,\mathbf{r}}^{(2)},
\end{align*}
where $\pi_{ij}:Y(n^2q^2)\rightarrow Y(n^2)$ denotes the composite map
\begin{equation*}
Y(n^2 q^2)\overset{\pi_i}\longrightarrow Y(n^2q)\overset{\pi_j}\longrightarrow Y(n^2).
\end{equation*}
\end{lemma}

\begin{proof}
To better distinguish between the degeneracy maps $\pi_i$ for different levels, in this proof we use $\varpi_i$ to denote the map $\pi_i$ descending the level by $q$, so that $\varpi_j\circ\varpi_i$ is the degeneracy map $\pi_{ij}$ in the statement of the lemma. Thus we find
\begin{align*}
(\varpi_1,1,1)_\ast \kappa_{nq,\mathbf{r}}^{(2)}
&=n^{r_2}q^{r_2}(1,1,\langle nq\rangle')(1,\pi_1,\pi_2)_\ast (\varpi_1,\varpi_1,\varpi_2)_\ast \kappa_{n^2 q^2,\mathbf{r}}^{(1)} \\
& =n^{r_2}q^{r_2}(1,1,\langle nq\rangle')(1,\pi_1,\pi_2)_\ast (1,1,T_q) \kappa_{n^2 q,\mathbf{r}}^{(1)},
\end{align*}
using Lemma~\ref{lemma:normrelations1} for the second equality; and similarly,
\begin{align*}
(\varpi_2,1,1)_\ast \kappa_{nq,\mathbf{r}}^{(2)}
&=n^{r_2}q^{r_2}(1,1,\langle nq\rangle')(1,\pi_1,\pi_2)_\ast (\varpi_2,\varpi_1,\varpi_2)_\ast \kappa_{n^2 q^2,\mathbf{r}}^{(1)} \\
& =n^{r_2}q^{r}(1,1,\langle nq\rangle')(1,\pi_1,\pi_2)_\ast (1,T_q',1) \kappa_{n^2q,\mathbf{r}}^{(1)}.
\end{align*}
Descending the level again by $q$ this gives
\begin{align*}
(\pi_{11},1,1)_\ast \kappa_{nq,\mathbf{r}}^{(2)}
&=n^{r_2}q^{r_2}(1,1,\langle nq\rangle')(1,\pi_1,\pi_2)_\ast (\varpi_1,\varpi_1,{\varpi_2})_\ast(1,1,{T_q})\kappa_{n^2q,\mathbf{r}}^{(1)} \\
& =n^{r_2}q^{r_2}(1,1,\langle nq\rangle')(1,\pi_1,\pi_2)_\ast (\varpi_{1\ast},\varpi_{1\ast},{T_q\varpi_{2\ast}-q^{r_3}\langle q\rangle\varpi_{1\ast}})\kappa_{n^2q,\mathbf{r}}^{(1)} \\
& =n^{r_2}q^{r_2} (1,1,\langle nq\rangle')(1,\pi_1,\pi_2)_\ast\left\{(1,1,T_q^2)-(q+1)q^{r_3}(1,1,\langle q\rangle)\right\}\kappa_{n^2,\mathbf{r}}^{(1)} \\
& = q^{r_2}\left\{(1,1,T_q T_q')-(q+1)q^{r_3}(1,1,1)\right\}n^{r_2}(1,1,\langle n\rangle')(1,\pi_1,\pi_2)_\ast\kappa_{n^2,\mathbf{r}}^{(1)} \\
& = \left\{q^{r_2}(1,1,T_q T_q')-(q+1)q^{r_2+r_3}(1,1,1)\right\} \kappa_{n,\mathbf{r}}^{(2)},
\end{align*}
and similarly
\begin{align*}
(\pi_{21},1,1)_\ast \kappa_{nq,\mathbf{r}}^{(2)}
&=n^{r_2}q^{r}(1,1,\langle nq\rangle')(1,\pi_1,\pi_2)_\ast (\varpi_1,\varpi_1,\varpi_2)_\ast (1,T_q',1)\kappa_{n^2q,\mathbf{r}}^{(1)} \\
&=n^{r_2}q^{r}(1,1,\langle nq\rangle')(1,\pi_1,\pi_2)_\ast (\varpi_{1\ast},T_q'\varpi_{1\ast}-\langle q\rangle'\varpi_{2\ast},\varpi_{2\ast})\kappa_{n^2q,\mathbf{r}}^{(1)} \\
& = n^{r_2}q^{r}(1,1,\langle nq\rangle')(1,\pi_1,\pi_2)_\ast\left\{(1,T_q',T_q)-q^{r-r_1}(T_q',\langle q\rangle',1)\right\}\kappa_{n^2,\mathbf{r}}^{(1)} \\
& = q^{r}\left\{(1, T_q',T_q')-q^{r-r_1}(T_q',\langle q\rangle',\langle q\rangle')\right\}n^{r_2}(1,1,\langle n\rangle')(1,\pi_1,\pi_2)_\ast\kappa_{n^2,\mathbf{r}}^{(1)} \\
& = \left\{q^r(1,T_q',T_q')-q^{r_2+r_3}(T_q',\langle q\rangle',\langle q\rangle')\right\} \kappa_{n,\mathbf{r}}^{(2)},
\end{align*}
and
\begin{align*}
(\pi_{22},1,1)_\ast \kappa_{nq,\mathbf{r}}^{(2)}
&= n^{r_2}q^{r}(1,1,\langle nq\rangle')(1,\pi_1,\pi_2)_\ast (\varpi_2,\varpi_1,\varpi_2)_\ast(1,T_q',1)\kappa_{n^2q,\mathbf{r}}^{(1)} \\
& =n^{r_2}q^{r}(1,1,\langle nq\rangle')(1,\pi_1,\pi_2)_\ast (\varpi_{2\ast},T_q'\varpi_{1\ast}-\langle q\rangle'\varpi_{2\ast},\varpi_{2\ast})\kappa_{n^2q,\mathbf{r}}^{(1)} \\
& = n^{r_2}q^{r}(1,1,\langle nq\rangle')(1,\pi_1,\pi_2)_\ast\left\{q^{r-r_2}(1,T_q'^2,1)-(q+1)q^{r}(1,\langle q\rangle',1)\right\}\kappa_{n^2,\mathbf{r}}^{(1)} \\
& = q^{r}\left\{q^{r-r_2}(1,T_q'^2,\langle q\rangle')-(q+1)q^{r}(1,\langle q\rangle',\langle q\rangle')\right\}n^{r_2}(1,1,\langle n\rangle')(1,\pi_1,\pi_2)_\ast\kappa_{n^2,\mathbf{r}}^{(1)} \\
& = \left\{q^{r_1+r_3}(1,T_q'^2,\langle q\rangle')-(q+1)q^{2r}(1,\langle q\rangle',\langle q\rangle')\right\} \kappa_{n,\mathbf{r}}^{(2)},
\end{align*}
hence the result.
\end{proof}

Projection of the classes $\kappa_{n,\mathbf{r}}^{(2)}$ to the $(1,1,1)$-component in the K\"unneth decomposition yields classes $\kappa_{n,\mathbf{r}}^{(3)}$ in
\begin{equation*}
H^1\left(\Q,H^1_{\et}(Y(n^2)_{\overline{\Q}},\mathscr{L}_{r_1}(1))\otimes H^1_{\et}(Y(1)_{\overline{\Q}},\mathscr{L}_{r_2}(1))\otimes H^1_{\et}(Y(1)_{\overline{\Q}},\mathscr{L}_{r_3}(1))\otimes_{\Z_p} \Q_p(-1-r)\right).
\end{equation*}

Now set $(r_1, r_2, r_3)=(0,l-2,m-2)$. Fix test vectors
\[
\breve{f}\in S_k(N,\chi\varepsilon_K)[\theta_\psi],\quad  \breve{g}\in S_l(N,\chi_g)[g],\quad \breve{h}\in S_m(N,\chi_h)[h].
\]
These test vectors determine maps
\begin{align*}
H^1_{\et}(Y(n^2)_{\overline{\Q}},\bb{Z}_p(1)) &\rightarrow H^1_{\et}(Y(1,N_\psi(n^2))_{\overline{\Q}},\bb{Z}_p(1)) \\
H^1_{\et}(Y(1)_{\overline{\Q}},\mathscr{L}_{r_2}(1)) &\rightarrow H^1_{\et}(Y_1(N_g)_{\overline{\Q}},\mathscr{L}_{r_2}(1)) \\
H^1_{\et}(Y(1)_{\overline{\Q}},\mathscr{L}_{r_3}(1)) &\rightarrow H^1_{\et}(Y_1(N_h)_{\overline{\Q}},\mathscr{L}_{r_3}(1))
\end{align*}
which we use to project the classes $\kappa_{n,\mathbf{r}}^{(3)}$ to classes $\kappa_{n,\psi gh}^{(3)}$ in
\begin{equation*}
H^1(\Q,\cl{O}[R_{\frk{f},n}]\otimes_{(\bb{T}'(1,N_\psi(n^2))\otimes\bb{Z}_p,\phi_{n})}H^1_{\et}(Y(1,N_\psi(n^2))_{\overline{\Q}},\Z_p(1))\otimes_\cl{O} T_g\otimes_\cl{O} T_h\otimes_{\bb{Z}_p}\bb{Q}_p(-1-r)).
\end{equation*}

Let
\[
T_{g,h}^\psi= T_g\otimes_{\cl{O}}T_h(\psi_\frk{P}^{-1})(-1-r),\quad V_{g,h}^\psi=T_{g,h}^\psi\otimes_{\bb{Z}_p}\bb{Q}_p,
\]
Using the isomorphisms
\begin{center}
\begin{tikzpicture}
\matrix(m) [matrix of math nodes, row sep=2.6em, column sep=2.8em, text height=1.5ex, text depth=0.25ex]
{\nu_{n}:\cl{O}[R_{\frk{f},n}]\otimes_{(\bb{T}'(1,N_\psi(n^2))\otimes\bb{Z}_p,\phi_{n})} H^1_{\et}(Y(1,N_\psi(n^2))_{\overline{\Q}},\bb{Z}_p(1)) & \Ind_{K(\frk{f})[n]}^\bb{Q}\cl{O}(\psi_\frk{P}^{-1}) \\};
\path[->,font=\scriptsize,>=angle 90]
(m-1-1) edge node [above] {$\simeq$} (m-1-2);
\end{tikzpicture}
\end{center}
of Corollary~\ref{cor:LLZ}, and taking the projection of both sides via the quotient map $\cl{O}[R_{\frk{f},n}]\rightarrow \cl{O}[R_n]$, we obtain new isomorphisms
\begin{center}
\begin{tikzpicture}
\matrix(m) [matrix of math nodes, row sep=2.6em, column sep=2.8em, text height=1.5ex, text depth=0.25ex]
{\tilde{\nu}_{n}:\cl{O}[R_{n}]\otimes_{(\bb{T}'(1,N_\psi(n^2))\otimes\bb{Z}_p,\phi_{n})} H^1_{\et}(Y(1,N_\psi(n^2))_{\overline{\Q}},\bb{Z}_p(1)) & \Ind_{K[n]}^\bb{Q}\cl{O}(\psi_\frk{P}^{-1}) ,\\};
\path[->,font=\scriptsize,>=angle 90]
(m-1-1) edge node [above] {$\simeq$} (m-1-2);
\end{tikzpicture}
\end{center}
so that applying the corresponding projection map to the classes $\kappa_{n,\psi gh}^{(3)}$ and using Shapiro's lemma
we obtain classes
\[
\tkapn\in H^1(K[n],V_{g,h}^\psi).
\]

\begin{proposition}\label{prop:normrelations3}
Let $n$ be as above, and let $q$ be a rational prime coprime to $p$, $N$ and $n$.
\begin{enumerate}
\item[{\rm (i)}] If $q$ splits in $K$ as $(q) = \frk{q}\overline{\frk{q}}$, then
\begin{align*}
\cor_{K[nq]/K[n]}(\tkapnq) &=q^{l+m-4}\bigg\{\chi_g(q)\chi_h(q)q\left(\frac{\psi(\frk{q})}{q}\Fr_\frk{q}^{-1}\right)^2-\frac{a_q(g)a_q(h)}{q^{(l+m-4)/2}}\left(\frac{\psi(\frk{q})}{q}\Fr_\frk{q}^{-1}\right) \\
& +\frac{\chi_g(q)^{-1}a_q(g)^2}{q^{l-1}}+\frac{\chi_h(q)^{-1}a_q(h)^2}{q^{m-2}}-\frac{q^2+1}{q} \\
& -\frac{a_q(g)a_q(h)}{q^{(l+m-4)/2}}\left(\frac{\psi(\overline{\frk{q}})}{q}\Fr_{\overline{\frk{q}}}^{-1}\right)+\chi_g(q)\chi_h(q)q\left(\frac{\psi(\overline{\frk{q}})}{q}\Fr_{\overline{\frk{q}}}^{-1}\right)^2\bigg\}\,\tkapn.
\end{align*}
\item[{\rm (ii)}] If $q$ is inert in $K$, then
\begin{align*}
\cor_{K[nq]/K[n]}(\tkapnq) &=q^{l+m-4}\left\{\frac{\chi_g(q)^{-1}a_q(g)^2}{q^{l-1}}+\frac{\chi_h(q)^{-1}a_q(h)^2}{q^{m-2}}-\frac{(q+1)^2}{q}\right\}\tkapn.
\end{align*}
\end{enumerate}
\end{proposition}

\begin{proof}
We have the commutative diagram
\begin{center}
\begin{tikzpicture}
\matrix(m) [matrix of math nodes, row sep=2.6em, column sep=1.5em, text height=1.5ex, text depth=0.25ex]
{H^1(K[nq],V_{g,h}^\psi) & H^1(\Q,\Ind_{K[nq]}^\bb{Q}\cl{O}(\psi_{\frk{P}}^{-1})\otimes_\cl{O} T_g\otimes_{\cl{O}} T_h\otimes_{\bb{Z}_p}\bb{Q}_p(-1-r)) \\
H^1(K[n],V_{g,h}^\psi) & H^1(\Q,\Ind_{K[n]}^\bb{Q}\cl{O}(\psi_{\frk{P}}^{-1})\otimes_\cl{O} T_g\otimes_{\cl{O}} T_h\otimes_{\bb{Z}_p}\bb{Q}_p(-1-r)), \\};
\path[->,font=\scriptsize,>=angle 90]
(m-1-1) edge node [auto] {$\cong$} (m-1-2)
(m-2-1) edge node [auto] {$\cong$} (m-2-2)
(m-1-1) edge node [auto] {$\cor_{K[nq]/K[n]}$} (m-2-1)
(m-1-2) edge node [auto] {} (m-2-2);
\end{tikzpicture}
\end{center}
where the horizontal isomorphisms are given by Shapiro's lemma and the right vertical arrow comes from the natural norm map between induced representations. Using the isomorphisms $\tilde{\nu}_{n}$ above, the vertical arrows in the previous diagram correspond to the map
\begin{center}
\begin{tikzpicture}
\matrix(m) [matrix of math nodes, row sep=2.6em, column sep=2.8em, text height=1.5ex, text depth=0.25ex]
{H^1(\Q,\cl{O}[R_{nq}]\otimes_{\phi_{nq}} H^1_{\et}(Y(1,N_\psi(n^2q^2))_{\overline{\Q}},\bb{Z}_p(1))\otimes_\cl{O} T_g\otimes_{\cl{O}} T_h\otimes_{\bb{Z}_p}\bb{Q}_p(-1-r)) \\
H^1(\Q,\cl{O}[R_{n}]\otimes_{\phi_{n}} H^1_{\et}(Y(1,N_\psi(n^2))_{\overline{\Q}},\bb{Z}_p(1))\otimes_\cl{O} T_g\otimes_{\cl{O}} T_h\otimes_{\bb{Z}_p}\bb{Q}_p(-1-r)). \\};
\path[->,font=\scriptsize,>=angle 90]
(m-1-1) edge node [auto] {$\cl{N}_{\frk{f},n}^{\frk{f},nq}\otimes\id\otimes\id$} (m-2-1);
\end{tikzpicture}
\end{center}
If $q$ splits in $K$, the map $\cl{N}_{\frk{f},n}^{\frk{f},nq}$ is given by
\[
\cl{N}_{\frk{f},n}^{\frk{f},nq}=\pi_{11\ast}-\left(\frac{\psi(\frk{q})[\frk{q}]}{q}+\frac{\psi(\overline{\frk{q}})[\overline{\frk{q}}]}{q}\right)\pi_{21\ast}+\frac{\chi(q)}{q}\pi_{22\ast},
\]
using the notations introduced in Lemma~\ref{lemma:normrelations2} for the degeneracy maps, and from the relations in that lemma we find
\begin{align*}
\cl{N}_{\frk{f},n}^{\frk{f},nq}&(\tkapnq)=\biggl[1\otimes\left\lbrace q^{r_2}(1,1,T_qT_q')-(q+1)q^{r_2+r_3}(1,1,1)\right\rbrace\biggr.\\
&-\left(\frac{\psi(\frk{q})[\frk{q}]}{q}+\frac{\psi(\overline{\frk{q}})[\overline{\frk{q}}]}{q}\right)\otimes\left\{q^r(1,T_q',T_q')-q^{r_2+r_3}(T_q',\langle q\rangle',\langle g\rangle')\right\}\\
&+\frac{\chi(q)}{q}\otimes\left\lbrace q^{r_1+r_3}(1,{T_q'}^{2},\langle q\rangle')-(q+1)q^{2r}(1,\langle q\rangle',\langle q\rangle')\right\rbrace\biggr]\,\tkapn\\
&=\biggl[\chi_h(q)^{-1}a_q(h)^2q^{r_2}+(q+1)q^{r_2+r_3}\biggr.\\
&-\left(\frac{\psi(\frk{q})[\frk{q}]}{q}+\frac{\psi(\overline{\frk{q}})[\overline{\frk{q}}]}{q}\right)\left\lbrace a_q(g)a_q(h)q^r-\chi_g(q)\chi_h(q)q^{r_2+r_3}(\psi(\frk{q})[\frk{q}]+\psi(\overline{\frk{q}})[\overline{\frk{q}}])\right\rbrace\\
&+\biggl.\frac{\chi(q)}{q}\left\lbrace \chi_h(q)a_q(g)^2q^{r_1+r_3}-\chi_g(q)\chi_h(q)(q+1)q^{2r}\right\rbrace\biggr]\,\tkapn.
\end{align*}
This implies the result in this case. When $q$ is inert in $K$, we have
$$
\cl{N}_{\frk{f},n}^{\frk{f},nq}=\pi_{11\ast}-\frac{\chi(q)}{q}\pi_{22\ast},
$$
and the result in this case follows by a very similar computation that we leave to the reader.
\end{proof}

In particular, restricting to positive integers $n$ as above that are divisible only by primes $q$ which split in $K$, Proposition~\ref{prop:normrelations3} yields  the following result. (Note that since in this section we assume $\psi$ has infinity type $(-1,0)$, the balanced condition forces $l=m$.)

\begin{theorem}\label{principal:tame}
Suppose the weights of $g, h$ are $l=m$.
Let $\mathcal{S}$ be the set of squarefree products of primes $q$ which split in $K$ and are coprime to $p$ and $N$. Assume that $H^1(K[n],T_{g,h}^{\psi})$ is torsion-free for every $n\in\mathcal{S}$. There exists a collection of classes
\[
\left\lbrace\kapn\in H^1(K[n],T_{g,h}^{\psi})\;\colon\; n\in\mathcal{S}\right\rbrace
\]
such that whenever $n, nq\in\mathcal{S}$ with $q$ a prime, we have
\begin{equation}\label{eq:norm-split}
{\rm cor}_{K[nq]/K[n]}(\kapnq)=P_{\frk{q}}({\rm Fr}_{\frk{q}}^{-1})\,\kapn,\nonumber
\end{equation}
where $\mathfrak{q}$ is any of the primes of $K$ above $q$, and $P_{\frk{q}}(X)=\det(1-{\rm Fr}_{\frk{q}}^{-1}X\vert (V_{g,h}^\psi)^\vee(1))$.
\end{theorem}

\begin{proof}
We begin by noting that the only possible denominators of the classes $\tkapn$ are divisors of $(l-2)!(m-2)!$ (as follows from \cite[Rmk.~3.3]{BSV}), so after multiplying them by a suitable power of $p$
they all have coefficients in $T_{g,h}^\psi$.

Now given a prime $q\in\mathcal{S}$, we note that for any prime $v$ of $K$ above $q$ we have
\begin{align*}
P_v(X)
&= 1-\frac{a_q(g)a_q(h)}{q^{(l+m-2)/2}}\frac{\psi(v)}{q}X \\
&\quad+ \left(\frac{\chi_g(q)a_q(h)^2}{q^{m-1}}+\frac{\chi_h(q)a_q(g)^2}{q^{l-1}}-2\chi_g(q)\chi_h(q)\right)\frac{\psi(v)^2}{q^2} X^2 \\
&\quad-\frac{\chi_g(q)\chi_h(q)a_q(g)a_q(h)}{q^{(l+m-2)/2}}\frac{\psi(v)^3}{q^3} X^3+\chi_g(q)^2\chi_h(q)^2 \frac{\psi(v)^4}{q^4} X^4.
\end{align*}
Writing $(q)=\frk{q}\overline{\frk{q}}$ and using that $\psi(\frk{q})\psi(\overline{\frk{q}})=\chi(q)q$ and $\chi_g(q)\chi_h(q)\chi(q)=1$, we therefore find the congruences
\begin{align*}
P_\frk{q}(\Fr_\frk{q}^{-1})\chi_g(q)\chi_h(q)\psi(\overline{\frk{q}})^2\Fr_{\overline{\frk{q}}}^{-2}
&\equiv  P_{\overline{\frk{q}}}(\Fr_{\overline{\frk{q}}}^{-1})\chi_g(q)\chi_h(q)\psi(\frk{q})^2\Fr_{\frk{q}}^{-2}\pmod{q-1}\\
&\equiv \chi_g(q)\chi_h(q)\psi(\frk{q})^2\Fr_\frk{q}^{-2}-a_q(g)a_q(h)\psi(\frk{q})\Fr_\frk{q}^{-1}\\
&\quad+\chi_g(q)^{-1}a_q(g)^2+\chi_h(q)^{-1}a_q(h)^2-2\\
&\quad-a_q(g)a_q(h)\psi(\overline{\frk{q}})\Fr_{\overline{\frk{q}}}^{-1}+\chi_g(q)\chi_h(q)\psi(\overline{\frk{q}})^2\Fr_{\overline{\frk{q}}}^{-2}\pmod{q-1}
\end{align*}
as endomorphisms of $H^1(K[n],T_{g,h}^\psi)$. Since these expressions agree modulo $q-1$ with the factor appearing in the norm relation of Proposition~\ref{prop:normrelations3}(i), together with 
\cite[Lem.~9.6.1 and Lem.~9.6.3]{Rub} the result follows.
\end{proof}

\begin{remark}\label{rem:torsionfree}
The condition that $H^1(K[n],T_{g,h}^{\psi})$ is torsion-free for every $n\in\mathcal{S}$ holds, for example, under the assumptions in Lemma~\ref{lemma:big-image} below.
Indeed, since $\mathrm{SL}_2(\bb{Z}_p)\times \mathrm{SL}_2(\bb{Z}_p)$ has no proper normal subgroups of finite $p$-power index, it follows from this lemma that the residual $G_{K[n]}$-representation attached to $T_{g,h}^{\psi}$ is absolutely irreducible for every $n\in\cl{S}$, so that $H^0(K[n],V_{g,h}^\psi/T_{g,h}^\psi)$ is trivial for every $n\in\mathcal{S}$ and the condition follows.
\end{remark}

\begin{remark}
In the inert case, writing $\frk{q}=(q)$ we have
\begin{align*}
& P_\frk{q}(X)=\det(1-\Fr_\frk{q}^{-1}X\vert (T_{g,h}^\psi)^\vee(1)) \\
&= 1-\left(\frac{a_q(g)^2}{q^{l-1}}-2\chi_g(q)\right)\left(\frac{a_h(q)^2}{q^{m-1}}-2\chi_h(q)\right)\frac{\psi(\frk{q})}{q^2}X \\
& +\left(\chi_h(q)^2\left(\frac{a_q(g)^2}{q^{l-1}}-2\chi_g(q)\right)^2 + \chi_g(q)^2\left(\frac{a_q(h)^2}{q^{m-1}}-2\chi_h(q)\right)^2 - 2\chi_g(q)^2 \chi_h(q)^2 \right)\frac{\psi(\frk{q})^2}{q^4} X^2 \\
& -\chi_g(q)^2
\chi_h(q)^2\left(\frac{a_q(g)^2}{q^{l-1}}-2\chi_g(q)\right)\left(\frac{a_q(h)^2}{q^{m-1}}-2\chi_h(q)\right)\frac{\psi(\frk{q})^3}{q^6} X^3+\chi_g(q)^4\chi_h(q)^4\frac{\psi(\frk{q})^4}{q^8} X^4,
\end{align*}
and similarly as in the proof of Theorem~\ref{principal:tame} we find  the congruence
\begin{align*}
P_\frk{q}(\Fr_\frk{q}^{-1}) &\equiv \chi_g(q)^{-2} a_q(g)^4+\chi_h(q)^{-2}a_q(h)^4+2\chi_g(q)^{-1}\chi_h(q)^{-1}a_q(g)^2 a_q(h)^2 q \\
&-4\frac{\chi_g(q)^{-1}a_q(g)^2(q+1)}{q^{l-1}}-4\frac{\chi_h(q)^{-1}a_q(h)^2 (q+1)}{q^{m-1}}+8(q+1)\pmod{q^2-1}
\end{align*}
as endomorphisms of $H^1(K[n],T_{g,h}^\psi)$. Similarly as above, this expression agrees modulo $q^2-1$ with the square of the Euler factor appearing in the norm relation of Proposition~\ref{prop:normrelations3}(ii).
\end{remark}

Now assume that $(p)=\frk{p}\overline{\frk{p}}$ splits in $K$, with $\frk{p}$ the prime of $K$ above $p$ induced by our fixed embedding $\iota_p:\overline{\Q}\hookrightarrow\overline{\Q}_p$, and let $f=\theta_\psi$ be the theta series associated to $\psi$. Assume also that both $g$ and $h$ are ordinary at $p$. Then, for $\phi\in\{f,g,h\}$, the $G_{\bb{Q}_p}$-representation $V_\phi$ admits a filtration
\begin{equation*}
0 \longrightarrow V_{\phi}^+ \longrightarrow V_{\phi} \longrightarrow V_{\phi}^- \longrightarrow 0
\end{equation*}
where $V_\phi^\pm$ is one-dimensional and $V_\phi^-$ is unramified with $\Fr_p$ acting as multiplication by $\alpha_\phi$, the unit root of the Hecke polynomial of $\phi$ at $p$. Letting $V_{fgh}=V_f\otimes V_g\otimes V_h(-1-r)$, we can therefore consider the $G_{\Q_p}$-subrepresentation
\[
\mathscr{F}^2V_{fgh}= (V_f\otimes V_g^+\otimes V_h^+ +V_f^+\otimes V_{g}\otimes V_{h}^+ +V_{f}^+\otimes V_{g}^+\otimes V_h)(-1-r)
\]
and define the \emph{balanced} local condition $H^1_{\bal}(\Q_p,V_{fgh})\subset H^1(\Q_p,V_{fgh})$ to be the image of the natural map $H^1(\Q_p,\mathscr F^2V_{fgh})\rightarrow H^1(\Q_p,V_{fgh})$. 

Setting
\begin{equation}\label{eq:bal-Sh}
\cl{F}_\frk{p}^+(V_{g,h}^\psi)=(V_g^+\otimes V_h+V_g\otimes V_h^+)(\psi_{\frk{P}}^{-1})(-1-r),\quad
\cl{F}_{\overline{\frk{p}}}^+(V_{g,h}^\psi)=(V_g^+\otimes V_h^+)(\psi_{\frk{P}}^{-1})(-1-r),
\end{equation}
one readily checks that under the isomorphism
$H^1(\Q,V_{fgh})\cong H^1(K,\Vrep)$ of Shapiro's lemma, the balanced local condition $H^1_{\bal}(\Q_p,V_{fgh})$ corresponds to the natural image of
\[
\bigoplus_{v\mid p}H^1(K_v,\cl{F}_v^+(V_{g,h}^\psi))\longrightarrow\bigoplus_{v\mid p}H^1(K_v,V_{g,h}^\psi).
\]
This motivates the following definition. Let $L/K$ be a finite extension, and for every finite prime $v$ of $L$ put
\[
H^1_{\rm bal}(L_v,V_{g,h}^\psi)=
\begin{cases}
{\rm im}\bigl(H^1(L_v,\mathcal{F}_v^+(V_{g,h}^\psi))\rightarrow H^1(L_v,V_{g,h}^\psi)\bigr)&\textrm{if $v\mid p$},\\[0.3em]
{\rm ker}\bigl(H^1(L_v,V_{g,h}^\psi)\rightarrow H^1(L_v^{\rm nr},V_{g,h}^\psi)\bigr)&\textrm{if $v\nmid p$,}
\end{cases}
\]
where $L_v^{\rm nr}$ is the maximal unramified extension of $L_v$. 
We then let $H^1_{\rm bal}(L_v,T_{g,h}^\psi)$ be the inverse image of $H^1_{\rm bal}(L_v,V_{g,h}^\psi)$ under the natural map $H^1(L_v,T_{g,h}^\psi)\rightarrow H^1(L_v,V_{g,h}^\psi)$, and let ${\rm Sel}_{\rm bal}(L,T_{g,h}^\psi)\subset H^1(L,T_{g,h}^\psi)$ be the Greenberg Selmer group cut out by these local conditions. (Note that this is a special case of the more general construction discussed in \S\ref{sec:ES}.)




\begin{proposition}\label{prop:in-Sel}
For every $n\in\mathcal{S}$, the class $\kapn$ lies in the group $\Sel_{\bal}(K[n],T_{g,h}^\psi)$.
\end{proposition}

\begin{proof}
Fix $n\in\mathcal{S}$ and $v$ a finite prime of $K[n]$. If $v\nmid p$, then 
it follows from the Weil conjectures that $V_{g,h}^\psi$ is pure of weight $-1$, and hence
\begin{equation}\label{eq:ur=0}
H^1_{\rm ur}(K[n]_v,V_{g,h}^\psi):={\rm ker}\bigl(H^1(K[n]_v,V_{g,h}^\psi)\rightarrow H^1(K[n]_v^{\rm nr},V_{g,h}^\psi)\bigr)=0.
\end{equation}
By \cite[Cor.~1.3.3(i)]{Rub} and local Tate duality (using the fact that the $G_K$-representation $V_{g,h}^\psi$ is conjugate self-dual), it follows that
\[
H^0(K[n]_v,V_{g,h}^\psi)=H^2(K[n]_{\overline{v}},V_{g,h}^\psi)=0.
\]
Repeating the argument with the roles of $v$ and $\overline{v}$ reversed, from (\ref{eq:ur=0}) and \cite[Cor.~1.3.3(ii)]{Rub} we conclude that 
\[
H^1(K[n]_v,V_{g,h}^\psi)=H^1_{\rm ur}(K[n]_v,V_{g,h}^\psi)=0, 
\]
and so the inclusion ${\rm res}_v(\kapn)\in H^1_{\rm bal}(K[n]_v,T_{g,h}^\psi)$ is automatic.

Now suppose $v\mid p$. As noted in \cite[Prop.~3.2]{BSV}, it follows from the results of \cite{NN16} that the classes $\kappa_{m,\mathbf{r}}^{(1)}$ are geometric at $p$, and therefore the class ${\rm res}_v(\kapn)\in H^1(K[n]_v,T_{g,h}^\psi)$ lands in the inverse image of
\[
H^1_{\text{geo}}(K[n]_v,V_{g,h}^\psi)=\ker\bigl(H^1(K[n]_v,V_{g,h}^\psi)\rightarrow H^1(K[n]_v,V_{g,h}^\psi\otimes_{\bb{Q}_p}B_{\text{dR}})\bigr)
\]
under the natural map $H^1(K[n]_v,T_{g,h}^\psi)\rightarrow H^1(K[n]_v,V_{g,h}^\psi)$.
Since 
$H^1_{\text{geo}}(K[n]_v,V_{g,h}^\psi)$ agrees with the Bloch--Kato finite subspace $H^1_{\text{fin}}(K[n]_v,V_{g,h}^\psi)$ (see \cite[Prop.~1.24(2)]{nekovar-ht}), and the latter agrees with $H^1_{\rm bal}(K[n]_v,V_{g,h}^\psi)$ (see Lemma~\ref{lem:BK} below), the result follows.
\end{proof}



\section{Hida families and Galois representations}

In the next section we will prove that the classes $\kapn$ of Theorem~\ref{principal:tame} extend along the anticyclotomic $\mathbb Z_p$-extension of $K$, i.e., they are anticyclotomic universal norms, and explain the construction of $\kapn$ for more general weights. In this section we collect the background results we shall need, closely following the treatment in \cite{BSV}.  

\subsection{Hida families}\label{subsubsec:Hida}

Let $\Lambda =  \mathbb Z_p[[1+p\mathbb Z_p]]$ and let
\[
\cW=\rm{Spf} (\Lambda)
\]
be the weight space. Then, for any extension $E$ of $\bb{Q}_p$, we have $\cW(E)=\Hom_{\text{cont}}(1+p\bb{Z}_p,E^\times)$. Points of the form $\nu_{r,\epsilon}(n)= \epsilon(n)n^r$, where $r$ is a non-negative integer and $\epsilon$ is a finite order character, will be called \emph{arithmetic}. We refer to $k=r+2$ as the \emph{weight} of $\nu_{r,\epsilon}$. Arithmetic points of the form $\nu_r=\nu_{r,1}$ will be called \emph{classical}.

More generally, let $\cl{R}$ be a normal domain finite flat over $\Lambda$ and let $\cW_{\cl{R}}=\rm{Spf}(\cl{R})$. Then, a point $x\in\cW_{\cl{R}}(\overline{\bb{Q}}_p)$ will be called \emph{arithmetic} if it lies above an arithmetic point $\nu_{r,\epsilon}$ of $\cW(\overline{\bb{Q}}_p)$ and \emph{classical} if it lies above a classical point $\nu_r$ of $\cW(\overline{\bb{Q}}_p)$. Again, we refer to $k=r+2$ as the \emph{weight} of $x$.

Let $M$ be a positive integer coprime to $p$. A \emph{Hida family} of tame level $M$ and character $\chi:(\bb{Z}/Mp\bb{Z})^\times\rightarrow \overline{\bb{Q}}_p^\times$ is a formal $q$-expansion
\[
\hf = \sum_{n\geq 1} a_n(\hf) q^n \in \Lambda_{\hf}[[q]],
\]
where $\Lambda_\hf$ is a normal domain finite flat over $\Lambda$, such that, for any arithmetic point $x\in \cW_{\Lambda_\hf}(\overline{\bb{Q}}_p)$ lying over some $\nu_{r,\epsilon}$, the corresponding specialization is a $p$-ordinary eigenform $f_x\in S_k(Mp^s,\chi\epsilon\omega^{-r})$. As above, we have denoted by $k$ the weight of $x$ and we can take $s=\max\lbrace 1, \ord_p(\cond(\epsilon))\rbrace$. We say that a Hida family $\hf$ is \emph{primitive} if the specializations $f_x$ at arithmetic points $x$ are $p$-stabilized newforms. We say that it is \emph{normalized} if $a_1(\hf)=1$.

Let $\hf$ be a normalized primitive Hida family of tame level $M$. For each arithmetic point $x\in \cW_{\Lambda_\hf}(\overline{\bb{Q}}_p)$, let $\hf_x$ denote the specialization of $\hf$ at $x$ and let $f_x$ be the corresponding newform. There exists a locally-free rank-two $\Lambda_\hf$-module $\mathbb{V}_\hf$ equipped with a continuous action of $G_\bb{Q}$ such that, for any arithmetic point $x\in \cW_{\Lambda_\hf}(\overline{\bb{Q}}_p)$, the corresponding specialization $\bb{V}_\hf\otimes_{\Lambda_\hf,x}\overline{\bb{Q}}_p$ recovers the $G_\bb{Q}$-representation $V_{f_x}$ attached to $f_x$. In particular, the representation $\bb{V}_\hf$ is unramified at any prime $q\nmid Mp$ and $\Tr(\Fr_q)=a_q(\hf)$. We refer to $\bb{V}_\hf$ as the \emph{big Galois representation attached to $\hf$}. If for some (equivalently all) arithmetic point $x_0\in \cW_{\Lambda_\hf}(\overline{\bb{Q}}_p)$ the $G_\bb{Q}$-representation $T_{f_{x_0}}$ attached to $f_{x_0}$ is residually irreducible, then $\bb{V}_\hf$ is a free $\Lambda_\hf$-module.

\subsection{Continuous functions and distributions}\label{subsubsec:dist}



Define the semigroups
\[
\Sigma_0(p)=\begin{pmatrix} \bb{Z}_p^\times & \bb{Z}_p \\ p\bb{Z}_p & \bb{Z}_p \end{pmatrix} \quad\text{and}\quad \Sigma_0'(p)=\begin{pmatrix} \bb{Z}_p & \bb{Z}_p \\ p\bb{Z}_p & \bb{Z}_p^\times \end{pmatrix}.
\]
The sets $\mathsf{T}=\bb{Z}_p^\times\times\bb{Z}_p$ and $\mathsf{T}'=p\bb{Z}_p\times\bb{Z}_p^\times$ bear a right action of $\Sigma_0(p)$ and $\Sigma_0'(p)$, respectively.

Let $\nu$ be a character of $\bb{Z}_p^\times$ taking values in a finite extension $E$ of $\bb{Q}_p$. Let $\cl{O}$ be the ring of integers of $E$ and denote by $\frk{m}$ its maximal ideal. Let $\text{Cont}(\bb{Z}_p,\cl{O})$ denote the module of continuous functions on $\bb{Z}_p$ with values in $\cl{O}$. Define $\cl{O}$-modules
\[
\cl{A}_{\nu}=\left\{ f:\mathsf{T}\rightarrow \cl{O}\; \vert\; f(1,z)\in \text{Cont}(\bb{Z}_p,\cl{O}) \text{ and } f(a\cdot t)=\nu(a)\cdot f(t) \text{ for all } a\in\bb{Z}_p^\times,\, t\in \mathsf{T}\right\},
\]
\[
\cl{A}_{\nu}'=\left\{ f:\mathsf{T}'\rightarrow \cl{O}\; \vert\; f(pz,1)\in \text{Cont}(\bb{Z}_p,\cl{O}) \text{ and } f(a\cdot t)=\nu(a)\cdot f(t) \text{ for all } a\in\bb{Z}_p^\times,\, t\in \mathsf{T}'\right\}
\]
equipped with the $\frk{m}$-adic topology, and $\cl{O}$-modules
\[
\cl{D}_{\nu}=\Hom_{\text{cont},\cl{O}}(\cl{A}_{\nu},\cl{O}),\quad\cl{D}_{\nu}'=\Hom_{\text{cont},\cl{O}}(\cl{A}_{\nu}',\cl{O})
\]
equipped with the weak-$\ast$ topology. The right $\Sigma_0^\cdot(p)$-action on $\mathsf{T}^\cdot$ yields naturally a left $\Sigma_0^\cdot(p)$-action on $\cl{A}_{\nu}^\cdot$ and a right $\Sigma_0^\cdot(p)$-action on $\cl{D}_{\nu}^\cdot$.


\subsection{Group cohomology and \'etale cohomology}\label{subsubsec:group-et}

Let $N$ and $m$ be coprime positive integers which are also coprime to $p$, let $Y=Y(1,N(pm))$ and let $\Gamma$ be the corresponding modular group. Let $\mathcal{E}\rightarrow Y$ be the universal elliptic curve over $Y$, and denote by $C_p$ the canonical cyclic $p$-subgroup. Let $\mathscr{T}$ be the relative $p$-adic Tate module of $\mathcal{E}$ over $Y$. Fix a geometric point $\eta:\text{Spec}(\overline{\bb{Q}})\rightarrow Y$, and choose an isomorphism $\mathscr{T}_{\eta}\cong\bb{Z}_p\oplus\bb{Z}_p$ such that the Weil pairing on $\mathscr{T}_{\eta}$ corresponds to the natural determinant map on the right and the reduction modulo $p$ of the element (0,1) generates $C_{p,\eta}$.

Let $\mathcal{G}=\pi_1^{\et}(Y,\eta)$. The action of $\mathcal{G}$ on $\mathscr{T}$ yields an action of $\mathcal{G}$ on $\bb{Z}_p\oplus \bb{Z}_p$, and hence a continuous representation $\rho:\mathcal{G}\rightarrow \text{GL}_2(\bb{Z}_p)$. More precisely, for any $g\in\cl{G}$,
$$
g\cdot(a,b)=(a,b)\rho(g)^{-1}.
$$
In fact, since the action of $\mathcal{G}$ preserves the canonical subgroup, we have a continuous representation $\rho:\mathcal{G}\rightarrow \Gamma_0(p\bb{Z}_p)$, where
$$
\Gamma_0(p\bb{Z}_p)=\left\{\begin{pmatrix} a & b \\ c & d\end{pmatrix}\in \text{GL}_2(\bb{Z}_p): p\mid c\right\}.
$$

The anti-involution of $\text{GL}_2(\bb{Z}_p)$ given by $\gamma\mapsto\gamma^{\iota}=\det(\gamma)\gamma^{-1}$ restricts to $\Gamma_0(p\bb{Z}_p)$ and allows us to think of this group as acting on the right or left as convenient.

Taking the stalk at $\eta$ gives an equivalence of categories between the category $\mathbf{S}_f(Y_{\et})$ of locally constant constructible sheaves with finite stalk of $p$-power order at $\eta$ and the category $\mathbf{M}_f(\mathcal{G})$ of finite $\mathcal{G}$-sets of $p$-power order. For any topological group $G$, define $\mathbf{M}_f(G)$ as we did for $\cl{G}$. Let $\mathbf{M}_{\text{cont}}(G)$ be the category of $G$-modules which are filtered unions $\cup_{i\in I} M_i$ with $M_i\in \mathbf{M}_f(G)$ and let $\mathbf{M}(G)\subset \mathbf{M}_{\text{cont}}(G)^{\bb{N}}$ be the category of inverse systems of objects in $\mathbf{M}_{\text{cont}}(G)$. Define $\mathbf{S}(Y_{\et})$ similarly. Then, there is an equivalence of categories between $\mathbf{M}(\cl{G})$ and $\mathbf{S}(Y_{\et})$. Moreover, the representation $\rho$ defined above yields a functor $\mathbf{M}(\Gamma_0(p\bb{Z}_p))\rightarrow \mathbf{M}(\mathcal{G})$. Regarding this functor, we adopt the following criterion: if an object $\cl{F}\in \mathbf{M}(\Gamma_0(p\bb{Z}_p))$ is given as a left $\Gamma_0(p\bb{Z}_p)$-module, we define the left $\cl{G}$-action via the map $\rho:\cl{G}\rightarrow \Gamma_0(p\bb{Z}_p)$; if it is given as a right $\Gamma_0(p\bb{Z}_p)$-module, we define the left $\cl{G}$-action via the map $g\mapsto \rho(g)^{-1}$.

Given an inverse system of sheaves $\bm{\mathcal{F}}=(\bm{\mathcal{F}}_i)_{i\in \bb{N}}\in \mathbf{S}(Y_{\et})$, we use the notation $H^j_{\et}(Y,\bm{\mathcal{F}})$ for continuous \'etale cohomology as defined by Jannsen, and write $\mathtt{H}^j_{\et}(Y,\bm{\mathcal{F}})=\varprojlim_i H^j_{\et}(Y,\bm{\mathcal{F}}_i)$. There is a natural surjective morphism $H^j_{\et}(Y,\bm{\mathcal{F}})\rightarrow \mathtt{H}^j_{\et}(Y,\bm{\mathcal{F}})$. The compactly supported cohomology groups $H^j_{\et,c}(Y,\bm{\mathcal{F}})$ and $\mathtt{H}^j_{\et,c}(Y,\bm{\mathcal{F}})$ are defined similarly.

There is an isomorphism $\pi_1^{\et}(Y_{\overline{\bb{Q}}},\eta)\cong \hat{\Gamma}$. Thus, if $\cl{F}\in\mathbf{M}_f(\cl{G})$ is a discrete $\cl{G}$-module and $\bm{\cl{F}}$ is the corresponding object in $\mathbf{S}_f(Y_{\et})$, there are natural isomorphisms
\begin{equation}\label{eq:ES}
H^1_{\et}(Y_{\overline{\bb{Q}}},\bm{\mathcal{F}})\cong H^1(\hat{\Gamma},\cl{F})\cong H^1(\Gamma,\cl{F}).
\end{equation}

Let $\cl{F}\in \mathbf{M}_f(\Gamma_0(p\bb{Z}_p))$ be a left $\Gamma_0(p\bb{Z}_p)$-module, and assume that the $\Gamma_0(p\bb{Z}_p)$-action on $\cl{F}$ extends to a left action of $\Sigma_0^\cdot(p)$. Let $S=\Sigma_0^\cdot(p)\cap \text{GL}_2(\bb{Q})$. The pair $(\Gamma,S)$ is then a Hecke pair in the sense of \cite[\S1.1]{AS} and there is a covariant (left) action of the Hecke algebra $D(\Gamma, S)$ on $H^1(\Gamma,\cl{F})$. For each $g\in S$, let $T(g)=\Gamma g\Gamma$. Following \cite[\S1]{GrSt}, we define, for each positive integer $n$, the Hecke operators
$$
T_n=T\left(\begin{pmatrix} 1 &  \\  & n \end{pmatrix}\right), \quad T_n'=T\left(\begin{pmatrix} n &  \\  & 1 \end{pmatrix}\right).
$$
Also, for each positive integer $a$ coprime to $p$, let
$$
[a]_p = T\left(\begin{pmatrix} a &  \\  & a \end{pmatrix}\right),\quad [a]_p'=T\left(\begin{pmatrix} a &  \\  & a \end{pmatrix}\right).
$$
Finally, for each positive integer $a$ coprime to $N$, choose $\beta_a$ (respectively $\beta_a'$) in $\Gamma_0(Npm)$ whose lower right entry is congruent to $a$ (respectively $a^{-1}$) modulo $N$ and let
$$
[a]_N=T(\beta_a),\quad [a]_N'=T(\beta_a').
$$

The isomorphism (\ref{eq:ES}) is compatible with Hecke actions in the following sense. To distinguish between different levels, we shall now write $\tilde{Y}(m)$ and $\tilde{\Gamma}(m)$ for the above $Y$ and $\Gamma$, respectively. Let $s$ be a positive integer. Choose as above a geometric point $\eta:\text{Spec}(\overline{\bb{Q}})\rightarrow \tilde{Y}(m)$ and let $\eta_s:\text{Spec}(\overline{\bb{Q}})\rightarrow \tilde{Y}(ms)$ be a geometric point lying above $\eta$. Let $r=1+\ord_p(s)$ and choose an isomorphism $\mathscr{T}_{\eta_s}\cong\bb{Z}_p\oplus\bb{Z}_p$ such that the Weil pairing on $\mathscr{T}_{\eta_s}$ corresponds to the natural determinant map on the right, and the reduction modulo $p^r$ of the element (0,1) generates the canonical subgroup $C_{p^r,\eta_s}$. Using these choices to define the corresponding isomorphisms between group cohomology and \'etale cohomology, there are commutative diagrams
\begin{center}
\begin{tikzpicture}
\matrix(m) [matrix of math nodes, row sep=2.6em, column sep=2.8em, text height=1.5ex, text depth=0.25ex]
{H^1_{\et}(\tilde{Y}(ms)_{\overline{\bb{Q}}},\bm{\mathcal{F}}) &  H^1_{\et}(\tilde{Y}(m)_{\overline{\bb{Q}}},\bm{\mathcal{F}}) & H^1_{\et}(\tilde{Y}(m)_{\overline{\bb{Q}}},\bm{\mathcal{F}})  &  H^1_{\et}(\tilde{Y}(ms)_{\overline{\bb{Q}}},\bm{\mathcal{F}})\\
 H^1(\tilde{\Gamma}(ms),\cl{F}) & H^1(\tilde{\Gamma}(m),\cl{F}) & H^1(\tilde{\Gamma}(m),\cl{F}) &  H^1(\tilde{\Gamma}(ms),\cl{F}).\\};
\path[->,font=\scriptsize,>=angle 90]
(m-1-1) edge node [auto] {$\pi_{1\ast}$} (m-1-2)
(m-1-3) edge node [auto] {$\pi_1^\ast$} (m-1-4)
(m-2-1) edge node [auto] {$\cor$} (m-2-2)
(m-2-3) edge node [auto] {$\res$} (m-2-4)
(m-1-1) edge node [auto] {$\cong$} (m-2-1)
(m-1-2) edge node [auto] {$\cong$} (m-2-2)
(m-1-3) edge node [auto] {$\cong$} (m-2-3)
(m-1-4) edge node [auto] {$\cong$} (m-2-4);
\end{tikzpicture}
\end{center}
Also, if $\left(\begin{smallmatrix}s &  \\  & 1 \end{smallmatrix}\right)\in \Sigma_0^\cdot(p)$, we have the commutative diagram
\begin{center}
\begin{tikzpicture}
\matrix(m) [matrix of math nodes, row sep=2.6em, column sep=1.0em, text height=1.5ex, text depth=0.25ex]
{H^1_{\et}(\tilde{Y}(ms)_{\overline{\bb{Q}}},\bm{\mathcal{F}}) &  H^1_{\et}(\tilde{Y}(ms)_{\overline{\bb{Q}}},\varphi_s^\ast(\bm{\mathcal{F}})) & H^1_{\et}(Y(1(s),N(pm)_{\overline{\bb{Q}}},\bm{\mathcal{F}})  &  H^1_{\et}(\tilde{Y}(m)_{\overline{\bb{Q}}},\bm{\mathcal{F}})\\
H^1(\tilde{\Gamma}(ms),\cl{F}) & H^1(\tilde{\Gamma}(ms),\varphi_s^\ast(\cl{F})) & H^1(\Gamma(1(s),N(pm)),\cl{F}) &  H^1(\tilde{\Gamma}(m),\cl{F}),\\};
\path[->,font=\scriptsize,>=angle 90]
(m-1-1) edge node [auto] {$\lambda_{s\ast}$} (m-1-2)
(m-1-2) edge node [auto] {$\varphi_{s\ast}$} (m-1-3)
(m-1-3) edge node [auto] {$\check{\nu}_{s\ast}$} (m-1-4)
(m-2-1) edge node [auto] {$\lambda_{s\ast}$} (m-2-2)
(m-2-2) edge node [auto] {$\varphi_{s\ast}$} (m-2-3)
(m-2-3) edge node [auto] {$\cor$} (m-2-4)
(m-1-1) edge node [auto] {$\cong$} (m-2-1)
(m-1-2) edge node [auto] {$\cong$} (m-2-2)
(m-1-3) edge node [auto] {$\cong$} (m-2-3)
(m-1-4) edge node [auto] {$\cong$} (m-2-4);
\end{tikzpicture}
\end{center}
and, if $\left(\begin{smallmatrix} 1 &  \\  & s \end{smallmatrix}\right)\in \Sigma_0^\cdot(p)$, the commutative diagram
\begin{center}
\begin{tikzpicture}
\matrix(m) [matrix of math nodes, row sep=2.6em, column sep=1.2em, text height=1.5ex, text depth=0.25ex]
{H^1_{\et}(\tilde{Y}(m)_{\overline{\bb{Q}}},\bm{\mathcal{F}}) &  H^1_{\et}(Y(1(s),N(pm))_{\overline{\bb{Q}}},\bm{\mathcal{F}}) & H^1_{\et}(\tilde{Y}(ms)_{\overline{\bb{Q}}},\varphi_s^\ast(\bm{\mathcal{F}}))  &  H^1_{\et}(\tilde{Y}(m)_{\overline{\bb{Q}}},\bm{\mathcal{F}})\\
H^1(\tilde{\Gamma}(m),\cl{F}) & H^1(\Gamma(1(s),N(pm)),\cl{F}) & H^1(\tilde{\Gamma}(ms),\varphi_s^\ast(\cl{F})) &  H^1(\tilde{\Gamma}(m),\cl{F}).\\};
\path[->,font=\scriptsize,>=angle 90]
(m-1-1) edge node [auto] {$\check{\nu}_s^\ast$} (m-1-2)
(m-1-2) edge node [auto] {$\varphi_{s}^\ast$} (m-1-3)
(m-1-3) edge node [auto] {$\lambda_s^\ast$} (m-1-4)
(m-2-1) edge node [auto] {$\res$} (m-2-2)
(m-2-2) edge node [auto] {$\varphi_{s}^\ast$} (m-2-3)
(m-2-3) edge node [auto] {$\lambda_s^\ast$} (m-2-4)
(m-1-1) edge node [auto] {$\cong$} (m-2-1)
(m-1-2) edge node [auto] {$\cong$} (m-2-2)
(m-1-3) edge node [auto] {$\cong$} (m-2-3)
(m-1-4) edge node [auto] {$\cong$} (m-2-4);
\end{tikzpicture}
\end{center}
In the bottom lines of the previous two diagrams, $\varphi_s^\ast(\cl{F})$ is $\cl{F}$ with the action of $\Gamma_0(p^r\bb{Z}_p)$ conjugated by $\left(\begin{smallmatrix} s &  \\  & 1\end{smallmatrix}\right)$; the map $\lambda_{s\ast}$ is induced by the map $\cl{F}\rightarrow \varphi_s^\ast(\cl{F})$ defined by $c\mapsto \left(\begin{smallmatrix} s &  \\  & 1\end{smallmatrix}\right)c$; $\varphi_{s\ast}$ is induced by the pair of compatible maps $\Gamma(1(s),N(pm))\rightarrow \tilde{\Gamma}(ms)$ and $\varphi_s^\ast(\cl{F})\rightarrow \cl{F}$ defined by $\gamma\mapsto \left(\begin{smallmatrix} s^{-1} &  \\  & 1\end{smallmatrix}\right) \gamma \left(\begin{smallmatrix} s &  \\  & 1\end{smallmatrix}\right)$ and $c\mapsto c$, respectively; $\lambda_{s}^\ast$ is induced by the map $\varphi_s(\cl{F})\rightarrow \cl{F}$ defined by $c\mapsto\left(\begin{smallmatrix} 1 &  \\  & s\end{smallmatrix}\right)c$, and $\varphi_{s}^\ast$ is induced by the pair of compatible maps $\tilde{\Gamma}(m)\rightarrow \Gamma(1(s),N(pm))$ and $\cl{F}\rightarrow \varphi_s^\ast(\cl{F})$ defined by $\gamma\mapsto \left(\begin{smallmatrix} 1 &  \\  & s^{-1}\end{smallmatrix}\right) \gamma \left(\begin{smallmatrix} 1 &  \\  & s\end{smallmatrix}\right)$ and $c\mapsto c$, respectively.

We shall denote by $\pi_{2\ast}$ and $\pi_2^\ast$, respectively, the composition of the maps in the rows of the previous two diagrams, both in \'etale cohomology and in group cohomology. Similarly, we shall also use $\pi_{1\ast}$ and $\pi_1^\ast$ to denote the corresponding corestriction and restriction maps.

For any rational prime $q$, a simple calculation shows that the following identities hold in group cohomology whenever the maps involved are defined:
$$
T_q=\pi_{1\ast}\circ\pi_2^\ast,\quad T_q'=\pi_{2\ast}\circ\pi_1^\ast.
$$
Therefore, under the isomorphism 
(\ref{eq:ES}), the covariant action of the operators $T_q$, $T_q'$ on \'etale cohomology corresponds to the covariant action of the operators $T_q$, $T_q'$ on group cohomology, whenever defined. Similarly, the covariant action of the operators $\langle d\rangle$, $\langle d\rangle'$ on \'etale cohomology corresponds to the covariant action of the operators $[d]_N$, $[d]_N'$ on group cohomology.

The anti-involution $\iota$ extends to $\text{Mat}_{2\times 2}(\bb{Z}_p)$ in the obvious way and turns a left (respectively right) action of $\Sigma_0(p)$ into a right (respectively left) action of $\Sigma_0'(p)$. Thus, given an object $\cl{F}\in\mathbf{M}(\Gamma_0(p\bb{Z}_p))$ whose right $\Gamma_0(p\bb{Z}_p)$-action extends to a right $\Sigma_0^\cdot(p)$-action, there is an isomorphism $H^1_{\et}(Y_{\overline{\bb{Q}}},\bm{\cl{F}})\cong H^1(\Gamma,\cl{F})$ under which the contravariant action of the operators $T_q$, $T_q'$, $\langle d\rangle$, $\langle d\rangle'$ on \'etale cohomology corresponds to the contravariant action of the operators $T_q$, $T_q'$, $[d]_N$, $[d]_N'$ on group cohomology, whenever defined.

Consider the modules $\cl{A}_{\nu}^\cdot$ and $\cl{D}_{\nu}^\cdot$ defined earlier in this section. The action of $\Gamma_0(p\bb{Z}_p)$ on $\mathsf{T}'$ is transitive and the stabilizer of the element $(0,1)\in\mathsf{T}'$ is the subgroup
\[
P(\bb{Z}_p)=\left\{\begin{pmatrix} a & b \\ 0 & 1 \end{pmatrix}\in \text{GL}_2(\bb{Z}_p)\right\},
\]
so we can identify $\mathsf{T}'$ with $P(\bb{Z}_p)\backslash  \Gamma_0(p\bb{Z}_p)$. Similarly, the action of $\Gamma_0(p\bb{Z}_p)$ on $\mathsf{T}$ is transitive and the stabilizer of the element $(1,0)\in\mathsf{T}$ is the subgroup
\[
P(\bb{Z}_p)^w=\left\{\begin{pmatrix} 1 & 0 \\ pc & d \end{pmatrix}\in \text{GL}_2(\bb{Z}_p)\right\},
\]
so we can identify $\mathsf{T}$ with $P(\bb{Z}_p)^w\backslash  \Gamma_0(p\bb{Z}_p)$. For any positive integer $j$, let
\[
\Gamma_1(p^j\bb{Z}_p)=\left\{\begin{pmatrix}a & b \\ c & d\end{pmatrix}\in\text{GL}_2(\bb{Z}_p): c\equiv 0\text{ (mod }p^j),\, d\equiv 1\text{ (mod }p^j)\right\},
\]
\[
\Gamma_1(p^j\bb{Z}_p)^w=\left\{\begin{pmatrix}a & b \\ pc & d\end{pmatrix}\in\text{GL}_2(\bb{Z}_p): a\equiv 1\text{ (mod }p^j),\, b\equiv 0\text{ (mod }p^{j-1})\right\}.
\]
Then, for any positive integers $i,j$, we can define
\begin{align*}
\cl{A}_{\nu,i,j}'=\big\{ f:\Gamma_1(p^j\bb{Z}_p)\backslash  \Gamma_0(p\bb{Z}_p)\rightarrow \cl{O}/\frk{m}^i\; \vert\;  f(a\cdot \gamma)=\nu(a)\cdot f(\gamma) \\
\text{ for all } a\in\bb{Z}_p^\times,\, \gamma\in \Gamma_1(p^j\bb{Z}_p)\backslash  \Gamma_0(p\bb{Z}_p)\big\},
\end{align*}
\begin{align*}
\cl{A}_{\nu,i,j}=\big\{ f:\Gamma_1(p^j\bb{Z}_p)^w\backslash  \Gamma_0(p\bb{Z}_p)\rightarrow \cl{O}/\frk{m}^i\; \vert\;  f(a\cdot \gamma)=\nu(a)\cdot f(\gamma) \\
\text{ for all } a\in\bb{Z}_p^\times,\, \gamma\in \Gamma_1(p^j\bb{Z}_p)^w\backslash  \Gamma_0(p\bb{Z}_p)\big\}.
\end{align*}
The objects $\cl{A}_{\nu,i,j}^\cdot$ can be regarded as left $\cl{O}[\Sigma^\cdot_0(p)]$-modules. Let $\cl{A}_{\nu,i}^\cdot=\varinjlim_j \cl{A}_{\nu,i,j}^\cdot$. Then $\cl{A}_\nu^\cdot\cong\varprojlim_i \cl{A}_{\nu,i}^\cdot$. We denote by $\bm{\cl{A}}_\nu^\cdot$ the object in $\mathbf{S}(Y_{\et})$ corresponding to $\lbrace \cl{A}_{\nu,i}^\cdot\rbrace_i\in\mathbf{M}(\Gamma_0(p\bb{Z}_p))$. We also define $\cl{D}_{\nu,i}^\cdot=\Hom_{\cl{O}}(\cl{A}_{\nu,i,i}^\cdot,\cl{O}/\frk{m}^i)$. These objects can be regarded as right $\cl{O}[\Sigma^\cdot_0(p)]$-modules and we have $\cl{D}_\nu^\cdot\cong\varprojlim_i \cl{D}_{\nu,i}^\cdot$. We denote by $\bm{\cl{D}}_\nu^\cdot$ the object in $\mathbf{S}(Y_{\et})$ corresponding to $\lbrace \cl{D}_{\nu,i}^\cdot\rbrace_i\in\mathbf{M}(\Gamma_0(p\bb{Z}_p))$. There are natural morphisms of $\cl{O}$-modules
\begin{equation}\label{eq:A}
H^1_{\et}(Y_{\overline{\bb{Q}}},\bm{\cl{A}}_{\nu}^\cdot)\rightarrow \mathtt{H}^1_{\et}(Y_{\overline{\bb{Q}}},\bm{\cl{A}}_{\nu}^\cdot)\cong H^1(\Gamma,\cl{A}_{\nu}^\cdot)\nonumber
\end{equation}
and
\begin{equation}\label{eq:D}
H^1_{\et}(Y_{\overline{\bb{Q}}},\bm{\cl{D}}_{\nu}^\cdot)\cong \mathtt{H}^1_{\et}(Y_{\overline{\bb{Q}}},\bm{\cl{D}}_{\nu}^\cdot)\cong H^1(\Gamma,\cl{D}_{\nu}^\cdot)\nonumber
\end{equation}
compatible with the action of Hecke operators. We also have Hecke-equivariant isomorphisms
\begin{equation}\label{eq:D-c}
H^1_{\et,c}(Y_{\overline{\bb{Q}}},\bm{\cl{D}}_{\nu}^\cdot)\cong \mathtt{H}^1_{\et,c}(Y_{\overline{\bb{Q}}},\bm{\cl{D}}_{\nu}^\cdot)\cong H^1_c(\Gamma,\cl{D}_{\nu}^\cdot),\nonumber
\end{equation}
where $H^j_c(\Gamma,-)=H^{j-1}(\Gamma,\Hom_{\bb{Z}}(\text{Div}^0(\bb{P}^1(\bb{Q})),-))$. These isomorphisms
allow us to define continuous $G_{\bb{Q}}$-actions on the groups $H^1(\Gamma,\cl{A}_{\nu}^\cdot)$, $H^1(\Gamma,\cl{D}_{\nu}^\cdot)$ and $H^1_c(\Gamma,\cl{D}_{\nu}^\cdot)$.

Given a character $\chi:\bb{Z}_p^\times\rightarrow \cl{O}^\times$, let $\cl{O}(\chi)$ be the module $\cl{O}$ with $\Gamma_0(p\bb{Z}_p)$ acting via $\chi\circ\det$, where $\det:\Gamma_0(p\bb{Z}_p)\rightarrow \bb{Z}_p^\times$ is the determinant map.

The natural $\cl{G}$-equivariant evaluation map $\cl{A}_{\nu}^\cdot\otimes_{\cl{O}}\cl{D}_{\nu}^\cdot\rightarrow \cl{O}$ yields a $G_\bb{Q}$-equivariant cup-product pairing
\begin{equation}
H^1(\Gamma,\cl{A}_{\nu}^\cdot)\otimes_\cl{O} H^1_c(\Gamma,\cl{D}_{\nu}^\cdot)\longrightarrow \cl{O}(-1)
\label{eq:AD-pairing}
\end{equation}
under which the Hecke operators $T_q$, $T_q'$, $[d]_N$, $[d]_N'$ acting covariantly on the left, whenever defined, are adjoint to these same operators acting contravariantly on the right.

Let $\det:\mathsf{T}'\times\mathsf{T}\rightarrow \bb{Z}_p^\times$ be the function defined by $\det((x_1,x_2),(y_1,y_2))=x_1y_2-x_2y_1$ and let $\det_\nu$ be the composition of this function with $\nu:\bb{Z}_p^\times\rightarrow \cl{O}$. Evaluation at this function defines a $\cl{G}$-equivariant map $\cl{D}_{\nu}'\otimes_\cl{O} \cl{D}_{\nu}\rightarrow \cl{O}(-\nu)$ which yields a $G_\bb{Q}$-equivariant cup-product pairing
\begin{equation}
H^1(\Gamma,\cl{D}_{\nu}')\otimes_{\cl{O}} H^1_c(\Gamma,\cl{D}_{\nu})\longrightarrow \cl{O}(\bm{\nu})(-1),
\label{eq:DD-pairing}
\end{equation}
where $\bm{\nu}=\nu\circ\epsilon_{\text{cyc}}:G_\bb{Q}\rightarrow \cl{O}^\times$. Under this pairing, the Hecke operators $T_q$, $T_q'$, $[d]_N$, $[d]_N'$ acting contravariantly on the left, whenever defined, are adjoint to the Hecke operators $T_q'$, $T_q$, $[d]_N'$, $[d]_N$ acting contravariantly on the right. We obtain a similar pairing interchanging the roles of $\cl{D}_{\nu}$ and $\cl{D}_{\nu}'$.

\subsection{Ordinary cohomology}\label{subsec:ord-cohomology}

For any $\bb{Z}_p$-algebra $B$, let $S_r(B)$ be the set of two-variable homogeneous polynomials of degree $r$ in $B[x_1,x_2]$. It is a left $B[\Sigma_0^\cdot(p)]$-module with the action of $\Sigma_0^\cdot(p)$ defined by
\[
gP(x_1,x_2)=P((x_1,x_2)\cdot g)
\]
for all $g\in \Sigma_0^\cdot(p)$ and $P(x_1,x_2)\in S_r(B)$. To the $p$-adic $\Gamma_0(p\bb{Z}_p)$-representation $S_r=S_r(\bb{Z}_p)$ there corresponds the locally contant $p$-adic sheaf $\mathscr{S}_{r}$ on $Y_{\et}$ defined in $\S\ref{subsubsec:relT}$. Therefore we have an isomorphism
\[
H^1_{\et}(Y_{\overline{\bb{Q}}},\mathscr{S}_r)\cong H^1(\Gamma,S_r)
\]
which is Hecke-equivariant when we consider the covariant action of Hecke operators on both sides, and we use this isomorphism to define an action of $G_{\bb{Q}}$ on $H^1(\Gamma,S_r)$.

We also define $L_r(B)=\Hom_{B}(S_r(B),B)$, which we regard as a right $B[\Sigma_0^\cdot(p)]$-module defining the $\Sigma_0^\cdot(p)$-action by
\[
(\mu\cdot g)(P(x_1,x_2))=\mu(g P(x_1,x_2))
\]
for all $g\in \Sigma_0^\cdot(p)$, $\mu\in L_r(B)$ and $P(x_1,x_2)\in S_r(B)$. To the $p$-adic $\Gamma_0(p\bb{Z}_p)$-representation $L_r=L_r(\bb{Z}_p)$ there corresponds the locally constant $p$-adic sheaf $\mathscr{L}_{r}$ on $Y_{\et}$. Therefore we have an isomorphism
\[
H^1_{\et}(Y_{\overline{\bb{Q}}},\mathscr{L}_r)\cong H^1(\Gamma,L_r)
\]
which is Hecke-equivariant when we consider the contravariant action of Hecke operators on both sides, and we use this isomorphism to define an action of $G_{\bb{Q}}$ on $H^1(\Gamma,L_r)$.

The natural $\Gamma_0(p\bb{Z}_p)$-equivariant evaluation map $S_r\otimes_{\bb{Z}_p} L_r\rightarrow \bb{Z}_p$ yields a $G_\bb{Q}$-equivariant cup-product pairing
\begin{equation}
H^1(\Gamma,S_r)\otimes_{\bb{Z}_p} H^1_c(\Gamma,L_r)\longrightarrow \bb{Z}_p(-1)
\label{eq:SL-pairing}
\end{equation}
under which the Hecke operators $T_q$, $T_q'$, $[d]_N$, $[d]_N'$ acting covariantly on the left, whenever defined, are adjoint to these same operators acting contravariantly on the right. This pairing becomes perfect after inverting $p$.

Let $\nu_r: \bb{Z}_p^\times \rightarrow \bb{Z}_p^\times$ be the character defined by $\nu_r(z)=z^r$. Evaluation at the polynomial $(x_1y_2-x_2y_1)^r\in S_r\otimes_{\bb{Z}_p}S_r$ defines a $\Gamma_0(p\bb{Z}_p)$-equivariant map $L_r\otimes_{\bb{Z}_p}L_r\rightarrow \bb{Z}_p(-\nu_r)$ and thus yields a $G_\bb{Q}$-equivariant cup-product pairing
\begin{equation}
H^1(\Gamma,L_r)\otimes_{\bb{Z}_p} H^1_c(\Gamma,L_r)\longrightarrow \bb{Z}_p(r-1)
\label{eq:LL-pairing}
\end{equation}
under which the Hecke operators $T_q$, $T_q'$, $[d]_N$, $[d]_N'$ acting contravariantly on the left, whenever defined, are adjoint to the Hecke operators $T_q'$, $T_q$, $[d]_N'$, $[d]_N$ acting contravariantly on the right. This pairing becomes perfect after inverting $p$.

Combining these two pairings we can define a morphism
\[
\mathtt{s}_{r\ast}: H^1(\Gamma,S_r(\bb{Q}_p))\longrightarrow H^1(\Gamma,L_r(\bb{Q}_p))(-r).
\]
This map is $G_\bb{Q}$-equivariant and intertwines the covariant action of the operators $T_q$, $[d]_N$, $[a]_p$ on the source with the contravariant action of the operators $T_q'$, $[d]_N'$, $[a]_p'$ on the target.
We can also define $\mathtt{s}_{r\ast}$ directly via the isomorphism $S_r(\bb{Q}_p)\cong L_r(\bb{Q}_p)(\nu_r)$ arising from the perfect pairing $L_r(\bb{Q}_p)\otimes_{\bb{Q}_p}L_r(\bb{Q}_p)\rightarrow \bb{Q}_p(-\nu_r)$ defined by evaluation at $(x_1y_2-x_2y_1)^r$. Therefore, the denominators introduced by this map are bounded by $r!$, i.e., an element in
\[
\im\left(H^1(\Gamma,S_r)\rightarrow H^1(\Gamma,S_r(\bb{Q}_p))\right)
\]
is mapped to an element in
\[
\frac{1}{r!}\im\left(H^1(\Gamma,L_r)\rightarrow H^1(\Gamma,L_r(\bb{Q}_p))\right),
\]
as follows from \cite[Rmk. 3.3]{BSV}.

To slightly simplify the notation, we will write $\cl{A}_r^\cdot$ and $\cl{D}_r^\cdot$ for $\cl{A}_{\nu_r}^\cdot$ and $\cl{D}_{\nu_r}^\cdot$, respectively. Regarding two-variable polynomials as functions on $\mathsf{T}^\cdot$, we obtain a natural morphism of left $\bb{Z}_p[\Sigma_0^\cdot(p)]$-modules $S_r\rightarrow \cl{A}_{r}^\cdot$. Also, dualizing this map, we obtain a morphism of right $\bb{Z}_p[\Sigma_0^\cdot(p)]$-modules $\cl{D}_{r}^\cdot\rightarrow L_r$. Thus, we have $G_\bb{Q}$-equivariant and Hecke-equivariant morphisms
\[
H^1(\Gamma,S_r)\rightarrow H^1(\Gamma,\cl{A}_{r}^\cdot) \quad\text{and}\quad H^1(\Gamma,\cl{D}_r^\cdot)\rightarrow H^1(\Gamma,L_r).
\]

Applying Hida's (anti-)ordinary projector $e_{\text{ord}}^\cdot:= \lim_{n\to\infty} (T_p^\cdot)^{n!}$, the previous morphisms become isomorphisms
\[
e_{\text{ord}}^\cdot  H^1(\Gamma,S_r)\cong e_{\text{ord}}^\cdot H^1(\Gamma,\cl{A}_r^\cdot),\quad e_{\text{ord}}^\cdot H^1(\Gamma,\cl{D}_r^\cdot)\cong e_{\text{ord}}^\cdot  H^1(\Gamma,L_r).
\]
Under these isomorphisms, the pairings (\ref{eq:SL-pairing}) and (\ref{eq:LL-pairing}) correspond to the pairings (\ref{eq:AD-pairing}) and (\ref{eq:DD-pairing}), respectively, after applying the corresponding (anti-)ordinary projector to every term involved.

\subsection{$\Lambda$-adic Poincar\'e pairing}\label{subsubsec:Lambda-pairing}

It will be convenient to write $\langle a; b\rangle$, with $a\in(\bb{Z}/N\bb{Z})^\times$ and $b\in(\bb{Z}/p^r\bb{Z})^\times$, for the diamond operator $\langle d\rangle$, where $d\in(\bb{Z}/Np^r)^\times$ is congruent to $a$ modulo $N$ and to $b$ modulo $p^r$. We also write $\epsilon_N:G_\bb{Q}\rightarrow (\bb{Z}/N\bb{Z})^\times$ for the mod $N$ cyclotomic character.

For any positive integer $r$, let
$$
G_r=1+p(\bb{Z}/p^r\bb{Z}), \quad \tilde{G}_r=(\bb{Z}/p^r\bb{Z})^\times,
$$
and define
$$
\Lambda_r=\bb{Z}_p[G_r],\quad \tilde{\Lambda}_r=\bb{Z}_p[\tilde{G}_r], \quad \Lambda=\varprojlim_r\Lambda_r=\bb{Z}_p[[1+p\bb{Z}_p]],\quad \tilde{\Lambda}=\varprojlim_r \tilde{\Lambda}_r=\bb{Z}_p[[\bb{Z}_p^\times]].
$$
We have natural factorizations $(\bb{Z}/p^r\bb{Z})^\times=\mu_{p-1}\times(1+p\bb{Z}/p^r\bb{Z})$ and $\bb{Z}_p^\times=\mu_{p-1}\times(1+p\bb{Z}_p)$ which give natural embeddings $\Lambda_r\xhookrightarrow{} \tilde{\Lambda}_r$ and $\Lambda\xhookrightarrow{} \tilde{\Lambda}$. We define idempotents
$$
e_i=\frac{1}{p-1}\sum_{\zeta\in\mu_{p-1}}\zeta^{-i}[\zeta]
$$
for any integer $i$ modulo $p-1$. Let $\kappa_i:\bb{Z}_p^\times\rightarrow \Lambda^\times$ be the character defined by $z\mapsto \omega^i(z)[\langle z\rangle]$ and let $\bm{\kappa}_i=\kappa_i\circ\epsilon_{\text{cyc}}:G_\bb{Q}\rightarrow \Lambda^\times$.

We will shorten notation by writing
\begin{equation}\label{eq:Xinfty}
X_r(m)=X(1,Np^r(m)), \quad H^1_{\et}(X_{\infty}(m)_{\overline{\bb{Q}}},\bb{Z}_p)=\varprojlim_r H^1_{\et}(X_r(m)_{\overline{\bb{Q}}},\bb{Z}_p).
\end{equation}
We have a natural action of $\tilde{\Lambda}_r$ and $\tilde{\Lambda}$ on the previous groups defined by letting group-like elements $[u]$ act like the diamond operators $\langle 1; u\rangle'$.


Fix compatible primitive $p$-power roots of unity $\zeta_{p^r}$ and a primitive $N$-th root of unity $\zeta_N$. Then one can define Atkin--Lehner automorphisms $w_r$ and $w$ for the curve  $X_r(m)$ similarly as in \cite[\S{1.2}]{DR2}. More precisely, $X_r(m)$ parameterizes quadruples $(E,P,Q,C)$, where $E$ is an elliptic curve, $P$ is a point of order $N$, $Q$ is a point of order $p^r$ and $C$ is a cyclic subgroup of $E$ of order $Nm$ containing $P$. Then, we define
$$
w_r(E,P,Q,C)=(E/C_Q,P+C_Q,Q'+C_Q,C+C_Q/C_Q),
$$
where $C_Q\subseteq E$ is the subgroup generated by $Q$ and $Q'\in E[p^r]$ satisfies $\langle Q,Q'\rangle = \zeta_{p^r}$. Similarly, we define
$$
w(E,P,Q,C)=(E/C,P'+C,Q+C,E[Nm]/C),
$$
where $P'\in E[N]$ satisfies $\langle P,P'\rangle = \zeta_N$. These Atkin--Lehner automorphisms satisfy, for any $\sigma\in G_\bb{Q}$,
$$
w_r^\sigma =\langle 1;\epsilon_{\text{cyc}}(\sigma)\rangle w_r,\quad w^\sigma =\langle \epsilon_{N}(\sigma); 1 \rangle w.
$$
We let $w$ and $w_r$ act on cohomology via pullback.

Define $G_\bb{Q}$-equivariant pairings
$$
\langle\,,\rangle_{G_r}: e_i H^1_{\et}(X_{r}(m)_{\overline{\bb{Q}}},\bb{Z}_p)\times e_{-i} H^1_{\et}(X_{r}(m)_{\overline{\bb{Q}}},\bb{Z}_p)\longrightarrow \Lambda_r(-1)
$$
by the formula
$$
\langle a,b\rangle_{G_r}=\sum_{\sigma\in G_r}\langle a^\sigma,b\rangle_r\cdot \sigma^{-1},
$$
where $\langle\,,\rangle_r$ stands for the natural Poincar\'e pairing. These pairings are $\Lambda_r$-linear and anti-linear in the first and second argument, respectively. Then we get $G_\bb{Q}$-equivariant $\Lambda_r$-pairings
$$
[\,,]_{G_r}: e_i H^1_{\et}(X_{r}(m)_{\overline{\bb{Q}}},\bb{Z}_p)\times e_i H^1_{\et}(X_{r}(m)_{\overline{\bb{Q}}},\bb{Z}_p)(\langle\epsilon_N^{-1};1\rangle')\longrightarrow \Lambda_r(\bm{\kappa}_i)(-1)
$$
via the following modification of the previous pairing:
$$
[a,b]_{G_r}=\langle a,ww_r\cdot (T_p')^r\cdot b\rangle_{G_r}.
$$
These pairings are compatible in the sense that the diagram
\begin{center}
\begin{tikzpicture}
\matrix(m) [matrix of math nodes, row sep=2.6em, column sep=2.8em, text height=1.5ex, text depth=0.25ex]
{e_i H^1_{\et}(X_{r+1}(m)_{\overline{\bb{Q}}},\bb{Z}_p)\times e_i H^1_{\et}(X_{r+1}(m)_{\overline{\bb{Q}}},\bb{Z}_p)(\langle\epsilon_N^{-1};1\rangle') & \Lambda_{r+1}(\bm{\kappa}_i)(-1) \\
e_i H^1_{\et}(X_{r}(m)_{\overline{\bb{Q}}},\bb{Z}_p)\times e_i H^1_{\et}(X_{r}(m)_{\overline{\bb{Q}}},\bb{Z}_p)(\langle\epsilon_N^{-1};1\rangle') & \Lambda_r(\bm{\kappa}_i)(-1) \\};
\path[->,font=\scriptsize,>=angle 90]
(m-1-1) edge node [auto] {$[\,,]_{G_{r+1}}$} (m-1-2)
(m-1-1) edge node [auto] {$\pi_{1\ast}\times\pi_{1\ast}$} (m-2-1)
(m-2-1) edge node [auto] {$[\,,]_{G_r}$} (m-2-2)
(m-1-2) edge node [auto] {} (m-2-2);
\end{tikzpicture}
\end{center}
commutes, which can be proved as in \cite[Lem. 1.1]{DR2}. This yields a $\Lambda$-adic perfect $G_\bb{Q}$-equivariant pairing
\begin{equation}
e_i H^1_{\et}(X_{\infty}(m)_{\overline{\bb{Q}}},\bb{Z}_p)^{\text{ord}}\times e_i H^1_{\et}(X_{\infty}(m)_{\overline{\bb{Q}}},\bb{Z}_p)^{\text{ord}}(\langle\epsilon_N^{-1};1\rangle') \longrightarrow  \Lambda(\bm{\kappa}_i)(-1),
\label{eq:lambda-pairing}
\end{equation}
where $H^1_{\et}(X_{\infty}(m)_{\overline{\bb{Q}}},\bb{Z}_p)^{\text{ord}}=e_{\text{ord}}' H^1_{\et}(X_{\infty}(m)_{\overline{\bb{Q}}},\bb{Z}_p)$.  All Hecke operators are self-adjoint for this pairing.

\subsection{Big Galois representations}\label{subsubsec:BGR}


Let $\frk{m}_\Lambda$ be the maximal ideal of $\Lambda$, let $\text{Cont}(\bb{Z}_p,\Lambda)$ be the $\Lambda$-module of continuous functions on $\bb{Z}_p$ with values in $\Lambda$, and let $\univ$ be any of the $\kappa_i$ above. Define the $\Lambda$-module
\begin{align*}
\cl{A}_{\univ}'=\big\{ f:\mathsf{T}'\rightarrow \Lambda \; \vert\; f(pz,1)\in \text{Cont}(\bb{Z}_p,\Lambda) \text{ and } f(a\cdot \gamma)=\univ(a)\cdot f(\gamma) 
\text{ for all } a\in\bb{Z}_p^\times,\, \gamma\in \mathsf{T}'\big\},
\end{align*}
equipped with the $\frk{m}_\Lambda$-adic topology, and the $\Lambda$-module
\[
\cl{D}_{\univ}'=\Hom_{\text{cont},\Lambda}(\cl{A}_{\univ}',\Lambda)
\]
equipped with the weak-$\ast$ topology. As in $\S\ref{subsubsec:dist}$, we can regard $\cl{A}_{\kappa}'$ (respectively $\cl{D}_{\kappa}'$) as a left (respectively right) $\Lambda[\Sigma'_0(p)]$-module.

Similarly to what we did in $\S\ref{subsubsec:group-et}$, define, for any positive integers $j,r$,
\begin{align*}
\cl{A}_{\univ,j,r}'=\big\{ f:\Gamma_1(p^r\bb{Z}_p)\backslash  \Gamma_0(p\bb{Z}_p)\rightarrow \Lambda/\frk{m}_\Lambda^j\; \vert\;  f(a\cdot \gamma)=\univ(a)\cdot f(\gamma) \\
\text{ for all } a\in\bb{Z}_p^\times,\, \gamma\in \Gamma_1(p^r\bb{Z}_p)\backslash  \Gamma_0(p\bb{Z}_p)\big\}
\end{align*}
and $\cl{A}_{\univ,j}'=\varinjlim_r \cl{A}_{\univ,j,r}'$. Then $\cl{A}_\univ'=\varprojlim_j \cl{A}_{\univ,j}'$. We denote by $\bm{\cl{A}}_{\univ}'$ the object in $\mathbf{S}(Y_{\et})$ corresponding to $\lbrace \cl{A}_{\univ,j}'\rbrace_j\in\mathbf{M}(\Gamma_0(p\bb{Z}_p))$. We also define $\cl{D}_{\univ,j}'=\Hom_{\Lambda}(\cl{A}_{\kappa,j,j}',\Lambda/\frk{m}_\Lambda^j)$, so that $\cl{D}_\univ'=\varprojlim_j \cl{D}_{\univ,j}'$, and denote by $\bm{\cl{D}}_\kappa'$ the object in $\mathbf{S}(Y_{\et})$ corresponding to $\lbrace \cl{D}_{\univ,j}'\rbrace_j\in\mathbf{M}(\Gamma_0(p\bb{Z}_p))$. There are natural Hecke-equivariant morphisms of $\Lambda$-modules
\begin{equation*}
H^1_{\et}(Y_{\overline{\bb{Q}}},\bm{\cl{A}}_{\univ}')\rightarrow \mathtt{H}^1_{\et}(Y_{\overline{\bb{Q}}},\bm{\cl{A}}_{\univ}')\cong H^1(\Gamma,\cl{A}_{\univ}'),
\end{equation*}
\begin{equation*}
H^1_{\et}(Y_{\overline{\bb{Q}}},\bm{\cl{D}}_{\univ}')\cong \mathtt{H}^1_{\et}(Y_{\overline{\bb{Q}}},\bm{\cl{D}}_{\univ}')\cong H^1(\Gamma,\cl{D}_{\univ}'),
\end{equation*}
\begin{equation*}
H^1_{\et,c}(Y_{\overline{\bb{Q}}},\bm{\cl{D}}_{\univ}')\cong \mathtt{H}^1_{\et,c}(Y_{\overline{\bb{Q}}},\bm{\cl{D}}_{\univ}')\cong H^1_c(\Gamma,\cl{D}_{\univ}').
\end{equation*}
which allow us to define continuous $G_{\bb{Q}}$-actions on the groups $H^1(\Gamma,\cl{A}_{\univ}')$, $H^1(\Gamma,\cl{D}_{\univ}')$ and $H^1_c(\Gamma,\cl{D}_{\univ}')$.

The evaluation map $\cl{A}_{\univ}'\otimes_{\Lambda}\cl{D}_{\univ}'\rightarrow \Lambda$ yields a $G_\bb{Q}$-equivariant cup-product pairing
\begin{equation}
H^1(\Gamma,\cl{A}_{\univ}')\otimes_{\Lambda} H^1_c(\Gamma,\cl{D}_{\univ}')\longrightarrow \Lambda(-1)
\label{eq:AD-cont-pairing}
\end{equation}
under which the Hecke operators $T_q$, $T_q'$, $[d]_N$, $[d]_N'$ acting covariantly on the left, whenever defined, are adjoint to these same operators acting contravariantly on the right.

Recall that in this section we have set $\Gamma=\Gamma(1,N(pm))$ and let $S=\Sigma_0'(p)\cap \text{GL}_2(\bb{Q})$. For any positive integer $r$, define
\[
\Sigma_1'(p^r)=\begin{pmatrix} \bb{Z}_p & \bb{Z}_p \\ p^r\bb{Z}_p & 1+p^r\bb{Z}_p \end{pmatrix},\quad S_r=\Sigma_1'(p^r)\cap\text{GL}_2(\bb{Q}),\quad \Gamma_r=\Gamma(1,Np^r(m)).
\]

We define compatibility of Hecke pairs as in \cite[Def. 1.1.2]{AS}, but changing left-right conventions. More precisely, we say that the Hecke pair $(\Gamma_\alpha,S_\alpha)$ is compatible to the Hecke pair $(\Gamma_\beta,S_\beta)$ if $(\Gamma_\alpha,S_\alpha)\subseteq (\Gamma_\beta,S_\beta)$, $S_\alpha\Gamma_\beta=S_\beta$ and $\Gamma_\beta\cap S_\alpha^{-1}S_\alpha=\Gamma_\alpha$. With this definition, the Hecke pair $(\Gamma_r, S_r)$ is compatible to the Hecke pair $(\Gamma_t,S_t)$, if $r\geq t$, and to the Hecke pair $(\Gamma,S)$.

Suppose that the Hecke pair $(\Gamma_\alpha,S_\alpha)$ is compatible to $(\Gamma_\beta, S_\beta)$ and that $\Gamma_\alpha$ has finite index in $\Gamma_\beta$. For any right $S_\alpha$-module $E$, we define
\[
\Ind_{\Gamma_{\alpha}}^{\Gamma_\beta} E = \big\{ \varphi :\Gamma_\beta\rightarrow E \; \vert \; \varphi (xy)=\varphi (y)x^{-1} \text{ for all } x\in \Gamma_\alpha,\, y\in \Gamma_\beta\big\}
\]
This module is equipped with a right action of $S_\beta$: given $\varphi\in\Ind_{\Gamma_{\alpha}}^{\Gamma_\beta} E$ and $g\in S_\beta$
\[
(\varphi g)(x)=\sum \varphi (\gamma)\gamma g x^{-1},
\]
where the sum is over representatives $\gamma$ for the cosets in $\Gamma_{\alpha}\backslash \Gamma\cap S_\alpha x g^{-1}$.

Now define
\begin{align*}
A_{\univ,r}'=\big\{ f:\Gamma_1(p^r\bb{Z}_p)\backslash  \Gamma_0(p\bb{Z}_p)\rightarrow \Lambda_{r}\; \vert\;  f(a\cdot \gamma)=\univ(a)\cdot f(\gamma) \\
\text{ for all } a\in\bb{Z}_p^\times,\, \gamma\in \Gamma_1(p^r\bb{Z}_p)\backslash  \Gamma_0(p\bb{Z}_p)\big\}
\end{align*}
and let $D_{\univ,r}'=\Hom_{\Lambda_{r}}(A_{\univ,r}',\Lambda_r)$. With these definitions $\cl{D}_\univ'=\varprojlim_r D_{\kappa, r}'$. Let $S_r$ act trivially on $\bb{Z}_p$ and consider the right $\bb{Z}_p[S_1]$-module $\text{Ind}_{\Gamma_r}^{\Gamma_1}\,\bb{Z}_p$. Let $R$ be a set of representatives for the cosets in $\Gamma_r\backslash \Gamma_1$. The map $\Ind_{\Gamma_r}^{\Gamma_1}\bb{Z}_p\rightarrow D_{\univ,r}'$ defined by
\[
\varphi \mapsto \big[ f\mapsto \sum_{r\in R} \varphi(r) f(r)\big]
\]
is an isomorphism of right $\bb{Z}_p[S_1]$-modules. Therefore, there are natural isomorphisms
\[
H^1(\Gamma_1,\cl{D}_\univ')\cong \varprojlim_r H^1(\Gamma_1,D_{\univ,r}')\cong \varprojlim_r H^1(\Gamma_r,\bb{Z}_p).
\]
According to \cite[Lem. 1.1.3]{AS} and \cite[Lem. 1.1.4]{AS}, both corestriction and the Shapiro isomorphism commute with the action of $D(\Gamma,S)$ via restriction of Hecke algebras, so the previous isomorphisms are Hecke-equivariant.

Similarly to (\ref{eq:Xinfty}), but omitting $m$ from the notation, we let $Y_r=Y(1,Np^r(m))$ and put
\[
H^1_{\et}(Y_{\infty,{\overline{\bb{Q}}}},\bb{Z}_p):=\varprojlim_r H^1_{\et}(Y_{r,{\overline{\bb{Q}}}},\bb{Z}_p),
\]
where the inverse limit is with respect to the maps $\pi_{1\ast}$. Then
\[
H^1(\Gamma_1,\cl{D}_\univ')\cong \varprojlim_r H^1(\Gamma_r,\bb{Z}_p)
\cong H^1_{\et}(Y_{\infty,{\overline{\bb{Q}}}},\bb{Z}_p),
\]
where the last isomorphism is defined by choosing a compatible system of geometric points for the curves $Y_r$ and suitable compatible bases for the corresponding Tate modules. Under the isomorphisms above, the contravariant operators $T_q'$, $[d]_N'$, $[a]_p'$ on the first term correspond to the contravariant operators $T_q'$, $\langle d; 1\rangle'$, $\langle 1; a\rangle'$ defined on the last term via the compatibility of these operators with the pushforward maps $\pi_{1\ast}$.

Also, according to \cite[Lem. 1.1.5]{AS}, the restriction map yields a Hecke-equivariant isomorphism
\[
H^1(\Gamma, \cl{D}_\univ')\cong e_i H^1(\Gamma_1,\cl{D}_\univ')
\]
(recall that we have set $\univ=\univ_i$). Combining this isomorphism with the previous ones, we obtain a Hecke-equivariant isomorphism
\[
H^1(\Gamma, \cl{D}_\univ')\cong e_i H^1_{\et}(Y_{\infty,{\overline{\bb{Q}}}},\bb{Z}_p).
\]



Similarly, using \cite[Prop. 4.2]{AS2}, one proves that there is a Hecke-equivariant isomorphism
\begin{equation}
H^1_c(\Gamma,\cl{D}_\univ') \cong e_i H^1_{\et,c}(Y_{\infty,{\overline{\bb{Q}}}},\bb{Z}_p).
\label{eq:etc-group}
\end{equation}

\section{Proof of the wild norm relations}\label{subsec:Iw}

Assume that $p$ splits in $K$ as $(p)=\frk{p}\overline{\frk{p}}$ and that it does not divide the class number $h_K$.

We keep most of the notations from $\S\ref{subsec:tameeulersystem}$. In particular, $(g,h)$ is a pair of newforms of weights $(l, m)$ of the same parity, levels $(N_g,N_h)$ and characters $(\chi_g,\chi_h)$, and we assume that the ring of integers $\cO\subset E=L_{\frk{P}}$ contains the Fourier coefficients of $g$ and $h$. In addition, we assume that $p$ does not divide $N_g$ nor $N_h$ and that both $g$ and $h$ are ordinary at $p$.

We now allow the Gr\"ossencharacter $\psi$ to have infinity type $(1-k,0)$ for any even integer $k\geq 2$, and let $\frk{f}$ be the conductor of $\psi$, which we assume to be coprime to $p$.  Let $\chi$ be the unique Dirichlet character modulo $N_{K/\mathbb Q}(\mathfrak f)$ such that $\psi((n))=n^{k-1} \chi(n)$ for integers $n$ coprime to $N_{K/\mathbb Q}(\mathfrak f)$.

As in \cite[\S3.2.1]{BL1}, we denote by $\psi_0$ the unique Gr\"ossencharacter of infinity type $(-1,0)$, conductor $\frk{p}$ and whose associated $p$-adic Galois character factors through $\Gamma_\frk{p}$, the Galois group of the unique $\Z_p$-extension of $K$ unramified outside $\frk{p}$. Then we can uniquely write $\psi=\alpha \psi_0^{k-1}$, where $\alpha$ is a 
ray class character of conductor dividing $\frk{f}\frk{p}$.
Since $(\frk{f},p)=1$ and $k$ is even, it easily follows that $\psi$ is \emph{non-Eisenstein} and \emph{$p$-distinguished}, meaning that
\begin{equation}\label{eq:513}
\alpha\psi_0\vert_{\cl{O}_{K,\frk{p}}^\times}\not\equiv\omega\;({\rm mod}\;{\frk{P}}),
\end{equation}
where $\omega$ is the Teichm\"uller character.

Let $\psi_{\mathfrak P}$ be the continuous $E$-valued character of $K^{\times} \backslash \mathbb A_{K,\text{f}}^{\times}$ defined by \[ \psi_{\mathfrak P}(x)=x_{\mathfrak p}^{1-k} \psi(x), \] where $x_{\mathfrak p}$ is the projection of the id\`ele $x$ to the component at $\mathfrak p$. We will also denote by $\psi_\frk{P}$ the corresponding character of $G_K$ obtained via the geometric Artin map. Then $\Ind_K^\bb{Q} E(\psi_\frk{P}^{-1})$ is the $p$-adic representation attached to $\theta_\psi$, and we note that by (\ref{eq:513}) the associated residual representation is absolutely irreducible and $p$-distinguished (see \cite[Rmk.~5.1.4]{LLZ}).

Consider the $q$-expansion
$$
\Theta=\sum_{(\frk{a},\frk{fp})=1}[\frk{a}]q^{N_{K/\bb{Q}}(\frk{a})}\in \cl{O}[[H_{\frk{fp}^\infty}]][[q]],
$$
where $H_{\frk{fp}^\infty}$ denotes the maximal pro-$p$ quotient of the ray class group of $K$ of conductor $\frk{fp}^\infty$, and $[\frk{a}]$ is the image of $\frk{a}$ in $H_{\frk{fp}^\infty}$ under the geometric Artin map. Since we assume that $p$ does not divide $h_K$, we can factor $H_{\frk{fp}^\infty}\cong H_\frk{f}\times \Gamma_\frk{p}$. Hence, we have
%
%
$\Theta\in \cl{O}[H_\frk{f}]\otimes_{\cl{O}}\cl{O}[[\Gamma_\frk{p}]][[q]]$, and we can specialize this to
\begin{equation}\label{eq:CM-F}
\hf = \sum_{(\frk{a},\frk{fp})=1}\alpha([\frk{a}])\psi_0([\frk{a}])[\frk{a}]q^{N_{K/\bb{Q}}(\frk{a})}\in \Lambda_{\hf}[[q]],
\end{equation}
where $\Lambda_{\hf}=\cl{O}[[\Gamma_{\frk{p}}]]$. We identify $\Gamma_{\frk{p}}$ with $\Gamma=1+p\bb{Z}_p$ via the isomorphism $\Gamma\cong\cl{O}_{K,\frk{p}}^{(1)}\rightarrow \Gamma_{\frk{p}}$ defined by $u\mapsto \text{art}_{\frk{p}}(u)^{-1}$, where $\text{art}_{\frk{p}}$ stands for the geometric local Artin map, and in this way we identify $\Lambda_{\hf}$ with $\Lambda_{\cl{O}}=\Lambda\otimes_{\bb{Z}_p}\cl{O}$.
We can therefore regard $\hf$ as a primitive Hida family specializing to
$$
\hf_{k'} = \sum_{(\frk{a},\frk{fp})=1}\alpha([\frk{a}])\psi_0([\frk{a}])^{k'-1}q^{N_{K/\bb{Q}}(\frk{a})}\in S_{k'}^{\text{ord}}(N_\psi p,\chi_\alpha\varepsilon_K\omega^{1-k'})
$$
at the arithmetic point $\nu_{k'-2}$, where $N_\psi=DN_{K/\bb{Q}}(\frk{f})$ and $\chi_\alpha(n)=\alpha((n))$. Note that $\hf$ has character $\chi=\chi_\alpha \omega^{1-k}$ and $\hf_k=\theta_\psi^{(p)}$ is the ordinary $p$-stabilization of $\theta_\psi$.

Let $\chi_{\bb{Q}}$ be the adelic character attached to $\chi$, let $\chi_K=\chi_\bb{Q}\circ N_{K/\bb{Q}}$ and let $\psi^\ast = \chi_K^{-1}\psi$. We can define a primitive Hida family $\hf^\ast$ attached to the Gr\"ossencharacter $\psi^\ast$ in the same way that we defined the Hida family $\hf$ attached to $\psi$. This is just the Hida family $\hf\otimes \chi^{-1}$.


We assume that $\chi\varepsilon_K\chi_g\chi_h=1$, i.e., the product of the characters of $\theta_\psi$, $g$ and $h$ is trivial. Similarly to what we did in \S\ref{subsec:tameeulersystem}, set
$(r_1,r_2,r_3)=(k-2,l-2,m-2)$. For every positive integer $m$, let
\[
\tilde{Y}(m)=Y(1,N(pm)),\quad\textrm{where $N={\rm lcm}(N_\psi, N_g, N_h)$,}
\]
and denote by $\tilde{\Gamma}(m)$ the corresponding modular group.
Let $\kappa=\kappa_{r_1}:\bb{Z}_p^\times\rightarrow \Lambda^\times$ and choose a square root of this character defined by $\kappa^{1/2}(u)=\omega(u)^{(k-2)/2}[\langle u\rangle^{1/2}]$.

We can define classes
\[
\mathbf{Det}_m^{\hf g h}\in H^0_{\et}(\tilde{Y}(m),\bm{\cl{A}}_{\univ}'\otimes\bm{\cl{A}}_{r_2}\otimes\bm{\cl{A}}_{r_3}(-\kappa^{1/2}-\nu_{(r_2+r_3)/2})).
\]
as in \cite[\S8.1]{BSV}, but replacing the Hida families $\hg, \hh$ in their construction by our $g,h$ and working with modules of continuous functions instead of modules of locally analytic functions.
Similarly to what is done in \emph{loc.\,cit.}, and adopting some of the notations there, we define the cohomology classes
\[
\kappa_{m,\hf g h}^{(1)} = (e_{\text{ord}}'\otimes e_{\text{ord}}\otimes e_{\text{ord}})\circ\mathtt{K}\circ\mathtt{HS}\circ d_\ast(\mathbf{Det}_m^{\hf g h}),
\]
inside the group
\[
H^1\left(\mathbb{Q},H^1(\tilde{\Gamma}(m),\cl{A}_\univ')^{\text{ord}}\hat{\otimes} H^1(\tilde{\Gamma}(m),\cl{A}_{r_2})^{\text{ord}}\hat{\otimes} H^1(\tilde{\Gamma}(m),\cl{A}_{r_3})^{\text{ord}}(\bm{\kappa}^{1/2}+2+(r_2+r_3)/2)\right),
\]
where $\bm{\kappa}^{1/2}=\kappa^{1/2}\circ\epsilon_{\text{cyc}}$;
and, for each squarefree positive integer $n$ coprime to $p$ and $N$, we define
\begin{equation*}
\kappa_{n,\hf gh}^{(2)} = \chi\varepsilon_K(n)\kappa(n)^{-1}n^{r_2}(\id\otimes\id\otimes [n]_N)(\id\otimes\pi_{1\ast}\otimes\pi_{2\ast}) \kappa_{n^2,\hf gh}^{(1)}
\end{equation*}
lying in the group
\begin{equation*}
H^1\left(\bb{Q},H^1(\tilde{\Gamma}(n^2),\cl{A}_\univ')^{\text{ord}}\hat{\otimes} H^1(\tilde{\Gamma}(1),\cl{A}_{r_2})^{\text{ord}}\hat{\otimes} H^1(\tilde{\Gamma}(1),\cl{A}_{r_3})^{\text{ord}}(\bm{\kappa}^{1/2}+2+(r_2+r_3)/2)\right).
\end{equation*}

Now we can prove norm relations for $\Lambda$-adic classes as we did for the classes in the previous section.

\begin{lemma}\label{lemma:lambdanormrelations1}
Let $m$ be a positive integer and let $q$ be a prime number. Assume that both $m$ and $q$ are coprime to $p$ and $N$. Then
\begin{align*}
&(\pi_{2\ast}\otimes\pi_{1\ast}\otimes\pi_{1\ast}) \kappa_{mq,\hf gh}^{(1)} = (T_q'\otimes\id\otimes\id)\kappa_{m,\hf gh}^{(1)};\\
&(\pi_{1\ast}\otimes\pi_{2\ast}\otimes\pi_{2\ast}) \kappa_{mq,\hf gh}^{(1)} = \kappa^{-1/2}(q)q^{(r_2+r_3)/2}(T_q\otimes\id\otimes\id)\kappa_{m,\hf gh}^{(1)}; \\
&(\pi_{1\ast}\otimes\pi_{2\ast}\otimes\pi_{1\ast}) \kappa_{mq,\hf gh}^{(1)} = (\id\otimes T_q'\otimes \id)\kappa_{m,\hf gh}^{(1)};\\
&(\pi_{2\ast}\otimes\pi_{1\ast}\otimes\pi_{2\ast}) \kappa_{mq,\hf gh}^{(1)} = \kappa^{1/2}(q)q^{(r_3-r_2)/2}(\id\otimes T_q\otimes\id)\kappa_{m,\hf gh}^{(1)}; \\
&(\pi_{1\ast}\otimes\pi_{1\ast}\otimes\pi_{2\ast}) \kappa_{mq,\hf gh}^{(1)} = (\id\otimes\id\otimes T_q')\kappa_{m,\hf gh}^{(1)};\\
&(\pi_{2\ast}\otimes\pi_{2\ast}\otimes\pi_{1\ast}) \kappa_{mq,\hf gh}^{(1)} = \kappa^{1/2}(q)q^{(r_2-r_3)/2}(\id\otimes\id\otimes T_q)\kappa_{m,\hf gh}^{(1)}.
\end{align*}
If $q$ is coprime to $m$ we also have
\begin{align*}
&(\pi_{1\ast}\otimes\pi_{1\ast}\otimes\pi_{1\ast}) \kappa_{mq,\hf gh}^{(1)} = (q+1)\kappa_{m,\hf gh}^{(1)}; \\
& (\pi_{2\ast}\otimes\pi_{2\ast}\otimes\pi_{2\ast}) \kappa_{mq,\hf gh}^{(1)} = (q+1)\kappa^{1/2}(q)q^{(r_2+r_3)/2}\kappa_{m,\hf gh}^{(1)}.
\end{align*}
\end{lemma}

\begin{proof}
As in Lemma~\ref{lemma:normrelations1}, the same arguments proving equations (174) and (176) in  \cite{BSV} apply \emph{mutatis mutandis} to yield the proof of these identities.
\end{proof}

\begin{lemma}\label{lemma:lambdanormrelations2}
Let $n$ be a squarefree positive integer coprime to $p$ and $N$ and let $q$ be a rational prime coprime to $p$, $N$ and $n$. Then
\begin{align*}
(\pi_{11\ast}\otimes\id\otimes\id) \kappa_{nq,\hf gh}^{(2)} &=\big\{\chi(q)\kappa(q)^{-1}q^{r_2}(\id\otimes \id\otimes [q]_N^{-1}T_q^2) \\
&\quad-\chi(q)\kappa(q)^{-1}(q+1)q^{r_2+r_3}(\id\otimes\id\otimes\id)\big\}\kappa_{n,\hf gh}^{(2)}, \\
(\pi_{21\ast}\otimes\id\otimes\id) \kappa_{nq,\hf gh}^{(2)} &=\big\{\chi(q)\kappa^{-1/2}(q)q^{(r_2+r_3)/2}(\id\otimes T_q\otimes T_q) \\
&\quad -\chi(q)\kappa(q)^{-1}q^{r_2+r_3}(([q]_N')^{-1}T_q'\otimes [q]_N \otimes [q]_N)\big\}\kappa_{n,\hf gh}^{(2)}, \\
(\pi_{22\ast}\otimes\id\otimes\id) \kappa_{nq,\hf gh}^{(2)} &=\big\{\chi(q)q^{r_3}(\id\otimes T_q^2\otimes [q]_N) \\
&\quad -\chi(q)(q+1)q^{r_2+r_3}(\id\otimes [q]_N \otimes [q]_N)\big\}\kappa_{n,\hf gh}^{(2)},
\end{align*}
where $\pi_{ij*}$ denotes the composition
\[
H^1(\tilde{\Gamma}(n^2 q^2),\cl{F}) \overset{\pi_{i*}}\longrightarrow H^1(\tilde{\Gamma}(n^2 q),\cl{F})\overset{\pi_{j*}}
\longrightarrow H^1(\tilde{\Gamma}(n^2),\cl{F}).
\]
\end{lemma}

\begin{proof}
This can be deduced from Lemma~\ref{lemma:lambdanormrelations1} by the same calculation as in Lemma~\ref{lemma:normrelations2}.
\end{proof}

Let $\Gamma(m)=\Gamma(1,Np(m))$ and write $Y(m)$ and $X(m)$ for the corresponding affine and projective modular curves. The pairing in equation~(\ref{eq:AD-cont-pairing}) 
yields a map
\[
H^1(\tilde{\Gamma}(m),\cl{A}_\univ')\rightarrow \Hom_{\Lambda}(H^1_c(\Gamma(m),\cl{D}_\univ'),\Lambda)(-1)\cong \Hom_{\Lambda}(e_{r_1}H^1_{\et,c}(Y_\infty(m)_{\overline{\bb{Q}}},\bb{Z}_p),\Lambda)(-1),
\]
where the isomorphism comes from equation~(\ref{eq:etc-group}). Let $\cl{I}_n$ be the maximal ideal in Hida's big ordinary Hecke algebra $\bb{T}(1,Np^\infty(n^2))'_{\text{ord}}$ corresponding to the Hida family $\hf^\ast$; by (\ref{eq:513}) this ideal corresponds to a non-Eisenstein maximal ideal in $\bb{T}(1,Np(n^2))'$, so there are isomorphisms
\[
H^1_{\et,c}(Y_\infty(n^2)_{\overline{\bb{Q}}},\bb{Z}_p)^{\text{ord}}_{\cl{I}_n}\cong H^1_{\et}(X_\infty(n^2)_{\overline{\bb{Q}}},\bb{Z}_p)^{\text{ord}}_{\cl{I}_n} \cong H^1_{\et}(Y_\infty(n^2)_{\overline{\bb{Q}}},\bb{Z}_p)^{\text{ord}}_{\cl{I}_n}.
\]
Hence, the pairings (\ref{eq:AD-cont-pairing}) and (\ref{eq:lambda-pairing}) together with the isomorphism (\ref{eq:etc-group}) yield a morphism
\[
\mathtt{s}_{\hf,n\ast}:H^1(\tilde{\Gamma}(n^2),\cl{A}_\univ')^{\text{ord}} \longrightarrow e_{r_1}H^1_{\et}(Y_\infty(n^2)_{\overline{\bb{Q}}},\bb{Z}_p)^{\text{ord}}_{\cl{I}_n}(\langle\epsilon_N^{-1};1\rangle')(-\bm{\kappa}).
\]
This map is $G_{\bb{Q}}$-equivariant and intertwines the covariant action of the operators $T_q'$, $[d]_N'$, $[a]_p'$ on the source with the contravariant action of the operators $T_q'$, $\langle d;1\rangle'$, $\langle 1;a\rangle'$ on the target.


Fix a level-$N$ test vector $\breve{\hf}$ for $\hf$ and let $\breve{\hf}^\ast=\breve{\hf}\otimes \chi^{-1}\varepsilon_K$. Fix also test vectors
\[
\breve{g}\in S_l(N,\chi_g)[g],\quad \breve{h}\in S_m(N,\chi_h)[h]
\]
and write $\breve{g}_\alpha$ and $\breve{h}_\alpha$ for the corresponding ordinary $p$-stabilizations.

Define maps
\[
\phi_{n,r}:\bb{T}(1,N_\psi p^r(n^2))_{\text{ord}}'\longrightarrow \cl{O}[R_{\overline{\frk{f}}\frk{p}^r,n}]
\]
attached to the Gr\"ossencharacter $\alpha\chi_K^{-1}\psi_0$ as in Lemma~\ref{lemma:hecketorcg} and let
\[
\phi_{n,\infty}:\bb{T}(1,N_\psi p^\infty(n^2))_{\text{ord}}'\longrightarrow \cl{O}[[R_{\overline{\frk{f}}\frk{p}^\infty,n}]] = \cl{O}[R_{\overline{\frk{f}},n}]\otimes_{\cl{O}} \cl{O}[[\Gamma_{\frk{p}}]].
\]
be the inverse limit $\varprojlim_r\phi_{n,r}$. The test vector $\breve{\hf}^\ast$ determines a degeneracy map
\[
H^1_{\et}(Y_\infty(n^2)_{\overline{\bb{Q}}},\bb{Z}_p(1))^{\text{ord}}\rightarrow H^1_{\et}(Y(1,N_\psi p^\infty (n^2))_{\overline{\bb{Q}}},\bb{Z}_p(1))^{\text{ord}}.
\]
Composing this degeneracy map with the natural quotient map we get a morphism
\[
\pi_{\hf^\ast}:e_{r_1}H^1_{\et}(Y_\infty(n^2)_{\overline{\bb{Q}}},\cl{O}(1))^{\text{ord}}_{\cl{I}_n}\rightarrow (\cl{O}[R_n]\otimes_{\cl{O}} \cl{O}[[\Gamma_\frk{p}]])\otimes_{\phi_{n,\infty}} H^1_{\et}(Y(1,N_\psi p^\infty (n^2))_{\overline{\bb{Q}}},\cl{O}(1))^{\text{ord}}.
\]

The test vectors $\breve{g}_\alpha$ and $\breve{h}_\alpha$ determine degeneracy maps
\begin{align*}
H^1_{\et}(\tilde{Y}(1)_{\overline{\Q}},\mathscr{L}_{r_2}(1)) \xrightarrow{\mu_p^\ast} H^1_{\et}(Y_1(Np)_{\overline{\Q}},\mathscr{L}_{r_2}(1)) &\rightarrow H^1_{\et}(Y_1(N_g)_{\overline{\Q}},\mathscr{L}_{r_2}(1)) \\
H^1_{\et}(\tilde{Y}(1)_{\overline{\Q}},\mathscr{L}_{r_3}(1)) \xrightarrow{\mu_p^\ast} H^1_{\et}(Y_1(Np)_{\overline{\Q}},\mathscr{L}_{r_3}(1)) &\rightarrow H^1_{\et}(Y_1(N_h)_{\overline{\Q}},\mathscr{L}_{r_3}(1)).
\end{align*}
Composing these maps with projection to the $g$-isotypic and $h$-isotypic quotient, respectively, we obtain
\begin{align*}
&\pi_g : e_{\text{ord}}'H^1(\tilde{\Gamma}(1),L_{r_2}(1))\otimes_{\bb{Z}_p}\cl{O}\longrightarrow T_g \\
&\pi_h : e_{\text{ord}}' H^1(\tilde{\Gamma}(1),L_{r_3}(1))\otimes_{\bb{Z}_p}\cl{O}\longrightarrow T_h.
\end{align*}

For the ease of notation, we write
\[
H^1(\psi,\overline{\frk{f}},n)=(\cl{O}[R_{\overline{\frk{f}},n}]\otimes_{\cl{O}} \cl{O}[[\Gamma_\frk{p}]])\otimes_{\phi_{n,\infty}} H^1_{\et}(Y(1,N_\psi p^\infty(n^2))_{\overline{\bb{Q}}},\cl{O})^{\text{ord}}(\langle\epsilon_N^{-1};1\rangle')(\bm{\kappa}^{-1/2})
\]
and put  $H^1(\psi,n)=\cl{O}[R_n]\otimes_{\cl{O}[R_{\overline{\frk{f}},n}]}H^1(\psi,\overline{\frk{f}},n)$. Then we define the class
\begin{equation}\label{eq:lambda-3}
\kappa_{n,\hf gh}^{(3)}=(\pi_{\hf^\ast}\otimes\pi_g\otimes\pi_h)\circ(\mathtt{s}_{\hf\ast}\otimes\mathtt{s}_{r_2\ast}\otimes\mathtt{s}_{r_3\ast})\kappa_{n,\hf gh}^{(2)}\\
\end{equation}
lying in the group
\[
H^1\left(\bb{Q},H^1(\psi,n)\hat{\otimes}_\cl{O} (T_g \otimes_{\cl{O}} T_h)\otimes_{\bb{Z}_p}\bb{Q}_p(-1-(r_2+r_3)/2)\right).
\]

Let $\Gamma_{\text{ac}}$ be the Galois group of the anticyclotomic $\bb{Z}_p$-extension of $K$. We can identify this group with the anti-diagonal in $(1+p\bb{Z}_p)\times (1+p\bb{Z}_p)\cong \cl{O}_{K,\frk{p}}^{(1)}\times \cl{O}_{K,\overline{\frk{p}}}^{(1)}$ via the geometrically normalized Artin map. Let $\kappa_{\text{ac}}:\Gamma_{\text{ac}}\rightarrow \bb{Z}_p^\times$ be the character defined by mapping $((1+p)^{-1/2},(1+p)^{1/2})$ to $1+p$ and let $\bm{\kappa}_{\text{ac}}:\Gamma_{\text{ac}}\rightarrow \Lambda^\times$ be the character defined by mapping $((1+p)^{-1/2}, (1+p)^{1/2})$ to the group-like element $[1+p]$. We use the same notation for the corresponding characters of $G_\bb{Q}$. 
There is a $G_\bb{Q}$-equivariant isomorphism of $\Lambda_\cl{O}[R_{n}]$-modules
\begin{equation}\label{eq:ind-infty}
H^1(\psi,n)\cong \Ind_{K[n]}^\bb{Q} \Lambda_{\cl{O}}(\psi_{\frk{P}}^{-1}\kappa_{\text{ac}}^{r_1/2}\bm{\kappa}_{\text{ac}}^{-1})(-r_1/2).
\end{equation}

Let
\[
T_{g,h}^\psi= T_g\otimes_{\cl{O}}T_h(\psi_\frk{P}^{-1})(-1-r),\quad V_{g,h}^\psi=T_{g,h}^\psi\otimes_{\bb{Z}_p}\bb{Q}_p.
\]
In light of the isomorphism (\ref{eq:ind-infty}), using Shapiro's lemma
the classes $\kappa_{n,\hf gh}^{(3)}$ yield classes
\begin{equation}\label{eq:kapinfty}
\tilde{\kappa}_{\psi,g,h,n,\infty}\in H^1(K[n],\Lambda_{\cl{O}}(\bm{\kappa}_{\text{ac}}^{-1})\hat{\otimes}_{\cl{O}} T_{g,h}^\psi(\kappa_{\text{ac}}^{r_1/2}))\otimes_{\cl{O}}E
\end{equation}
for every squarefree integer $n$ coprime to $p$ and $N$.

\begin{proposition}\label{prop:lambdanormrelations3}
Let $n$ be as above and let $q$ be a rational prime coprime to $p$, $N$ and $n$. Then:
\begin{enumerate}
\item[{\rm (i)}] If $q$ splits in $K$ as $(q) = \frk{q}\overline{\frk{q}}$,
\begin{align*}
\cor_{K[nq]/K[n]}(\tilde{\kappa}_{\psi,g,h,nq,\infty})&=q^{l+m-4}\Bigg\{\chi_g(q)\chi_h(q)q\left(\frac{\kappa_{\ac}^{-(k-2)/2}\psi_{\frk{P}}(\Fr_\frk{q}^{-1})}{q^{k-1}}\Fr_\frk{q}^{-1}\right)^2 \\
&-\frac{a_q(g)a_q(h)}{q^{(l+m-4)/2}}\left(\frac{\kappa_{\ac}^{-(k-2)/2}\psi_{\frk{P}}(\Fr_\frk{q}^{-1})}{q^{k-1}}\Fr_\frk{q}^{-1}\right) \\
& +\frac{\chi_g(q)^{-1}a_q(g)^2}{q^{l-2}}+\frac{\chi_h(q)^{-1}a_q(h)^2}{q^{m-1}}-\frac{q^2+1}{q} \\
& -\frac{a_q(g)a_q(h)}{q^{(l+m-4)/2}}\left(\frac{\kappa_{\ac}^{-(k-2)/2}\psi_{\frk{P}}(\Fr_{\overline{\frk{q}}}^{-1})}{q^{k-1}}\Fr_{\overline{\frk{q}}}^{-1}\right) \\
&+\chi_g(q)\chi_h(q)q\left(\frac{\kappa_{\ac}^{-(k-2)/2}\psi_{\frk{P}}(\Fr_{\overline{\frk{q}}}^{-1})}{q^{k-1}}\Fr_{\overline{\frk{q}}}^{-1}\right)^2\Bigg\}\tilde{\kappa}_{\psi,g,h,n,\infty}.
\end{align*}
\item[{\rm (ii)}] If $q$ is inert in $K$,
\begin{align*}
\cor_{K[nq]/K[n]}(\tilde{\kappa}_{\psi,g,h,nq,\infty}) &= q^{l+m-4}\Bigg\{\frac{\chi_g(q)^{-1}a_q(g)^2}{q^{l-2}}+\frac{\chi_h(q)^{-1}a_q(h)^2}{q^{m-1}}-\frac{(q+1)^2}{q}\Bigg\}\tilde{\kappa}_{\psi,g,h,n,\infty}.
\end{align*}
\end{enumerate}

\end{proposition}
\begin{proof}
The proof of this proposition is similar to the proof of Proposition~\ref{prop:normrelations3}. We just remark that the maps $\mathtt{s}_{\hf,n\ast}$ interchange the degeneracy maps $\pi_1$ and $\pi_2$, and under the isomorphism
\begin{align*}
H^1(K[n],\Lambda_{\cl{O}}(\bm{\kappa}_{\text{ac}}^{-1})&\hat{\otimes}_{\cl{O}} T_{g,h}^\psi(\kappa_{\ac}^{(k-2)/2}))\otimes_{\cl{O}}E\\
&\cong H^1(\bb{Q},H^1(\psi,n)\hat{\otimes}_\cl{O} (T_g \otimes_{\cl{O}} T_h)\otimes_{\bb{Z}_p}\bb{Q}_p(-1-(r_2+r_3)/2))
\end{align*}
arising from (\ref{eq:ind-infty}), the corestriction ${\rm cor}_{K[nq]/K[n]}$ corresponds, in the case where $(q)=\frk{q}\overline{\frk{q}}$ splits in $K$, to the map
\begin{align*}
\cl{N}_{n}^{nq} &=\pi_{11\ast}-\chi^{-1}(q)\omega^{(k-2)/2}(q)\left(\frac{\kappa_{\ac}^{-(k-2)/2}\psi_{\mathfrak{P}}({\rm Fr}_{\frk{q}}^{-1})[\frk{q}]}{q^{k/2}}+\frac{\kappa_{\ac}^{-(k-2)/2}\psi_{\mathfrak{P}}({\rm Fr}_{\overline{\frk{q}}}^{-1})[\overline{\frk{q}}]}{q^{k/2}}\right)\pi_{21\ast} \\
&+\frac{\chi^{-1}(q)\omega^{k-2}(q)}{q}\pi_{22\ast},
\end{align*}
and similarly in the case where $q$ is inert in $K$. Since the result can be deduced from Lemma~\ref{lemma:lambdanormrelations1} by virtually the same calculation as in the proof of Lemma~\ref{lemma:normrelations2}, we omit the details.
\end{proof}

\begin{definition}
For any $E$-valued $G_K$-representation $V$, put
\[
H^1_{\text{Iw}}(K[np^\infty],T):=\varprojlim_rH^1(K[np^r],T),\quad
H^1_{\text{Iw}}(K[np^\infty],V):=
H^1_{\text{Iw}}(K[np^\infty],T)\otimes_{\cl{O}}E,
\]
where $T\subset V$ is a Galois stable $\cl{O}$-lattice.
\end{definition}

By another application of Shapiro's lemma, the classes $\kapinftyn$ in (\ref{eq:kapinfty}) naturally live in $H^1_{\text{Iw}}(K[np^\infty],V_{g,h}^\psi(\kappa_{\text{ac}}^{(k-2)/2}))$.
We thus arrive at the following theorem, which is the main result of this section.

\begin{theorem}\label{principal:wild}
Suppose that:
\begin{itemize}
\item $l\geq m\geq 2$ have the same parity and $k\geq 2$ is even,
\item $p$ splits in $K$,
\item $p$ does not divide the class number of $K$.
\end{itemize}
Let $\mathcal{S}$ be the set of squarefree products of primes $q$ which split in $K$ and are coprime to $p$ and $N$. Assume that $H^1(K[np^s],T_{g,h}^{\psi})$ is torsion-free for every $n\in\mathcal{S}$ and for every $s\geq 0$. There exists a collection of classes
\[
\left\lbrace\kapinftyn\in H^1_{\rm Iw}(K[np^\infty],T_{g,h}^{\psi})\;\colon\; n\in\mathcal{S}\right\rbrace
\]
such that whenever $n, nq\in\mathcal{S}$ with $q$ a prime, we have
\begin{equation}\label{eq:norm-split2}
{\rm cor}_{K[nq]/K[n]}(\kapinftynq)=P_{\frk{q}}({\rm Fr}_{\frk{q}}^{-1})\,\kapinftyn,\nonumber
\end{equation}
where $\mathfrak{q}$ is any of the primes of $K$ above $q$, and $P_{\frk{q}}(X)=\det(1-{\rm Fr}_{\frk{q}}^{-1}X\vert (V_{g,h}^\psi)^\vee(1))$.
\end{theorem}

\begin{proof}
The same argument as in the proof of Theorem~\ref{principal:tame} (but using Proposition~\ref{prop:lambdanormrelations3}) yields a system of Iwasawa cohomology classes with the stated norm-compatibilities for the representation $V_{g,h}^\psi(\kappa_{\text{ac}}^{(k-2)/2})$. By the twisting result of \cite[Thm.~6.3.5]{Rub}, the theorem follows.
\end{proof}

We conclude this section by proving that
the classes $\kapinftyn$ land in the \emph{balanced Selmer group}
\[
\Sel_{\bal}(K[np^\infty],T_{g,h}^\psi):=\varprojlim_r\Sel_{\bal}(K[np^r],T_{g,h}^\psi);
\]
in the terminology introduced in $\S\ref{sec:ES}$ below, 
this is the same as the Greenberg Selmer group ${\rm Sel}_{\rm Gr}(K[np^\infty],T_{g,h}^\psi)$ associated to the $G_{K_v}$-invariant subspaces $\mathcal{F}_v^+(V_{g,h}^\psi)\subset V_{g,h}^\psi$ in (\ref{eq:bal-Sh}) 
at the primes $v\mid p$.

\begin{proposition}\label{prop:in-Selinfty}
For all $n\in\mathcal{S}$, we have $\kapinftyn\in\Sel_{\bal}(K[np^\infty],T_{g,h}^\psi)$.
\end{proposition}

\begin{proof}
Let $v\nmid p$ be a finite prime of $K[np^\infty]$, and for every $r\geq 0$ denote also by $v$ the prime of $K[np^r]$ below $v$. As in the proof of Proposition~\ref{prop:in-Sel}, we have
\[
H^1(K[np^r]_v,V_{g,h}^\psi)=H^1_{\rm Gr}(K[np^r]_v,V_{g,h}^\psi)=0,
\]
and hence
\[
H^1(K[np^r]_v,T_{g,h}^\psi)=H^1(K[np^r]_v,T_{g,h}^\psi)_{\rm tors}=H^1_{\rm Gr}(K[np^r]_v,T_{g,h}^\psi),
\]
where the first equality follows from the local Euler characteristic formula. Hence the inclusion ${\rm res}_v(\kapinftyn)\in\varprojlim_rH^1_{\rm Gr}(K[np^r]_v,\Trep)$  follows. Since by \cite[Cor.~8.2]{BSV} it follows that the classes $\kapinftyn$ satisfy the balanced local condition at the primes above $p$, this concludes the proof.
\end{proof}

\part{Arithmetic applications}

\section{Iwasawa main conjectures}\label{2imc}

In this section we formulate Iwasawa main conjectures for triple products of modular forms. We give two formulations: one in terms of the triple product $p$-adic $L$-function (Conjecture~\ref{with}) and another in terms of diagonal cycle classes (Conjecture~\ref{without}). In Theorem~\ref{thm:equiv} we establish the equivalence of the two formulations.

\subsection{Triple product $p$-adic $L$-function}\label{subsec:Lp}


Fix a triple $(\hf,g,h)$ consisting of a primitive Hida family $\hf$ of tame level $N_{\hf}$ and character $\chi_{\hf}$ and two $p$-ordinary newforms $g, h$ of weights $l,m\geq 2$, levels $N_g,N_h$ prime-to-$p$, and nebentypus $\chi_g,\chi_h$. Assume that $\hf$ has coefficients in a ring $\Lambda_{\hf}$ as in \S\ref{subsubsec:Hida}. Assume that $\chi_{\hf}\chi_g\chi_h=\omega^{r_1}$ for some even integer $r_1$ and put
\[
N={\rm lcm}(N_{\hf},N_g,N_h).
\]

Let $\hg$ and $\hh$ be primitive Hida families with coefficients in $\Lambda_\hg$ and $\Lambda_\hh$ passing through $g$ and $h$, respectively. More precisely, there exist arithmetic points $y_0\in\cW_{\Lambda_{\hg}}(\overline{\bb{Q}}_p)$ and $z_0\in\cW_{\Lambda_{\hh}}(\overline{\bb{Q}}_p)$ such that $g_{y_0}$ and $h_{z_0}$ are the ordinary $p$-stabilizations of $g$ and $h$, respectively. The rings $\Lambda_\hg$ and $\Lambda_\hh$ need not be regular. However, for our purposes, we can consider the $\Lambda$-adic families, denoted again $\hg$ and $\hh$, that result from embedding $\Lambda_\hg$ and $\Lambda_\hh$ in the rings of functions of suitable wide open connected subsets $U_\hg$ and $U_\hh$ of $\cW(\overline{\bb{Q}}_p)=\mathrm{Spf}(\Lambda)(\overline{\bb{Q}}_p)$ defined over some finite extension $E$ of $\bb{Q}_p$ and containing the points $y_0$ and $z_0$, respectively. From now on, it is these rings of functions that we will denote by $\Lambda_\hg$ and $\Lambda_\hh$. These rings are now non-canonically isomorphic to $\cO[[T]]$, where $\cl{O}$ is the ring of integers of $E$; in particular, they are regular. Let $\bm{l}-l$ and $\bm{m}-m$ be generators in $\Lambda_\hg$ and $\Lambda_\hh$ of the prime ideals corresponding to the points $y_0$ and $z_0$, respectively.

We can and will assume that $\Lambda_\hf$ is a finite flat extension of $\Lambda_\cl{O}$ and we will only consider arithmetic points in $\cW_{\Lambda_\hf}(\overline{\bb{Q}}_p)$ lying in $\Hom_{\text{cont},\cl{O}}(\Lambda_\hf,\overline{\bb{Q}}_p)$

Recall that in $\S\ref{subsubsec:Lambda-pairing}$ we defined a character $\kappa_{r_1}:\bb{Z}_p^\times \rightarrow \Lambda^\times$  given by $u\mapsto \omega^{r_1}(u)[\langle u\rangle]$ and in $\S\ref{subsec:Iw}$ we fixed a square root $\kappa_{r_1}^{1/2}$ of this character given by $u\mapsto \omega^{r_1/2}(u)[\langle u\rangle^{1/2}]$. We let $\kappa_\hf$ and $\kappa_\hf^{1/2}$ be the composition of $\kappa_{r_1}$ and $\kappa_{r_1}^{1/2}$, respectively, with the embedding $\Lambda^\times\xhookrightarrow{} \Lambda_\hf^\times$. We also define a character $\kappa_{\hg \hh}:\bb{Z}_p^\times\rightarrow (\Lambda_{\hg}\hat{\otimes}_{\cl{O}} \Lambda_{\hh})^{\times}$ by
\[
\kappa_{\hg \hh}(u)=\omega(u)^{l+m-4}\langle u\rangle^{\bm{l}+{\mathbf m}-4}
\]
and choose a square-root of this character defined by
$\kappa_{\hg \hh}^{1/2}(u)=\omega(u)^{(l+m-4)/2}\langle u\rangle^{(\bm{l}+{\mathbf m}-4)/2}$. Let $\Lambda_{\hf \hg \hh}=\Lambda_\hf\hat{\otimes}_{\cl{O}}\Lambda_\hg \hat{\otimes}_{\cl{O}} \Lambda_\hh$ and consider the $\Lambda_{\hf \hg \hh}[G_\bb{Q}]$-module
\begin{equation}\label{eq:big-SD}
\mathbb V_{\hf\hg\hh}^\dag:=\mathbb V_{\hf}\hat\otimes_{\mathcal{O}}\mathbb V_{\hg}\hat\otimes_{\mathcal{O}}\mathbb V_{\hh}(\Xi_{\hf\hg\hh}),\quad\;\textrm{where\; $\Xi_{\hf\hg\hh}=\epsilon_{\rm cyc}^{-1} \kappa_{\hf}^{-1/2} \kappa_{\hg \hh}^{-1/2}$}
\end{equation}
and $\mathbb{V}_{\hf}$, $\mathbb{V}_{\hg}$ and $\mathbb{V}_{\hh}$ are the big Galois representations attached to $\hf$, $\hg$ and $\hh$, respectively. Then $\mathbb{V}_{\hf\hg\hh}^\dag$ is a self-dual twist of the tensor product of these representations. Consider also the $\Lambda_\hf[G_\bb{Q}]$-module
\begin{equation*}
\mathbb V_{\hf gh}^\dag:=\mathbb V_{\hf}\otimes_{\mathcal{O}} T_g \otimes_{\mathcal{O}} T_h(\Xi_{\hf gh}),\quad\;\textrm{where\; $\Xi_{\hf gh}=\epsilon_{\rm cyc}^{(2-l-m)/2} \kappa_{\hf}^{-1/2}$.}
\end{equation*}

Given test vectors $(\breve{\hf},\breve{g},\breve{h})$ for $(\hf,g,h)$ of level $N$, as explained in \cite{harris-tilouine} and \cite[\S{4.2}]{DR1}, a generalization of Hida's $p$-adic Rankin--Selberg convolution  produces an element $\mathscr{L}_p(\breve{\hf},\breve{g},\breve{h})$ in the fraction field of $\Lambda_{\hf}$ whose specializations to arithmetic points $x\in \cW_{\Lambda_\hf}(\overline{\Q}_p)$ of even weight $k\geq l+m$ 
recover (a square-root of) the central critical values of the triple product $L$-function 
$L(\mathbb V_{\hf_xgh}^\dag,s)$ for the specialization of $\mathbb V_{\hf gh}^\dag$ at $x$ by virtue of Harris--Kudla's proof of Jacquet's conjecture, \cite{harris-kudla}. A recent result by Hsieh \cite{Hs} constructs test vectors $(\underline{\breve{\hf}},\underline{\breve{g}},\underline{\breve{h}})$ for which a precise interpolation property for the resulting $\mathscr{L}_p(\underline{\breve{\hf}},\underline{\breve{g}},\underline{\breve{h}})$ is proved. To recall the result in the form that will be used here, for any arithmetic point $x\in \cW_{\Lambda_\hf}(\overline{\Q}_p)$ as above, we set
\[
\hf_k:=\hf_x,\quad\alpha_{k}:=a_p(\hf_k),\quad\beta_{k}:=\chi_{\hf}(p)p^{k-1}\alpha_{k}^{-1},
\]
let $\alpha_g, \beta_g$ be the roots of the Hecke polynomial of $g$ at $p$ with ${\rm ord}_p(\alpha_g)=0$, and define $\alpha_h, \beta_h$ similarly. 
As recalled in [\emph{op.\,cit.}, \S{1.4}], when the residual Galois representation $\bar{\rho}_{\hf}$ associated to $\hf$ is absolutely irreducible and $p$-distinguished, the local ring $\Lambda_{\hf}$ is known to be Gorenstein and by a result of Hida's the congruence module of $\hf$ is isomorphic to $\Lambda_{\hf}/(\xi)$ for some nonzero $\xi\in\Lambda_{\hf}$. We call $(\xi)$ the \emph{congruence ideal} of $\hf$. Finally, denote by 
$\varepsilon_\ell(\mathbb V_{\hf_kgh}^\dag)\in\{\pm{1}\}$ the epsilon factor of the Weil--Deligne representation attached to the restriction of $\mathbb V_{\hf_kgh}^\dag$ to $G_{\Q_\ell}$.


\begin{theorem}\label{thm:hsieh}
In addition to $\chi_{\hf}\chi_{g}\chi_{h}=\omega^{r_1}$, assume that:
\begin{itemize}
\item[{\rm (a)}] $\bar{\rho}_{\hf}$ is absolutely irreducible and $p$-distinguished, 
\item[{\rm (b)}] for some arithmetic point $x\in\cW_{\Lambda_\hf}(\overline{\Q}_p)$, we have $\varepsilon_\ell(\mathbb V_{\hf_kgh}^\dag)=+1$ for all primes $\ell\mid N$,
\item[{\rm (c)}] ${\rm gcd}(N_{\hf},N_g,N_h)$ is squarefree.
\end{itemize}
Let $\xi$ be a generator of the congruence ideal of $\hf$. There exist test vectors $(\underline{\breve{\hf}},\underline{\breve{g}},\underline{\breve{h}})$ for $(\hf,g,h)$ of level $N$, and an element
\[
\mathscr{L}_p^\xi(\underline{\breve{\hf}},\underline{\breve{g}},\underline{\breve{h}})\in\Lambda_{\hf}
\]
such that for all arithmetic points $x\in \cW_{\Lambda_\hf}(\overline{\Q}_p)$ of even weight $k\geq l+m$ with $k\equiv r_1+2\pmod{2(p-1)}$ 
we have
\[
\biggl(\frac{\mathscr{L}_p^\xi(\underline{\breve{\hf}},\underline{\breve{g}},\underline{\breve{h}})(x)}{\xi_x}\biggr)^2=\frac{\Gamma(k,l,m)}{2^{\alpha(k,l,m)}}\cdot\frac{\mathcal{E}(\hf_k,g,h)^2}{\mathcal{E}_0(\hf_k)^2\cdot\mathcal{E}_1(\hf_k)^2}\cdot\prod_{\ell\mid N}\tau_\ell\cdot\frac{L(\mathbb V_{\hf_kgh}^\dag,0)}{\pi^{2k}\cdot\langle\hf_k^\sharp, \hf_k^\sharp\rangle^2},
\]
where:
\begin{itemize}
	\item $\Gamma(k,l,m)=(c-1)!\cdot(c-m)!\cdot(c-l)!\cdot(c+1-l-m)!$, with $c=(k+l+m-2)/2$,
	\item $\alpha(k,l,m)\in\Lambda_{\hf}$ is a linear form in the variables $k$, $l$, $m$,
	\item $\mathcal{E}(\hf_k,g,h)=(1-\frac{\beta_k\alpha_{g}\alpha_{h}}{p^{c}})(1-\frac{\beta_k\beta_{g}\alpha_{h}}{p^{c}})(1-\frac{\beta_k\alpha_{g}\beta_{h}}{p^{c}})(1-\frac{\beta_k\beta_{g}\beta_{h}}{p^{c}})$,
\item
$\mathcal{E}_0(\hf_k)=(1-\frac{\beta_k}{\alpha_k})$,
$\mathcal{E}_1(\hf_k)=(1-\frac{\beta_k}{p\alpha_k})$,
\item $\tau_\ell$ is an explicit nonzero rational number independent of $k$,
\item $\hf_k^\sharp$ is the newform associated to the $p$-stabilized newform $\hf_k$,
\end{itemize}
and
$\lVert\hf_k^\sharp\rVert^2$ is the Petersson norm of $\hf_k^\sharp$.
\end{theorem}

\begin{proof}
Letting $\hg, \hh$ be the primitive Hida families of tame level $N_g, N_h$ passing through the ordinary $p$-stabilizations of $g, h$, this follows by specializing the three-variable $p$-adic $L$-function in \cite[Thm.~A]{Hs} attached to 
$({\hf},{\hg},{\hh})$
and the congruence ideal generator $\xi$.
\end{proof}

\begin{definition}\label{def:optimal-triple}
For the test vectors $(\underline{\breve{\hf}},\underline{\breve{g}},\underline{\breve{h}})$ of level $N$ provided by Theorem~\ref{thm:hsieh}, we set
\begin{equation}\label{optimal-triple}
L_p(\hf,g,h):=\mathscr{L}_p^\xi(\underline{\breve{\hf}},\underline{\breve{g}},\underline{\breve{h}})^2,\nonumber
\end{equation}
where $\xi$ is any fixed generator of the congruence ideal of $\hf$.
\end{definition}


Note that $L_p(\hf,g,h)$ depends on the choice of $\xi$, but the principal ideal in $\Lambda_{\hf}$ it generates is of course independent of that choice.

\subsection{Reciprocity law for diagonal cycles} \label{subsec:ERL}

Keep the notations in the previous subsection and without loss of generality assume that $l\geq m$ (reordering $g$ and $h$ if necessary).

Assume that the Galois representations attached to $\hf$, $g$ and $h$ are all residually irreducible and $p$-distinguished. Let $\boldsymbol{\phi}\in\{\hf,\hg,\hh\}$. As a $G_{\mathbb{Q}_p}$-representation, $\mathbb V_{\boldsymbol{\phi}}$ admits a filtration
\begin{equation}\label{primfil}
0 \rightarrow \mathbb V_{\boldsymbol{\phi}}^+ \rightarrow \mathbb V_{\boldsymbol{\phi}} \rightarrow \mathbb V_{\boldsymbol{\phi}}^- \rightarrow 0
\end{equation}
with each $\mathbb V_{\boldsymbol{\phi}}^\pm$ free of rank one over $\Lambda_{\boldsymbol{\phi}}$, and with the $G_{\mathbb{Q}_p}$-action on $\mathbb V_{\boldsymbol{\phi}}^-$ given by the unramified character sending ${\rm Fr}_p\mapsto a_p(\boldsymbol{\phi})$.
%
%
%
This induces an obvious three-step filtration
\[
0\subset\mathscr{F}^3\mathbb V_{\hf\hg\hh}^{\dag} \subset\mathscr{F}^2\mathbb V_{\hf\hg\hh}^\dag\subset\mathscr{F}^1\mathbb V_{\hf\hg\hh}^\dag \subset \mathbb V_{\hf\hg\hh}^\dag
\]
by $G_{\mathbb Q_p}$-stable $\Lambda_{\hf\hg\hh}$-submodules of ranks 1, 4, and 7, respectively, given by
\begin{equation}
\begin{split}\label{eq:+}
\mathscr{F}^1\mathbb V_{\hf\hg\hh}^\dag &= (\mathbb V_{\hf}\hat\otimes_{\mathcal{O}}\mathbb V_{\hg} \hat\otimes_{\mathcal{O}}\mathbb V_{\hh}^+ +\mathbb V_{\hf} \hat\otimes_{\mathcal{O}} \mathbb V_{\hg}^+ \hat\otimes_{\mathcal{O}} \mathbb V_{\hh} +\mathbb V_{\hf}^+ \hat\otimes_{\mathcal{O}} \mathbb V_{\hg} \hat\otimes_{\mathcal{O}} \mathbb V_{\hh})(\Xi_{\hf\hg\hh}),\\
\mathscr{F}^2\mathbb V_{\hf\hg\hh}^\dag &= (\mathbb V_{\hf}\hat\otimes_{\mathcal{O}}\mathbb V_{\hg}^+ \hat\otimes_{\mathcal{O}}\mathbb V_{\hh}^+ +\mathbb V_{\hf}^+ \hat\otimes_{\mathcal{O}} \mathbb V_{\hg} \hat\otimes_{\mathcal{O}} \mathbb V_{\hh}^+ +\mathbb V_{\hf}^+ \hat\otimes_{\mathcal{O}} \mathbb V_{\hg}^+ \hat\otimes_{\mathcal{O}} \mathbb V_{\hh})(\Xi_{\hf\hg\hh}),\\
\mathscr{F}^3\mathbb V_{\hf\hg\hh}^{\dag}&=\mathbb V_{\hf}^+\hat\otimes_{\mathcal{O}}\mathbb V_{\hg}^+ \hat\otimes_{\mathcal{O}}\mathbb V_{\hh}^+(\Xi_{\hf\hg\hh}).
\end{split}
\end{equation}
The middle term $\mathscr{F}^2\mathbb V_{\hf\hg\hh}^\dag$ will play a special role in the following, and we note that
\begin{equation}\label{eq:2-3}
\mathscr{F}^2\mathbb V_{\hf\hg\hh}^\dag/\mathscr{F}^3\mathbb V_{\hf\hg\hh}^{\dag}\cong\mathbb V_\hf^{\hg\hh}\oplus\mathbb V_\hg^{\hf\hh}\oplus\mathbb V_\hh^{\hf\hg},
\end{equation}
where $\mathbb V_\hf^{\hg\hh}:=(\mathbb V_{\hf}^-\hat\otimes_{\mathcal{O}}\mathbb V_{\hg}^+\hat\otimes_{\mathcal{O}}\mathbb V_{\hh}^+)(\Xi_{\hf\hg\hh})$ and similarly for the other two direct summands. We similarly denote the induced subquotients on the specializations of $\mathbb V_{\hf\hg\hh}^\dag$ (that is, $\mathscr{F}^i\mathbb V_{\hf gh}^\dagger, \mathbb V_{\hf}^{gh}$, etc.).

Consider the class $\kappa_{1,\hf gh}^{(3)}$ defined in (\ref{eq:lambda-3}) for the choice of level-$N$ test vectors $(\underline{\breve{\hf}},\underline{\breve{g}},\underline{\breve{h}})$ given by Theorem~\ref{thm:hsieh} and let $\kappa(\hf,g,h)\in H^1(\mathbb Q, \mathbb V_{\hf gh}^{\dag})$ be the image of this class via the morphism obained from the augmentation map $\cl{O}[R_1]\rightarrow \cl{O}$. By \cite[Cor.~8.2]{BSV}, the image of $\kappa(\hf,g,h)$ under the restriction map at $p$ 
is contained in
\begin{equation}\label{eq:bal}
H^1_{\bal}(\mathbb Q_p,\mathbb V_{\hf gh}^\dag):={\rm im}\bigl(H^1(\Q_p,\mathscr{F}^2\mathbb V_{\hf gh}^\dag)\rightarrow H^1(\Q_p,\mathbb V_{\hf gh}^\dag)\bigr).\nonumber
\end{equation}
It is easily seen that this map 
is an injection, so we may and will view ${\rm res}_p(\kappa(\hf,g,h))$ as a class in $H^1(\Q_p,\mathscr{F}^2\mathbb V_{\hf gh}^\dag)$. 
Let
\[
{\rm pr}_{(\hf,g,h)}:\mathscr{F}^2\mathbb V_{\hf gh}^\dag\longrightarrow\mathbb V_{\hf}^{gh}
\]
be the map induced by the projection onto the first direct summand in (\ref{eq:2-3}). 
The ``reciprocity law'' from \cite{BSV,DR3} recalled in Theorem~\ref{rec-law} below relates the image of ${\rm res}_p(\kappa(\hf,g,h))$ under the natural projection
\[
{\rm pr}_{(\hf,g,h)*}:H^1_{\bal}(\mathbb Q_p,\mathbb V_{\hf gh}^\dag)\longrightarrow H^1(\mathbb Q_p,\mathbb V_{\hf}^{gh})
\]
to the triple product $p$-adic $L$-function of $\S\ref{subsec:Lp}$. 
%
%
Recall that $\xi\in\Lambda_{\hf}$ denotes a generator of the congruence ideal of $\hf$.

\begin{propo}\label{perrin}
There is an injective $\Lambda_{\hf}$-module homomorphism with pseudo-null cokernel
\[
\mathfrak{Log}^\xi
:H^1(\mathbb Q_p,\mathbb{V}_{\hf}^{gh})\longrightarrow\Lambda_{\hf}
\]
characterized by the following interpolation property: for all $\mathfrak{Z}\in H^1(\mathbb Q_p,\mathbb{V}_{\hf}^{gh})$ and all classical points $x\in \cW_{\Lambda_\hf}(\overline{\Q}_p)$ of weight $k\geq l+m$ with $k\equiv r_1+2\pmod{2(p-1)}$ we have
\begin{align*}
\frac{\mathfrak{Log}^\xi(\mathfrak{Z})(x)}{\xi_x}&=(p-1)\alpha_{k}\biggl(1-\frac{\beta_{k} \alpha_g \alpha_h}{p^{c}}\biggr)\biggl(1-\frac{\alpha_{k} \beta_g \beta_h}{p^{c}}\biggr)^{-1}\\
&\quad\times\begin{cases}
\frac{(-1)^{c-k}}{(c-k)!}\cdot\left\langle{\rm log}_p(\mathfrak{Z}_k),\eta_{\hf_k}\otimes\omega_{\hg_l}\otimes\omega_{\hh_m}\right\rangle_{\rm dR}, &\textrm{if $l-m<k<l+m$,}\\[0.5em]
(k-c-1)!\cdot\left\langle{\rm exp}_p^*(\mathfrak{Z}_{k}),\eta_{\hf_k}\otimes\omega_{\hg_l}\otimes\omega_{\hh_m}\right\rangle_{\rm dR},&\textrm{if $k\geq l+m$,}
\end{cases}
\end{align*}
where $c=(k+l+m-2)/2$.
\end{propo}

\begin{proof}
The construction of $\mathfrak{Log}^\xi$ will follow by specializing the three-variable $p$-adic regulator constructed in \cite[\S{7.1}]{BSV} (building on a generalization of the construction in \cite{LZ2} given by  Kings--Loeffler--Zerbes \cite{KLZ}).

Let $\vartheta_{gh}:\Lambda_{\hf \hg \hh}\rightarrow\Lambda_{\hf}$ be the map given by reduction modulo $(\bm{l}-l,\bm{m}-m)$. This induces  isomorphisms
\[
\mathbb V_{\hf\hg\hh}^\dag\otimes_{\Lambda_{\hf\hg\hh}}\Lambda_\hf\cong\mathbb V_{\hf gh}^\dag,\quad\mathbb V_{\hf}^{\hg\hh}\otimes_{\Lambda_{\hf\hg\hh}}\Lambda_\hf\cong\mathbb V_{\hf}^{gh},
\]
and a natural map
\begin{equation}\label{sp-gh}
\vartheta_{gh*}:H^1(\mathbb Q_p,\mathbb V_{\hf}^{\hg\hh})\otimes_{\Lambda_{\hf\hg\hh}}\Lambda_\hf\longrightarrow H^1(\mathbb Q_p,\mathbb V_{\hf}^{gh}).\nonumber
\end{equation}
This map is clearly injective, and its surjectivity can be shown easily by an application of local Tate duality and the Ramanujan--Petersson conjecture (\emph{cf.} proof of \cite[(154)]{BSV}). Letting
\begin{equation}\label{3-var}
\mathscr{L}_{\hf}^\xi:H^1(\mathbb Q_p,\mathbb V_{\hf}^{\hg\hh})\longrightarrow\Lambda_{\hf\hg\hh}\nonumber
\end{equation}
be the $p$-adic regulator $\mathscr{L}_{\hf}$ defined as in \cite[Prop.~7.3]{BSV} 
and multiplied by $\xi$, the map defined by the composition
\[
\mathfrak{Log}^\xi:H^1(\mathbb Q_p,\mathbb V_{\hf}^{gh})\xrightarrow{\vartheta_{gh*}^{-1}} H^1(\mathbb Q_p,\mathbb V_{\hf}^{\hg\hh})\otimes_{\Lambda_{\hf\hg\hh}}\Lambda_\hf\xrightarrow{\mathscr{L}^\xi_{\hf}\otimes{\rm id}}\Lambda_\hf
\]
satisfies the interpolation properties in the statement of the proposition.

It remains to see that $\mathfrak{Log}^\xi$ is injective with pseudo-null cokernel. By definition, we have
\[
\mathbb V_{\hf}^{\hg\hh}=\mathbb U_{\hf}^{\hg\hh}(\epsilon_{\rm cyc}\kappa_{\hf}^{-1/2}\kappa_{\hg\hh}^{1/2}),
\]
where $\mathbb U_{\hf}^{\hg\hh}$ is an unramified $G_{\mathbb Q_p}$-module  
on which an arithmetic Frobenius ${\rm Fr}_p$ acts as multiplication by $\chi_f^{-1}(p)a_p(\hf)a_p(\hg)^{-1}a_p(\hh)^{-1}$, and $\mathscr{L}_{\hf}$
is obtained by specializing the four-variable $p$-adic regulator map in \cite[Thm.~8.2.3]{KLZ} for the module $\mathbb U_{\hf}^{\hg\hh}$, paired against the differential $\eta_{\hf}\otimes \omega_{\hg}\otimes\omega_{\hh}$.
In light of \cite[Rmk.~8.2.4]{KLZ}, the fact that $\mathfrak{Log}^\xi$ has the above properties can therefore be deduced from the vanishing of $H^0(\mathbb Q_p,\mathbb U_{\hf}^{gh})$, where $\mathbb{U}_{\hf}^{gh}$ is the image of $\mathbb U_{\hf}^{\hg\hh}$ under $\vartheta_{gh}$.
\end{proof}


\begin{theorem}[Reciprocity law]\label{rec-law}
We have the following equality
\[
\mathfrak{Log}^\xi({\rm res}_p(\kappa(\hf,g,h))) = \mathscr{L}_p^\xi(\underline{\breve{\hf}},\underline{\breve{g}},\underline{\breve{h}}).
\]
\end{theorem}

\begin{proof}
This is the specialization of the three-variable reciprocity law of Theorem~A in \cite{BSV} to $(\hf,g,h)$ (see also \cite[Thm.~10]{DR3}).
\end{proof}

%

\subsection{Selmer groups and formulation of the main conjectures}\label{subsec:IMC}

Let $(\hf,g,h)$ be as in the preceding subsection. Throughout the rest of this section, we assume that hypotheses (a)--(c) in Theorem~\ref{thm:hsieh} hold, so the $p$-adic $L$-function $L_p(\hf,g,h)$ in Definition~\ref{def:optimal-triple} is available.

Recall the $G_{\mathbb Q_p}$-stable rank-four $\Lambda_{\hf\hg\hh}$-submodule $\mathscr{F}^2\mathbb V_{\hf\hg\hh}^\dag\subset \mathbb V_{\hf\hg\hh}^{\dag}$ in $(\ref{eq:+})$, and set
\[
\mathbb V_{\hf\hg\hh}^f:=\mathbb V_{\hf}^+\hat\otimes_{\mathcal{O}}\mathbb V_{\hg}\hat\otimes_{\mathcal O}\mathbb V_{\hh}(\Xi_{\hf\hg\hh}).
\]
As before, we let $\mathscr{F}^2\mathbb V_{\hf gh}$ and $\mathbb V_{\hf gh}^f$ denote the corresponding specializations. 

Fix a finite set $\Sigma$ of places of $\mathbb Q$ containing $\infty$ and the primes dividing $Np$, and let $\mathbb Q^\Sigma$ be the maximal extension of $\mathbb Q$ unramified outside $\Sigma$.

\begin{definition}\label{def:Sel}
For $\mathcal{L}\in\{{\rm bal},\mathcal{F}\}$ define the Selmer group $\Sel_{\mathcal{L}}(\mathbb V_{\hf gh}^{\dag})$ by
\[
\Sel_{\mathcal{L}}(\mathbb V_{\hf gh}^{\dag}) = \ker \bigg(
H^1(\mathbb Q^\Sigma/\mathbb Q, \mathbb V_{\hf gh}^{\dag}) \longrightarrow
\frac{H^1(\mathbb Q_p,\mathbb V_{\hf gh}^\dag)}{H^1_{\mathcal{L}}(\mathbb Q_p,\mathbb V_{\hf gh}^\dag)}\bigg),
\]
where
\[
H_{\mathcal L}^1(\mathbb Q_p, \mathbb V_{\hf gh}^{\dag}) =
\begin{cases}
\ker \bigl(H^1(\mathbb Q_p, \mathbb V_{\hf gh}^{\dag}) \longrightarrow H^1(\mathbb Q_p, \mathbb V_{\hf gh}^{\dag}/\mathscr{F}^2\mathbb V_{\hf gh}^{\dag}) \bigr) &\textrm{if $\mathcal{L}={\rm bal}$},\\[0.5em]
\ker \bigl(H^1(\mathbb Q_p, \mathbb V_{\hf gh}^{\dag}) \longrightarrow H^1(\mathbb Q_p^{}, \mathbb V_{\hf gh}^{\dag}/ \mathbb V_{\hf gh}^{f}) \bigr) &\textrm{if $\mathcal{L}=\mathcal{F}$}.
\end{cases}
\]
We call $\Sel_{\rm bal}(\mathbb V_{\hf gh}^{\dag})$ (resp. $\Sel_{\mathcal F}(\mathbb V_{\hf gh}^{\dag})$)  the \emph{balanced} (resp. \emph{$f$-unbalanced}) Selmer group.
\end{definition}

\begin{remark}\label{rem:Gr}
The pairs
\[
\bigl(\mathscr{F}^2\mathbb V_{\hf gh}^\dag,\{k\in\Z_{\geq 2}\;\colon\; l-m<k<l+m\}\bigr),\quad\bigl(\mathbb V_{\hf gh}^f,\{k\in\Z_{}\;\colon\;k\geq l+m\}\bigr)
\]
satisfy the \emph{Panchishkin condition} in \cite{Gr94}. Thus ${\rm Sel}_{\bal}(\mathbb V_{\hf gh}^\dag)$ and ${\rm Sel}_{\mathcal F}(\mathbb V_{\hf gh}^\dag)$ may be viewed as instances of Greenberg's Selmer groups attached to different ranges of critical specializations of $\mathbb V_{\hf gh}^\dag$.
\end{remark}

Let 
\[
\mathbb A_{\hf gh}^\dag={\rm Hom}_{\mathbb Z_p}(\mathbb V_{\hf gh}^\dag,\mu_{p^\infty}).
\] 
Then for $\mathcal{L}\in\{\bal,\mathcal{F}\}$, we define the Selmer groups ${\rm Sel}_{\mathcal L}(\mathbb A_{\hf gh}^\dag)$ as above,
taking $H^1_{\mathcal L}(\mathbb Q_p,\mathbb A_{\hf gh}^\dag)$ to be the orthogonal complement of $H^1_{\mathcal L}(\mathbb Q_p,\mathbb V_{\hf gh}^\dag)$ under the local Tate duality
\[
H^1(\mathbb Q_p,\mathbb V_{\hf gh}^\dag)\times H^1(\mathbb Q_p,\mathbb A_{\hf gh}^\dag)\longrightarrow\bb{Q}_p/\bb{Z}_p,
\]
and set
\[
X_{\mathcal L}(\mathbb A_{\hf gh}^{\dag}):=\Hom_{\text{cont}}(\Sel_{\mathcal L}(\mathbb A_{\hf gh}^{\dag}), \mathbb Q_p/\mathbb Z_p).
\]


In light of Remark~\ref{rem:Gr}, the next conjecture may be viewed as an instance of the Iwasawa--Greenberg main conjectures \cite{Gr94}. In the two formulations below, we also assume conditions (b) and (c) from Theorem~\ref{thm:hsieh}, so that the $p$-adic $L$-function $L_p(\hf,g,h)$ in (\ref{optimal-triple}) is defined.

\begin{conj}[IMC ``with $p$-adic $L$-functions'']\label{with}
The modules $\Sel_{\mathcal F}(\mathbb V_{\hf gh}^{\dag})$ and $X_{\mathcal F}(\mathbb A_{\hf gh}^{\dag})$ are both $\Lambda_{\hf}$-torsion, and
\[
\Char_{\Lambda_{\hf}}(X_{\mathcal F}(\mathbb A_{\hf gh}^{\dag})) = (L_p(\hf,g,h))
\]
in $\Lambda_{\hf}\otimes_{\mathbb Z_p}\mathbb Q_p$.
\end{conj}

\begin{remark}
An integral formulation of the equality of ideals in Conjecture~\ref{with} would in general involve certain Tamagawa factors, accounting for the fact that the construction of $L_p(\hf,g,h)$ uses Hida's congruence number, while by definition the classes in the Selmer group $X_{\cF}(\mathbb A_{\hf gh}^\dag)$ are trivial at the places $v\in\Sigma\smallsetminus\{p,\infty\}$, rather than just unramified (\emph{cf.} \cite{PW}).
\end{remark}

Under the local root number hypothesis (b) in Theorem~\ref{thm:hsieh}, for all arithmetic point $x\in\cW_{\Lambda_\hf}(\overline{\bb{Q}}_p)$ of even weight $k\geq 2$ with $l-m<k<l-m$ the sign in the functional equation for $L(\mathbb V_{\hf_x gh}^\dag,s)$ is $-1$, so that the central value $L(\mathbb V_{\hf_x gh}^\dag,0)$ vanishes. Therefore, in the spirit of Perrin-Riou's main conjecture \cite[Conj.~B]{PR} in the setting of Heegner points, a natural formulation of the Iwasawa main conjecture for ${\rm Sel}_{\bal}(\mathbb V_{\hf gh}^\dag)$ takes the following form.

Note that it follows from \cite[Cor.~8.2]{BSV} that $\kappa(\hf,g,h)$ lands in $\Sel_{\bal}(\mathbb V_{\hf gh}^\dag)$.

\begin{conj}[IMC ``without $p$-adic $L$-functions'']\label{without}
Suppose $\kappa(\hf,g,h)\in\Sel_{\bal}(\Q,\mathbb V_{\hf gh}^\dag)$  is not $\Lambda_{\hf}$-torsion. Then the modules $\Sel_{\bal}(\mathbb V_{\hf gh}^{\dag})$ and $X_{\bal}(\mathbb A_{\hf gh}^{\dag})$ have both rank one, and
\[
\Char_{\Lambda_{\hf}}(X_{\bal}(\mathbb A_{\hf gh}^{\dag})_{\tors})=
\Char_{\Lambda_{\hf}} \left(\frac{\Sel_{\bal}(\mathbb V_{\hf gh}^{\dag})}{\Lambda_{\hf} \cdot\kappa(\hf,g,h)} \right)^2
\]
in $\Lambda_{\hf}\otimes_{\mathbb Z_p}\mathbb Q_p$, where the subscript ${\rm tors}$ denotes the $\Lambda_{\hf}$-torsion submodule.
\end{conj}

\begin{remark}
Working under different hypotheses on the local signs  ensuring that $L(\mathbb V_{\hf_kgh}^\dag,s)$ has sign $+1$ (rather than $-1$) for weights $k\geq 2$ with $l-m<k<l-m$, the Iwasawa main conjecture would relate the characteristic ideal of $X_{\bal}(\mathbb A_{\hf gh}^{\dag})$
to the balanced triple product $p$-adic $L$-function constructed in \cite[Thm.~B]{Hs} (see also \cite{GS}), rather than diagonal classes. In this setting, the $f$-unbalanced Selmer group $\Sel_{\mathcal F}(\mathbb V_{\hf gh}^{\dag})$ should have $\Lambda_{\hf}$-rank one, but the expected non-torsion Selmer class seems to not have been constructed yet.
\end{remark}

\subsection{Equivalence of the formulations}\label{subsec:equiv-IMC}

In this subsection we show that the two formulations of the Iwasawa main conjecture in the previous subsection are essentially equivalent, focusing on the case where $\hf$ is a CM Hida family\footnote{This case will suffice for our applications in this paper, and makes some of the arguments simpler, but we expect the equivalence to hold in general.} as in $\S\ref{subsec:Iw}$.
Similar equivalences between IMC ``with'' and ``without'' $p$-adic $L$-functions appear in \cite[\S{17}]{Kato}, \cite[\S{11}]{KLZ} and, in a setting more germane to ours, \cite{wan} and \cite[Appendix]{Cas}.

The following intermediate Selmer groups will allow us to bridge between ${\rm Sel}_{\bal}(\mathbb V_{\hf gh}^\dag)$ and ${\rm Sel}_{\mathcal F}(\mathbb V_{\hf gh}^\dag)$ in the comparison. Set
\[
\mathbb V_{\hf gh}^{f \cap +} = \mathbb V_{\hf gh}^f \cap  \mathscr{F}^2\mathbb V_{\hf gh}^\dag, \quad \mathbb V_{\hf gh}^{f \cup +} = \mathbb V_{\hf gh}^f + \mathscr{F}^2\mathbb V_{\hf gh}^\dag,
\]
which are $G_{\mathbb Q_p}$-stable $\Lambda_\hf$-submodules of $\mathbb V_{\hf gh}^\dag$ of ranks $3$ and $5$, respectively.  Define $\Sel_{\mathcal L}(\mathbb V_{\hf gh}^{\dag})$ for $\mathcal{L}\in\{\mathcal F\cap +,\mathcal{F}\cup +\}$ by the same recipe as in Definition~\ref{def:Sel}, with
\[
H_{\mathcal L}^1(\mathbb Q_p, \mathbb V_{\hf gh}^{\dag}) =
\begin{cases}
\ker \bigl(H^1(\mathbb Q_p, \mathbb V_{\hf gh}^{\dag}) \rightarrow H^1(\mathbb Q_p, \mathbb V_{\hf gh}^{\dag}/ \mathbb V_{\hf gh}^{f\cap+}) \bigr) &\textrm{if $\mathcal{L}=\mathcal{F}\cap +$},\\[0.5em]
\ker \bigl(H^1(\mathbb Q_p, \mathbb V_{\hf gh}^{\dag}) \rightarrow H^1(\mathbb Q_p^{}, \mathbb V_{\hf gh}^{\dag}/ \mathbb V_{\hf gh}^{f\cup +}) \bigr) &\textrm{if $\mathcal{L}=\mathcal{F}\cup +$}.
\end{cases}
\]
We define the Selmer groups $\Sel_{\mathcal F\cap +}(\mathbb A_{\hf gh}^{\dag})$ and $\Sel_{\mathcal F\cup +}(\mathbb A_{\hf gh}^{\dag})$ taking $H_{\mathcal F\cap +}^1(\mathbb Q_p, \mathbb A_{\hf gh}^{\dag})$ and $H_{\mathcal F\cup +}^1(\mathbb Q_p, \mathbb A_{\hf gh}^{\dag})$ to be the orthogonal complements of $H_{\mathcal F\cup +}^1(\mathbb Q_p, \mathbb V_{\hf gh}^{\dag})$ and $H_{\mathcal F\cap +}^1(\mathbb Q_p, \mathbb V_{\hf gh}^{\dag})$, respectively. As in the preceding section, we also define the corresponding $X_{\mathcal F\cap +}(\mathbb A_{\hf gh}^{\dag})$ and $X_{\mathcal F\cup +}(\mathbb A_{\hf gh}^{\dag})$.

Throughout this subsection, we keep the setting from $\S\ref{subsec:Iw}$. In particular, $\hf\in\Lambda_{\hf}[[q]]$ is the CM Hida family in $(\ref{eq:CM-F})$ associated with the Hecke character $\psi$ of conductor $\frk{f}$. In addition, we assume  conditions (b) and (c) from Theorem~\ref{thm:hsieh}, so the $p$-adic $L$-function $L_p(\hf,g,h)\in\Lambda_{\hf}$ is defined, and let $\kappa(\hf,g,h)\in H^1(\Q,\mathbb V_{\hf gh}^\dag)$ be as above.

For every height one prime $\frk{Q}$ of $\Lambda_{\hf}$ away from $p$, let $S_{\frk{Q}}$ be the integral closure of $\Lambda_{\hf}/\frk{Q}$
and let $\Phi_{\frk{Q}}$ be the fraction field of $S_\mathfrak Q$. Let $\mathbb V_{\hf gh,\frk{Q}}^\dag$ be the extension of scalars of $\mathbb V_{\hf gh}^{\dag}/\mathfrak Q\mathbb V_{\hf gh}^\dag$ to $S_{\mathfrak Q}$,
and let $\mathbb A^\dag_{\hf gh,\mathfrak Q}=\Hom(\mathbb V^\dag_{\hf gh,\mathfrak Q},\mu_{p^\infty})$.
Following \cite{MR}, define
\begin{equation}\label{eq:propagate}
H^1_{\underline{\rm bal}}(\Q_v,\mathbb V_{\hf gh,\mathfrak Q}
^{\dag}):=
\begin{cases}
\ker\bigl(H^1(\Q_v,\mathbb V_{\hf gh,\mathfrak Q}^{\dag})\rightarrow H^1(\Q_v^{\rm nr},\mathbb V_{\hf gh,\mathfrak Q}^{\dag}\otimes\Phi_{\frk{Q}})\bigr),&\textrm{if $v\nmid p$,}\\[0.5em]
\ker\bigl(H^1(\Q_v,\mathbb V_{\hf gh,\mathfrak Q}^{\dag})\rightarrow H^1(\Q_v,(\mathbb V_{\hf gh,\mathfrak Q}^{\dag}/\mathscr{F}^2\mathbb V_{\hf gh,\mathfrak Q}^\dag)\otimes\Phi_{\frk{Q}})\bigr),&\textrm{if $v\mid p$,}\end{cases}
\end{equation}
and let $H^1_{\underline{\bal}}(\Q,\mathbb V_{\hf gh,\mathfrak Q}^{\dag})$ be the associated Selmer group. Taking $H^1_{\underline{\bal}}(\Q_v,\mathbb A_{\hf gh,\mathfrak Q}^{\dag})$ to be the orthogonal complement of $H^1_{\underline{\bal}}(\Q_v,\mathbb V_{\hf gh,\mathfrak Q}^{\dag})$ under local Tate duality, we define the Selmer group $H^1_{\underline{\bal}}(\Q,\mathbb A_{\hf gh,\mathfrak Q}^{\dag})$ similarly.

Define $\bb{V}_{g,h}^\psi=\Lambda_{\cl{O}}(\bm{\kappa}_{\text{ac}}^{-1})\hat{\otimes}_{\cl{O}} T_{g,h}^\psi(\kappa_{\ac}^{r_1/2})$ and let $\bb{A}_{g,h}^\psi=\Hom((\bb{V}_{g,h}^\psi)^c,\mu_{p^\infty})$, where $(\bb{V}_{g,h}^\psi)^c$ denotes $\bb{V}_{g,h}^\psi$ with the $G_K$-action twisted by complex conjugation.
Note that $\mathbb V_{\hf gh}^{\dag}=\text{Ind}_{K}^\bb{Q}\bb{V}_{g,h}^\psi$, so we can define Selmer conditions for $\bb{V}_{g,h}^\psi$ using Shapiro's lemma and for $\bb{A}_{g,h}^\psi$ by duality. Define $\bb{A}_{g,h,\frk{Q}}^\psi=\Hom((\bb{V}_{g,h,\frk{Q}}^\psi)^c,\mu_{p^\infty})$. We have natural maps
\begin{equation}\label{eq:spQ}
\mathbb V_{g,h}^\psi/\frk{Q}\mathbb V_{g,h}^\psi\rightarrow\mathbb V_{g,h,\frk{Q}}^\psi,\quad
\mathbb A_{g,h,\frk{Q}}^\psi\rightarrow\mathbb A_{g,h}^\psi[\frk{Q}]
\end{equation}
preserving both the $G_K$ and the $\Lambda$-modules structure in the same  way as in \cite[p.~1461]{How}. Note that in the quotient $\mathbb V_{g,h}^\psi/\frk{Q}\mathbb V_{g,h}^\psi$ and in the submodule $\mathbb A_{g,h}^\psi[\frk{Q}]$ we can define Selmer conditions by propagating the balanced conditions for $\mathbb V_{g,h}^\psi$ and $\mathbb A_{g,h}^\psi$, respectively, and we denote these conditions in the same way.

\begin{lemma}\label{lem:spQ}
For every height one prime $\frk{Q}\subset\Lambda_{\hf}$ as above and every place $v$ of $K$, the maps $(\ref{eq:spQ})$ induce natural maps
\begin{align*}
H^1_{\bal}(K_v,\mathbb V_{gh}^{\psi}/\mathfrak{Q}\mathbb V_{gh}^\psi)&\longrightarrow H^1_{\underline{\bal}}(K_v,\mathbb V_{gh,\mathfrak Q}^{\psi}),\\
H^1_{\underline{\bal}}(K_v,\mathbb A_{gh,\mathfrak Q}^{\psi})&\longrightarrow H^1_{\bal}(K_v,\mathbb A_{gh}^{\psi}[\mathfrak{Q}])
\end{align*}
with finite kernel and cokernel, of order bounded by constants depending only on $[S_{\frk{Q}}:\Lambda_{\hf}/\frk{Q}]$.
\end{lemma}

\begin{proof}
For the primes $v\nmid p$, the same argument as in the proof of \cite[Lem.~5.3.13]{MR} applies, so it remains to consider the case $v\mid p$. Put
\begin{align*}
\mathcal{F}_{\frk{p}}^+(T_{g,h}^\psi)&=(T_g^+\otimes T_h+T_g\otimes T_h^+)(\psi_{\frk{P}}^{-1})(-1-r),\\
\mathcal{F}_{\overline{\frk{p}}}^+(T_{g,h}^\psi)&=(T_g^+\otimes T_h^+)(\psi_{\frk{P}}^{-1})(-1-r).
\end{align*}
Under the isomorphism $H^1(\Q,\mathbb V_{\hf gh}^\dag)\cong H^1(K,\Lambda_{\cl{O}}(\bm{\kappa}_{\text{ac}}^{-1})\hat{\otimes}_{\cl{O}}T_{g,h}^\psi(\kappa_{\ac}^{r_1/2}))$ coming from Shapiro's lemma, the balanced local condition  $H^1_{\rm bal}(\Q_p,\mathbb{V}_{\hf gh}^\dag)$ corresponds to
\[
H^1(K_\frk{p},\Lambda_{\cl{O}}(\bm{\kappa}_{\text{ac}}^{-1})\hat{\otimes}_{\cl{O}}\mathcal{F}_{\frk{p}}^+(T_{g,h}^\psi)(\kappa_{\ac}^{r_1/2}))
\oplus H^1(K_{\overline{\frk{p}}},\Lambda_{\cl{O}}(\bm{\kappa}_{\text{ac}}^{-1})\hat{\otimes}_{\cl{O}}\mathcal{F}_{\overline{\frk{p}}}^+(T_{g,h}^\psi)(\kappa_{\ac}^{r_1/2})).
\]
Let $A_g^-=T_g^-\otimes \bb{Q}_p/\bb{Z}_p$, and define $A_g^+$, $A_h^-$ and  $A_h^+$ similarly. Arguing as in the proof of \cite[Lem.~2.2.7]{How}, we reduce to showing that the groups
\[
H^0(K_{\infty,\frk{p}},(A_g^-\otimes A_h^-)(\psi_{\frk{P}}^{-1}\kappa_{\ac}^{r_1/2})(-1-r)),\quad
H^0(K_{\infty,\overline{\frk{p}}},(A_g^+\otimes A_h^-)(\psi_{\frk{P}}^{-1}\kappa_{\ac}^{r_1/2})(-1-r))
\]
are both finite, which follows from the fact that $\alpha_g\alpha_h\psi(\frk{p})/p^{k-1}\neq 1$ and $\beta_g\alpha_h\psi(\overline{\frk{p}})\neq 1$, and this is a consequence of the Ramanujan--Petersson conjecture since we are assuming that $p\nmid N$. Note that the other pieces in the quotient decomposition can be treated similarly. This yields the required bounds on the kernel and cokernel of the first map in the statement of the lemma, and the result for the second map follows as well by local duality.
\end{proof}

Let $\Sigma_\Lambda$ be the set of height one primes of $\Lambda_{\hf}$ consisting of $p$ and those for which either $H^2(\Q^\Sigma/\Q,\mathbb V_{\hf gh}^\dag)[\frk{Q}]$ is infinite or $H^2(\Q_p,\mathbb V_{\hf gh}^\dag)[\frk{Q}]$ is infinite. Since $H^2(\Q^\Sigma/\Q,\mathbb V_{\hf gh}^\dag)$ and $H^2(\Q_p,\mathbb V_{\hf gh}^\dag)$ are both finitely generated $\Lambda$-modules, the set $\Sigma_\Lambda$ is finite.

\begin{proposition}\label{prop:sp-Q}
For every height one prime $\frk{Q}\not\in\Sigma_\Lambda$, the maps (\ref{eq:spQ}) induce natural maps
\begin{align*}
\Sel_{\bal}(\mathbb V_{\hf gh}^{\dag})/\frk{Q}\Sel_{\bal}(\mathbb V_{\hf gh}^{\dag})&\longrightarrow \Sel_{\underline{\bal}}(\mathbb V_{\hf gh,\frk{Q}}^{\dag}),\\
\Sel_{\underline{\bal}}(\mathbb A_{\hf gh,\frk{Q}}^{\dag})&\longrightarrow \Sel_{\bal}(\mathbb A_{\hf gh}^{\dag})[\frk{Q}]
\end{align*}
with finite kernel and cokernel bounded by a constant depending only on $[S_{\frk{Q}}:\Lambda_{\hf}/\frk{Q}]$.
\end{proposition}

\begin{proof}
This follows from Lemma~\ref{lem:spQ} as in the proof of \cite[Prop.~5.3.14]{MR} (see also \cite[Lem.~2.2.8]{How} and \cite[Lem.~3.2.10]{How2}).
\end{proof}

For every height one prime $\frk{Q}\subset\Lambda_{\hf}$ as above, let $\mathfrak{m}_{\frk{Q}}=(\pi_{\frk{Q}})$ be the maximal ideal of $S_{\frk{Q}}$.

\begin{lemma}\label{rangs}
Assume that there is a perfect pairing $T_{g,h}^\psi\times T_{g,h}^\psi\rightarrow \cl{O}(1)$ such that $\langle x^\sigma,y^{c\sigma c}\rangle=\langle x,y\rangle^\sigma$ for all $x,y\in T_{g,h}^\psi$ and for all $\sigma\in G_K$, where $c$ stands for complex conjugation. The following hold:
\begin{enumerate}
\item ${\rm rank}_{\Lambda_{\hf}}\,\Sel_{\bal}(\mathbb V_{\hf gh}^{\dag})={\rm rank}_{\Lambda_{\hf}}\,X_{\bal}(\mathbb A_{\hf gh}^{\dag})$.
\item ${\rm rank}_{\Lambda_{\hf}}\,\Sel_{\mathcal F}(\mathbb V_{\hf gh}^{\dag})={\rm rank}_{\Lambda_{\hf}}\,X_{\mathcal F}(\mathbb A_{\hf gh}^{\dag})$.
\item ${\rm rank}_{\Lambda_{\hf}}\,X_{\mathcal F \cup +}(\mathbb A_{\hf gh}^{\dag}) = 1+{\rm rank}_{\Lambda_{\hf}}\,X_{\mathcal F \cap +}(\mathbb A_{\hf gh}^{\dag})$, and
\[
\Char_{\Lambda_{\hf}}(X_{\mathcal F \cup +}(\mathbb A_{\hf gh}^{\dag})_{\tors}) = \Char_{\Lambda_{\hf}}(X_{\mathcal F \cap +}(\mathbb A_{\hf gh}^{\dag})_{\rm tors}),
\]
in $\Lambda_{\hf}\otimes_{\Z_p}\Q_p$.
\end{enumerate}
\end{lemma}

\begin{proof}
For part (1), it suffices to show that for all height one primes $\mathfrak{Q}\subset\Lambda_{\hf}$ with $\frk{Q}\not\in\Sigma_\Lambda$, the modules $\Sel_{\bal}(\mathbb V_{\hf gh}^\dag)/\mathfrak{Q}\Sel_{\bal}(\mathbb V_{\hf gh}^\dag)$ and $\Sel_{\bal}(\mathbb{A}_{\hf gh}^\dag)[\mathfrak{Q}]$ have the same rank over $\Lambda_{\hf}/\frk{Q}$. Since $\Sel_{\underline{\bal}}(\mathbb V_{\hf gh,\mathfrak Q}^{\dag})$ is the $\pi_{\mathfrak Q}$-adic Tate module of $\Sel_{\underline{\bal}}(\mathbb A_{\hf gh,\mathfrak Q}^{\dag})$ (indeed, this is a consequence of \cite[Lem.~1.3.3]{How} since $\mathbb A_{\hf gh,\mathfrak Q}^{\dag}\cong \mathbb V_{\hf gh,\mathfrak Q}^{\dag}\otimes \bb{Q}_p/\bb{Z}_p$), the result thus follows from Proposition~\ref{prop:sp-Q}.

For part (2), under the isomorphism $H^1(\Q,\mathbb V_{\hf gh}^\dag)\cong H^1(K,\Lambda_{\cl{O}}(\bm{\kappa}_{\text{ac}}^{-1})\hat{\otimes}_{\cl{O}}T_{g,h}^\psi(\kappa_{\ac}^{r_1/2}))$ the $f$-unbalanced local condition $H^1_{\mathcal F}(\Q_p,\mathbb V_{\hf gh}^\dag)$ corresponds to
\[
H^1(K_\frk{p},\Lambda_{\cl{O}}(\bm{\kappa}_{\text{ac}}^{-1})\hat{\otimes}_{\cl{O}}T_{g,h}^\psi)\oplus\{0\}
\]
and hence an analogue of Lemma~\ref{lem:spQ} for the $f$-unbalanced Selmer groups follows from the finiteness of $H^0(K_{\infty,\overline{\frk{p}}},A_g\otimes A_h(\psi_{\frk{P}}^{-1}\kappa_{\ac}^{r_1/2})(-1-r))$. By the same reason as above, this yields the equality of ranks in part (2).

Finally, for the proof of part (3) we can argue similarly as in \cite[Thm.~1.2.2]{AH}. Keeping with the above notations, let $\Sel_{\underline{\mathcal F\cup+}}(\mathbb A_{\hf gh,\mathfrak{Q}}^\dag)$ and $\Sel_{\underline{\mathcal F\cap+}}(\mathbb A_{\hf gh,\mathfrak{Q}}^\dag)$ be the Selmer groups defined by the obvious analogues of (\ref{eq:propagate}), so from another application of the argument in Lemma~\ref{lem:spQ} we obtain natural maps
\begin{align*}
\Sel_{\underline{\mathcal F\cup +}}(\mathbb A_{\hf gh,\mathfrak Q}^{\dag})&\longrightarrow \Sel_{\mathcal F\cup +}(\mathbb A_{\hf gh}^{\dag})[\mathfrak{Q}],\\
\Sel_{\underline{\mathcal F\cap +}}(\mathbb A_{\hf gh,\mathfrak Q}^{\dag})&\longrightarrow \Sel_{\mathcal F\cap +}(\mathbb A_{\hf gh}^{\dag})[\mathfrak{Q}]
\end{align*}
with finite kernel and cokernel bounded by a constant depending only on $[S_{\frk{Q}}:\Lambda_{\hf}/\frk{Q}]$. Since the local condition $\mathcal{F}\cap+$ is the orthogonal complement of $\mathcal{F}\cup +$ under the local Tate pairing at $p$ induced by the self-duality of $\mathbb V_{\hf gh}^\dagger$, from \cite[Thm.~4.1.13]{MR} we obtain
\begin{equation}\label{corerank}
\Sel_{\underline{\mathcal F\cup +}}(\mathbb A_{\hf gh,\mathfrak Q}^{\dag})[\pi_\mathfrak Q^i]\cong(\Phi_\mathfrak{Q}/S_{\mathfrak{Q}})^r[\pi_\mathfrak Q^i]\oplus \Sel_{\underline{\mathcal F\cap +}}(\mathbb A_{\hf gh,\mathfrak Q}^{\dag})[\pi_\mathfrak Q^i]
\end{equation}
for all $i$, where $r$ is given (by the Greenberg--Wiles formula in \cite[Prop.~2.3.5]{MR}) by
\[
{\rm corank}_{S_\mathfrak Q}H^1(\mathbb{Q}_p,\mathbb{A}_{\hf gh,\mathfrak{Q}}^{f\cup +})-{\rm corank}_{S_\mathfrak Q}H^0(\mathbb R,\mathbb{A}_{\hf gh,\mathfrak{Q}}^\dag),
\]
so $r=5-4=1$. The proof of part (3) now follows from  (\ref{corerank}) as in \cite[Lem.~1.2.6]{AH}.
\end{proof}

\begin{remark}
The existence of the pairing in the previous lemma is not too restrictive. In particular, this holds automatically if $g$ and $h$ are non-Eisenstein.
\end{remark}

We are now ready to establish that both formulations of the Iwasawa main conjecture are equivalent.


\begin{theorem}\label{thm:equiv}
Keep the assumptions of the previous lemma and suppose $\kappa(\hf,g,h)$ is not $\Lambda_{\hf}$-torsion. Then the following are equivalent:
\begin{enumerate}
\item[(1)] ${\rm rank}_{\Lambda_{\hf}}\,\Sel_{\bal}(\mathbb V_{\hf gh}^\dag)={\rm rank}_{\Lambda_{\hf}}\,X_{\bal}(\mathbb A_{\hf gh}^\dag)=1$;
\item[(2)] ${\rm rank}_{\Lambda_{\hf}}\,\Sel_{\mathcal F}(\mathbb V_{\hf gh}^\dag)={\rm rank}_{\Lambda_{\hf}}\,X_{\mathcal F}(\mathbb A_{\hf gh}^\dag)=0$,
\end{enumerate}
and, in that case, we have $\Sel_{\bal}(\mathbb V_{\hf gh}^\dag)=\Sel_{\mathcal  F\cup +}(\mathbb V_{\hf gh}^\dag)$ and
\[
\Char_{\Lambda_{\hf}}(X_{\mathcal F}(\mathbb A_{\hf gh}^{\dag}))
\cdot\Char_{\Lambda_{\hf}}\left(\frac{\Sel_{\bal}(\mathbb V_{\hf gh}^{\dag})}{\Lambda_{\hf} \cdot\kappa(\hf,g,h)}\right)^2 = \Char_{\Lambda_{\hf}}(X_{\bal}(\mathbb A_{\hf gh}^{\dag})_{\tors})\cdot (L_p({\hf},{g},{h}))
\]
in $\Lambda_{\hf}\otimes_{\Z_p}\Q_p$. In particular, Conjecture~\ref{with} and Conjecture~\ref{without} are equivalent.
\end{theorem}

\begin{proof}
The Poitou--Tate global duality gives rise to the exact sequence
\begin{equation}\label{PT-1}
\begin{split}
0 \longrightarrow \Sel_{\mathcal F\cap +}(\mathbb V_{\hf gh}^{\dag})\longrightarrow \Sel_{\bal}(\mathbb V_{\hf gh}^{\dag}) &\overset{{{\rm res}_p}}\longrightarrow H^1(\mathbb Q_p, \mathbb V_{\hf}^{gh})\\ &\longrightarrow X_{\mathcal F \cup +}(\mathbb A_{\hf gh}^{\dag}) \longrightarrow X_{\bal}(\mathbb A_{\hf gh}^{\dag})\longrightarrow 0.
\end{split}
\end{equation}

Assume that $\Sel_{\bal}(\mathbb V_{\hf gh}^{\dag})$ and $X_{\bal}(\mathbb A_{\hf gh}^{\dag})$ have both $\Lambda_{\hf}$-rank one.  Since $H^1(\mathbb Q_p, \mathbb V_{\hf}^{gh})$ has $\Lambda_{\hf}$-rank one, from (\ref{PT-1}) and Theorem~\ref{rec-law} we see that $\Sel_{\mathcal F\cap +}(\mathbb V_{\hf gh}^{\dag})$ is $\Lambda_\hf$-torsion and $X_{\mathcal F \cup +}(\mathbb A_{\hf gh}^{\dag})$ has $\Lambda_\hf$-rank one. By Lemma~\ref{rangs}(3), it follows that $X_{\mathcal F \cap +}(\mathbb A_{\hf gh}^{\dag})$ is $\Lambda_{\hf}$-torsion, and from the exact sequence
\begin{equation}\label{PT-2}
\begin{split}
0 \longrightarrow \Sel_{\mathcal F}(\mathbb V_{\hf gh}^{\dag})\longrightarrow \Sel_{\mathcal F\cup +}(\mathbb V_{\hf gh}^{\dag}) &\overset{{{\rm res}_p}}\longrightarrow H^1(\mathbb Q_p, \mathbb V_{\hf}^{gh})\\
&\longrightarrow X_{\mathcal F}(\mathbb A_{\hf gh}^{\dag}) \longrightarrow X_{\mathcal F\cap +}(\mathbb A_{\hf gh}^{\dag})  \longrightarrow 0
\end{split}
\end{equation}
we get that $X_{\mathcal F}(\mathbb A_{\hf gh}^{\dag})$ and $\Sel_{\mathcal F}(\mathbb V_{\hf gh}^{\dag})$ are both $\Lambda_{\hf}$-torsion by Lemma~\ref{rangs}(2). This proves the implication ${\rm (1)}\Rightarrow{\rm (2)}$ in the statement of the theorem, and the converse is shown similarly. Moreover, from (\ref{PT-2}) we see that ${\rm rank}_{\Lambda_\hf}\Sel_{\mathcal F\cup +}(\mathbb V_{\hf gh}^{\dag})=1$, and hence the quotient $\Sel_{\mathcal F\cup +}(\mathbb V_{\hf gh}^{\dag})/\Sel_{\rm bal}(\mathbb V_{\hf gh}^{\dag})$ is a torsion $\Lambda_{\hf}$-module injecting into $H^1(\mathbb Q_p,\mathbb V^{f\cup+}_{\hf gh}/\mathscr{F}^2\mathbb V_{\hf gh}^\dag)$; since this is $\Lambda_\hf$-torsion free by Proposition~\ref{perrin}, it follows that
\begin{equation}\label{eq:f+}
\Sel_{\bal}(\mathbb V_{\hf gh}^\dag)=\Sel_{\mathcal  F\cup +}(\mathbb V_{\hf gh}^\dag).
\end{equation}

Now suppose that either (1) or (2) in the statement of theorem holds. Since $\bar{\rho}_{\hf}$ is absolutely irreducible by our hypotheses, the module $H^1(\Q^\Sigma/\Q,\mathbb V_{\hf gh}^\dag)$ is $\Lambda_{\hf}$-torsion free by \cite[\S{1.3.3}]{PR-book}. Being $\Lambda_{\hf}$-torsion, it follows that the module $\Sel_{\mathcal F\cap +}(\mathbb V_{\hf gh}^\dag)$ vanishes, and therefore from (\ref{PT-1}) we deduce the exact sequence
\begin{equation}\label{PT-2a}
0 \longrightarrow \frac{\Sel_{\bal}(\mathbb V_{\hf gh}^\dag)}{\Lambda_{\hf} \cdot\kappa(\hf,g,h)} \longrightarrow \frac{H^1(\mathbb Q_p, \mathbb V_{\hf}^{gh})}{\Lambda_{\hf} \cdot p_{\hf*}({\rm res}_p(\kappa(\hf,g,h)))} \longrightarrow {\rm coker}({\rm res}_p) \longrightarrow 0.
\end{equation}
Together with Theorem~\ref{rec-law} it follows that
\begin{equation}\label{PT-2b}
\Char_{\Lambda_{\hf}} \left(\frac{\Sel_{\bal}(\mathbb V_{\hf gh}^{\dag})}{\Lambda_{\hf}\cdot\kappa(\hf,g,h)} \right)\cdot \Char_{\Lambda_{\hf}}({\rm coker}({\rm res}_p))=(\mathscr{L}_p^\xi(\underline{\breve{\hf}},\underline{\breve{g}},\underline{\breve{h}})).
\end{equation}
On the other hand, in light of (\ref{eq:f+}), from (\ref{PT-1}) and (\ref{PT-2}) we deduce exact sequences
\begin{align*}
0\longrightarrow {\rm coker}({\rm res}_p)&\longrightarrow X_{\mathcal F\cup +}(\mathbb A_{\hf gh}^{\dag}) \longrightarrow X_{\bal}(\mathbb A_{\hf gh}^{\dag})  \longrightarrow 0,\\
0\longrightarrow {\rm coker}({\rm res}_p)&\longrightarrow X_{\mathcal F}(\mathbb A_{\hf gh}^{\dag}) \longrightarrow X_{\mathcal F\cap +}(\mathbb A_{\hf gh}^{\dag})  \longrightarrow 0.
\end{align*}
Taking characteristic ideals, these imply
\begin{equation}\label{eq:char}
\begin{aligned}
\Char_{\Lambda_{\hf}}(X_{\mathcal F}(\mathbb A_{\hf gh}^{\dag})&=\Char_{\Lambda_{\hf}}(X_{\mathcal F\cap +}(\mathbb A_{\hf gh}^{\dag}))\cdot\Char_{\Lambda_\hf}({\rm coker}({\rm res}_p))\\
&=\Char_{\Lambda_\hf}(X_{\mathcal F\cup +}(\mathbb A_{\hf gh}^{\dag})_{\rm tors})\cdot\Char_{\Lambda_\hf}({\rm coker}({\rm res}_p))\\
&=\Char_{\Lambda_\hf}(X_{\bal}(\mathbb A_{\hf gh}^{\dag})_{\rm tors})\cdot\Char_{\Lambda_\hf}({\rm coker}({\rm res}_p))^2,
\end{aligned}
\end{equation}
using Lemma~\ref{rangs}(3) for the second equality. Multiplying (\ref{eq:char}) by the square of a generator of the characteristic ideal of $\Sel_{\bal}(\mathbb V_{\hf gh}^{\dag})/\Lambda_{\hf} \cdot\kappa(\hf,g,h)$ and using (\ref{PT-2b}), the result follows.
\end{proof}

\section{Anticyclotomic Euler systems}\label{sec:applications}

In this section we highlight results from the recent work of Jetchev--Nekov\'{a}\v{r}--Skinner \cite{JNS}, where a general theory of Euler systems germane to \cite{Rub} is developed in the anticyclotomic setting. 

\subsection{The general theory}\label{sec:ES}

Let $K$ be an imaginary quadratic field and let $p$ be an odd prime. If $\frk{n}$ is an integral prime ideal of $K$, we denote by $K(\frk{n})^\circ$ the ray class field of conductor $\frk{n}$; as in the previous sections,
we write $K(\frk{n})$ for the maximal $p$-subextension in $K(\frk{n})^\circ$. For any positive integer $n$, we denote by $K[n]$ the maximal $p$-subextension in the ring class field of $K$ of conductor $n$.
We denote by $K_\infty$ the anticyclotomic $\bb{Z}_p$-extension of $K$.

Let $E$ be a finite extension of $\mathbb Q_p$ with ring of integers $\mathcal O$ and maximal ideal $\mathfrak m$. Let $T$ be a free $\cl{O}$-module of finite rank endowed with a continuous $G_K$-action  unramified outside a finite set of primes, and let $V=T\otimes_{\cl{O}}E$. Assume that there exists a non-degenerate symmetric $\mathcal O$-bilinear pairing
\[
\langle \, , \, \rangle \, : \, T \times T \longrightarrow \mathcal O(1)
\]
such that $\langle x^\sigma,y^{c\sigma c}\rangle=\langle x,y\rangle^\sigma$ for all $x,y\in T$ and $\sigma\in G_K$, where $c$ is complex conjugation. Thus $V^c\simeq V^{\vee}(1)$, where $V^c$ denotes the representation $V$ with the $G_K$-action twisted by $c$, and, if the above pairing is perfect, we also have $T^c\simeq T^\vee(1)$. We also define the $G_K$-module $A=V/T$.

If $L$ is a finite extension of $K$ and $v$ is a finite place of $L$, we write $\overline{v}=v^c$. Then, the pairing above induces a local pairing
\[
H^1(L_v,V)\times H^1(L_{\overline{v}},V)\longrightarrow E,
\]
and similarly replacing $V$ by $T$ and $E$ by $\cl{O}$. The pair of compatible maps $G_{L_{v}}\rightarrow G_{L_{\overline{v}}}$ and $V\rightarrow V^c$ defined by $\sigma\mapsto c\sigma c$ and $w\mapsto w$, respectively, induces an isomorphism $H^1(L_{\overline{v}},V)\cong H^1(L_v,V^c)\cong H^1(L_v, V^\vee(1))$ whereby the above local pairing is just the natural cup-product pairing.

For the results we shall discuss, we consider two different types of ``big image'' hypotheses, (HW) for the weaker ones, and (HS) for the stronger ones.

\begin{hyp}[HW]\hfill
\begin{enumerate}
\item[{\rm (1)}] $V$ is absolutely irreducible as a $G_K$-representation.
\item[{\rm (2)}] There exists an element $\sigma_0 \in \Gal(\bar K/K(1)^\circ K(\mu_{p^{\infty}}))$ such that the $E$-dimension of $V/(\sigma_0-1)V$ is one.
\end{enumerate}
\end{hyp}

\begin{hyp}[HS]\hfill
\begin{enumerate}
\item[{\rm (1')}] The residual representation $\bar T = T/\mathfrak m T$ is absolutely irreducible.
\item[{\rm (2')}] There exists an element $\sigma_0 \in \Gal(\bar K/K(p^\infty)^\circ)$ such that $T/(\sigma_0-1)T \simeq \mathcal O$ is a free $\mathcal O$-module of rank one.
\item[{\rm (3')}] There exists an element $\tau_0 \in G_K$ such that $\tau_0-1$ acts on $T$ as multiplication by a unit $a_{\tau_0}\in\mathcal{O}^\times$ with $a_{\tau_0}-1\in\mathcal O^{\times}$.
\item[{\rm (4')}] The above pairing $T \times T \longrightarrow \mathcal O(1)$ is perfect.
\end{enumerate}
\end{hyp}

For each prime $\frk{p}$ of $K$ above $p$, choose a $G_{K_\frk{p}}$-stable $\cl{O}$-submodule $\cl{F}_\frk{p}^+(T)$ of $T$, and let $\cl{F}_\frk{p}^-(T)=T/\cl{F}_\frk{p}^+(T)$. We also define $\cl{F}_\frk{p}^+(V)=\cl{F}_\frk{p}^+(T)\otimes_{\cl{O}}E\subseteq V$ and $\cl{F}_\frk{p}^-(V)=V/\cl{F}_\frk{p}^+(V)$. 
Let $L$ be a finite extension of $K$. For each place $v$ of $L$, we define a local condition
$$
H^1_{\Gr}(L_v,V)=\begin{cases} \ker\left(H^1(L_v,V)\rightarrow H^1(L_v^{\text{nr}},V)\right) & \text{if } v\nmid p, \\[0.5em]
\ker\left(H^1(L_v,V)\rightarrow H^1(L_v,\cl{F}_\frk{p}^-(V))\right) & \text{if } v\mid \frk{p}\text{ for some }\frk{p}\mid p. \end{cases}
$$
We define the \emph{Greenberg Selmer group}
\[ 
\Sel_{\Gr}(L,V)=\ker \Big(
H^1(L, V) \rightarrow \prod_{v} H^1(L_v,V)/H^1_{\Gr}(L_v,V)\Big),
\]
where the product is over all finite places of $L$.

We also define local conditions for $T$ and $A$ by propagation of the local conditions for $V$, i.e., for each place $v$ of $L$, we define
\begin{itemize}
\item $H^1_{\Gr}(L_v,T)$ as the preimage of $H^1_{\Gr}(L_v,V)$ by the map $H^1(L_v,T)\rightarrow H^1(L_v,V)$, and
\item $H^1_{\Gr}(L_v,A)$ as the image of $H^1_{\Gr}(L_v,V)$ by the map $H^1(L_v,V)\rightarrow H^1(L_v,A)$,
\end{itemize}
and use these to define the Selmer groups $\Sel_{\Gr}(L,T)$ and $\Sel_{\Gr}(L,A)$ as above. Finally, for each positive integer $n$, we also put
\begin{align*}
\Sel_{\Gr}(K[np^\infty],T)&=\varprojlim_r \Sel_{\Gr}(K[np^r],T)\quad{\rm and}\quad
\Sel_{\Gr}(K[np^\infty],A)=\varinjlim_r \Sel_{\Gr}(K[np^r],A),
\end{align*}
where the limits are with respect to the corestriction and restriction maps, respectively, and we define
\[
X_{\Gr}(K[np^\infty],A)=\Hom_{\mathrm{cont}}(\Sel_{\Gr}(K[np^\infty],A),\bb{Q}_p/\bb{Z}_p).
\]

Let $\mathcal{N}$ be an ideal of $K$ divisible by $p$ and all the primes at which $T$ is ramified, and let $\cl{S}$ be the set of all squarefree products of primes of $\bb{Q}$ which \emph{split} in $K$ and are coprime to $\cl{N}$. 

\begin{definition}
A \textit{``split'' anticyclotomic Euler system} for $(T,\{\cl{F}_\frk{p}^+(T)\}_{\frk{p}\mid p},\cl{N})$ is a collection of classes
$$
\bm{\kappa}=\left\{\kappa_n\in \Sel_{\Gr}(K[n],T)\;\colon\; n\in \cl{S}\right\}
$$
such that, whenever $q$ is a rational prime and $n,nq\in\cl{S}$,
\begin{equation}\label{eq:norm}
\cor_{K[nq]/K[n]}(\kappa_{nq})= P_\frk{q}(\Fr_\frk{q}^{-1})\,\kappa_n,
\end{equation}
where $\frk{q}$ is any of the primes of $K$ above $q$ and
$P_\frk{q}(X)=\det(1-\Fr_\frk{q}^{-1}X\vert T^\vee(1))$.

Similarly, a \textit{``split'' $\Lambda$-adic anticyclotomic Euler system} for $(T,\{\cl{F}_\frk{p}^+(T)\}_{\frk{p}\mid p},\cl{N})$ is a collection of classes
$$
\bm{\kappa}_\infty=\left\{\kappa_{n,\infty}\in \Sel_{\Gr}(K[np^\infty],T)\;\colon\; n\in \cl{S}\right\}
$$
satisfying the previous norm relations. In this case, the classes
\[
\kappa_n={\rm pr}_{K[n]}(\kappa_{n,\infty})\in \Sel_{\Gr}(K[n],T)
\]
form an anticyclotomic Euler system in the previous sense, and we say that the Euler system $\bm{\kappa}=\{\kappa_n\}_n$ extends along the anticyclotomic $\bb{Z}_p$-extension.
\end{definition}

A ($\Lambda$-adic) anticyclotomic Euler system for $(T,\{\cl{F}_\frk{p}^+(T)\}_{\frk{p}\mid p})$ is just a ($\Lambda$-adic) anticyclotomic Euler system for $(T,\{\cl{F}_\frk{p}^+(T)\}_{\frk{p}\mid p},\cl{N})$ for some $\cl{N}$ as above. We shall usually drop $\{\cl{F}_\frk{p}^+(T)\}_{\frk{p}\mid p}$ if there is no risk of confusion.


If $\bm{\kappa}$ is an anticyclotomic Euler system for $T$, we define
\[ \kappa_0 := \cor_{K[1]/K}(\kappa_1) \in \Sel_{\Gr}(K,T). \]
If it extends along the anticyclotomic $\bb{Z}_p$-extension, we similarly define
\[ \kappa_\infty := \cor_{K[1]/K}(\kappa_{1,\infty}) \in \Sel_{\Gr}(K_\infty,T), \]
where $\bm{\kappa}_\infty=\{\kappa_{n,\infty}\}$ is the $\Lambda$-adic anticyclotomic Euler system extending $\bm{\kappa}$.

When we have an Euler system as above, we will be interested in ensuring that the following orthogonality hypothesis holds.

\begin{hyp}[HO]
For all $n\in\cl{S}$ and for all places $v$ of $K[n]$ above $p$, the local conditions $H^1_{\Gr}(K[n]_v,V)$ and $H^1_{\Gr}(K[n]_{\overline{v}},V)$ are orthogonal complements under the local pairing
\[
H^1(K[n]_v,V) \times H^1(K[n]_{\overline{v}},V) \longrightarrow  E.
\]
\end{hyp}

\begin{remark}
The condition in hypothesis~(HO) holds automatically for all places away from $p$, by \cite[Prop. 1.4.2]{Rub}. Observe also that if (HO) holds, then for all $n\in\cl{S}$ and for all places $v$ of $K[n]$ the local conditions $H^1_{\Gr}(K[n]_v,T)$ and $H^1_{\Gr}(K[n]_{\overline{v}},T)$ are also orthogonal complements under the local pairing
\[
H^1(K[n]_v,T) \times H^1(K[n]_{\overline{v}},T) \longrightarrow  \cl{O},
\]
as follows easily from the definitions using \cite[Prop. B.2.4]{Rub} and the commutative diagram
\[
\xymatrix{
H^1(K[n]_v,T) \times  H^1(K[n]_{\overline{v}},T)\ar[d]\ar[r]&\cl{O}\ar[d]\\
H^1(K[n]_v,V) \times  H^1(K[n]_{\overline{v}},V)\ar[r]&E.
}
\]
\end{remark}

We assume in the rest of this subsection that hypothesis (HO) holds for our choice of local conditions at $p$.


\begin{theorem}[\cite{JNS}]\label{primer-res}
Assume that $p$ splits in $K$ and that Hypothesis~{\rm (HW)} is satisfied, and let $\bm{\kappa}=\{\kappa_n\}_n$ be an anticyclotomic Euler system for $T$ which extends along the anticyclotomic $\Z_p$-extension.  If $\kappa_0 \neq 0$, then the Selmer group $\Sel_{\Gr}(K,T)$ has $\cl{O}$-rank one.
\end{theorem}

\begin{remark}\label{rem:gamma}
One can replace the assumptions that $p$ splits in $K$ and the Euler system extends along the anticyclotomic $\bb{Z}_p$-extension by the assumption that there exists an element $\gamma \in G_K$ fixing the extension $K(1)^\circ (\mu_{p^\infty},(\cl{O}_K^\times)^{1/p^\infty})$ and such that $\gamma-1$ acts invertibly on $V$.
\end{remark}

Under the stronger Hypothesis~(HS), granted the non-triviality of a $\Lambda$-adic anticyclotomic Euler system, the results of \cite{JNS} yield a divisibility towards a corresponding Iwasawa main conjecture.  

\begin{theorem}[\cite{JNS}]\label{segon-res}
Assume that $p$ splits in $K$ and that Hypothesis~{\rm (HS)} is satisfied, and let $\bm{\kappa}$ be a $\Lambda$-adic anticyclotomic Euler system for $T$.
\begin{enumerate}
\item[{\rm (a)}] If $\kappa_0\neq 0$, then $\Sel_{\Gr}(K,A)$ has $\mathcal{O}$-corank one, $\Sel_{\Gr}(K,T)$ has $\cl{O}$-rank one, and
\[
{\rm length}_{\mathcal{O}}(\Sel_{\Gr}(K,A)_{/{\rm div}})\leq 2\;{\rm length}_{\mathcal O}\left(\frac{\Sel_{\Gr}(K,T)}{\mathcal O \cdot \kappa_0}\right),
\]
where $(-)_{/{\rm div}}$ denotes the quotient of $(-)$ by its  maximal divisible submodule.
\item[{\rm (b)}] If $\kappa_\infty$ is not $\Lambda_{\ac}$-torsion, then $X_{\Gr}(K_\infty,A)$ and $\Sel_{\Gr}(K_\infty,T)$ have both $\Lambda_{\ac}$-rank one, and
\[
\Char_{\Lambda_{\ac}}(X_{\Gr}(K_\infty,A)_{\tors})\supset\Char_{\Lambda_{\ac}} \left( \frac{\Sel_{\Gr}(K_\infty,T)}{\Lambda_{\ac} \cdot \kappa_\infty} \right)^2,
\]
where $(-)_{\rm tors}$ denotes the maximal $\Lambda_{\ac}$-torsion submodule of $(-)$.
\end{enumerate}
\end{theorem}

\subsection{Big image results}\label{subsec:verify}

We now give conditions under which the hypotheses in the general results of $\S\ref{sec:ES}$ are verified in our setting. To that end, we shall build on \cite{Loe}.

As before, let $K/\Q$ be an imaginary quadratic field of discriminant $-D$, let $(g,h)$ be a pair of newforms of weights $(l, m)$ of the same parity, levels $(N_g,N_h)$ and characters $(\chi_g,\chi_h)$, and let $\psi$ be a Gr\"ossencharacter of $K$ of infinity type $(1-k,0)$ for some positive even integer $k$ and of conductor $\frk{f}$. We denote by $\chi$ the unique Dirichlet character modulo $N_{K/\mathbb Q}(\mathfrak f)$ such that $\psi((n))=n^{k-1} \chi(n)$ for integers $n$ coprime to $N_{K/\mathbb Q}(\mathfrak f)$, and we assume that $\chi\varepsilon_K\chi_g\chi_h=1$.

We now make the further assumptions that:
\begin{itemize}
\item{} neither $g$ nor $h$ are of CM type,
\item{} $g$ is not Galois-conjugate to a twist of $h$.
\end{itemize}

As in \cite[\S3.1]{Loe}, we define the open subgroups $H_g$ and $H_h$ of $G_\bb{Q}$, the quaternion algebras $B_g$ and $B_h$, and the algebraic groups $G_g$ and $G_h$, and put
\[
B=B_g\times B_h,\quad G=G_g\times_{\bb{G}_m}G_h.
\]
We define $H$ to be the intersection of $H_g$, $H_h$ and $G_{K(\frk{f})^\circ}$. (Note that in \emph{loc.\,cit.} $H$ is defined to be the intersection of $H_g$ and $H_h$, so our $H$ might be a finite index subgroup of his $H$, but this will not affect the results that follow.) 
We have an adelic representation $\tilde{\rho}_{g,h}:H\rightarrow G(\hat{\bb{Q}})$, and representations
$$
\tilde{\rho}_{g,h,p}: H\longrightarrow G(\bb{Q}_p)
$$
for every rational prime $p$, and, by \cite[Thm. 3.2.2]{Loe}, $\tilde{\rho}_{g,h,p}(H)=G(\bb{Z}_p)$ for all but finitely many $p$.

\begin{remark}
Note that the representations studied in \cite{Loe} are the dual to the ones studied in this paper,
but as pointed out in \cite[Rmk.~2.1.2]{Loe}, this difference is unimportant when considering the image.
\end{remark}

Let $L$ be a finite extension of $K$ containing the Fourier coefficients of $g$ and $h$ and the image of $\psi$. Let $\frk{P}$ be a prime of $L$ above some rational prime $p$ and let $E=L_\frk{P}$.

\begin{definition}\label{def:good}
We say that the prime $\frk{P}$ is \textit{good} if the following conditions hold:
\begin{itemize}
\item $p\geq 7$;
\item $p$ is unramified in $B$;
\item $p$ is coprime to $\frk{f}$, $N_g$ and $N_h$;
\item $\tilde{\rho}_{g,h,p}(H)=G(\bb{Z}_p)$;
\item $E=\bb{Q}_p$.
\end{itemize}
\end{definition}

\begin{remark}
Observe that all but the last condition exclude only finitely many primes. The last condition could be somewhat relaxed in some cases, and will be used largely for simplicity. Note also that the above set of conditions holds for a set of primes of positive density.
\end{remark}

From now on, we assume that both $g$ and $h$ are ordinary, non-Eisenstein, and distinguished with respect to $\mathfrak{P}$.

\begin{lemma}\label{lemma:big-image}
Assume that there is at least one prime which divides $D$ but not $N_g$ and one prime which divides $D$ but not $N_h$. Then, if $\frk{P}$ is a good prime,
$$
(\rho_{g,\frk{P}}\times\rho_{h,\frk{P}})(H\cap G_{K(p^\infty)^\circ})=\mathrm{SL}_2(\bb{Z}_p)\times \mathrm{SL}_2(\bb{Z}_p).
$$
\end{lemma}
\begin{proof}
Let $\bb{Q}(\rho_g)$ and $\bb{Q}(\rho_h)$ be the Galois extensions of $\bb{Q}$ cut out by the representations $\rho_g$ and $\rho_h$ attached to $g$ and $h$, respectively. These extensions are unramified outside $pN_g$ and $pN_h$, respectively. Therefore, the condition on $D$ implies that $K\cap \bb{Q}(\rho_g)=\bb{Q}$ and $K\cap \bb{Q}(\rho_h)=\bb{Q}$. Moreover, since any Galois extension of $\bb{Q}$ contained in $K_\infty$ must itself contain $K$, we also have $K_\infty\cap\bb{Q}(\rho_g)=\bb{Q}$ and $K_\infty\cap\bb{Q}(\rho_h)=\bb{Q}$.

The conditions on $\frk{P}$ imply that
$$
(\rho_{g,\frk{P}}\times\rho_{h,\frk{P}})(H\cap G_{\bb{Q}(\mu_{p^\infty})})=\mathrm{SL}_2(\bb{Z}_p)\times \mathrm{SL}_2(\bb{Z}_p),
$$
and, from the remarks in the previous paragraph, it follows that
$$
(\rho_{g,\frk{P}}\times\rho_{h,\frk{P}})(H\cap G_{K_\infty(\mu_{p^\infty})})=\mathrm{SL}_2(\bb{Z}_p)\times \mathrm{SL}_2(\bb{Z}_p).
$$
Finally, since $H\cap G_{K(p^\infty)^\circ}$ is a normal subgroup of $H\cap G_{K_\infty(\mu_{p^\infty})}$ of index dividing $p-1$ and there are no such subgroups in $\mathrm{SL}_2(\bb{Z}_p)\times \mathrm{SL}_2(\bb{Z}_p)$, the lemma follows.
\end{proof}

Now we are able to give conditions under which the results of \cite{JNS} can be applied to our setting, i.e., to the representation $T_{g,h}^\psi$ defined above.

\begin{proposition}\label{prop:suff}
Assume that there is at least one prime which divides $D$ but not $N_g$ and one prime which divides $D$ but not $N_h$. Let $\frk{P}$ be a good prime. Suppose that there exists $\sigma\in G_{K(p^\infty)^\circ}$ such that $\psi_{\frk{P}}(\sigma)\neq \psi_{\frk{P}}^c(\sigma)$ modulo $p$. Then, hypotheses (HS) hold for $T_{g,h}^\psi$.
\end{proposition}
\begin{proof}
Since $\psi_{\frk{P}}$
is trivial when restricted to $H\cap G_{K(p^\infty)^\circ}$, condition (1') follows easily from the previous lemma.

To prove condition (2'), we closely follow the proof of \cite[Prop. 4.2.1]{Loe}. Write $\chi_g(\sigma)$ and $\chi_h(\sigma)$ for the images of $\sigma$ by $\chi_g$ and $\chi_h$ via the natural identifications $\Gal(\bb{Q}(\mu_{N_g})/\bb{Q})\cong (\bb{Z}/N_g\bb{Z})^\times$ and $\Gal(\bb{Q}(\mu_{N_h})/\bb{Q})\cong (\bb{Z}/N_h\bb{Z})^\times$. Then, by the previous lemma, the image of $\sigma H\cap G_{K(p^\infty)^\circ}$ under $\rho_{g,\frk{P}}\times\rho_{h,\frk{P}}$ contains all the elements of the form
$$
\left(\begin{pmatrix} x & 0 \\ 0 & x^{-1}\chi_g(\sigma)\end{pmatrix},\begin{pmatrix} y & 0 \\ 0 & y^{-1}\chi_h(\sigma)\end{pmatrix}\right),\quad x,y\in \bb{Z}_p^\times.
$$
Now choose $x\in\bb{Z}_p^\times$ such that $x^{-2}\chi_g(\sigma)\neq 1 \pmod p$ and $x^2\chi_h(\sigma)\psi_{\frk{P}}(\sigma)^{-2}\neq 1 \pmod p$, which is possible since $p\geq 7$, and let $y=x^{-1}\psi_{\frk{P}}(\sigma)$. Choose $\sigma_0\in \sigma H\cap G_{K(p^\infty)^\circ}$ whose image under $\rho_{g,\frk{P}}\times\rho_{h,\frk{P}}$ is given by the element above, with the choices of $x$ and $y$ which we have just specified. Then, the eigenvalues of $\sigma_0$ acting on $T_{g,h}^\psi$ are 1, $x^{-2}\chi_g(\sigma)$, $x^2\chi_h(\sigma)\psi_{\frk{P}}(\sigma)^{-2}$ and $\psi_{\frk{P}}^c(\sigma)\psi_{\frk{P}}(\sigma)^{-1}$, which proves condition (2').

To check condition (3'), we can argue as in \cite[Rmk. 11.1.3]{KLZ}. By the previous lemma, we can find an element $\tau_0\in H\cap G_{K(p^\infty)^\circ}$ such that
$$
(\rho_{g,\frk{P}}\times\rho_{h,\frk{P}})(\tau_0)=\left(\begin{pmatrix} -1 & 0 \\ 0 & -1\end{pmatrix},\begin{pmatrix} 1 & 0 \\ 0 & 1\end{pmatrix}\right),
$$
so $\tau_0$ acts on $T_{g,h}^\psi$ as multiplication by $-1$.

Finally, condition (4') follows from the assumption that $g$ and $h$ are non-Eisenstein and $p$-distinguished.
\end{proof}


\begin{remark}\label{rem:HW}
If we are just interested in ensuring that hypotheses (HW) hold for $T_{g,h}^\psi$, we can relax some of the assumptions above. For example,
we do not need to require $g$ and $h$ to be non-Eisenstein, and
we can require that there exist $\sigma\in G_{K(1)^\circ(\mu_{p^\infty})}$ such that $\psi_{\frk{P}}(\sigma)\neq \psi_{\frk{P}}^c(\sigma)$, without requiring this inequality to hold modulo $p$.
\end{remark}

\section{Proof of Theorems \ref{thmB}, \ref{thmC}, and \ref{thmD}}\label{sec:ABC}


Let the setting be as in the Introduction. In particular, $g\in S_l(N_g,\chi_g)$ and $h\in S_m(N_h,\chi_h)$ are newforms of weights $l \geq m \geq 2$ of the same parity, $K/\Q$ is an imaginary quadratic field of discriminant $-D<0$,
 $\psi$ is a Gr\"ossencharacter for $K$ of infinity type $(1-k,0)$ for some even integer $k\geq 2$, and we consider the $G_K$-representation
\[
V_{g,h}^\psi=V_g\otimes_EV_h(\psi_{\mathfrak{P}}^{-1})(1-c),
\]
where $c=(k+l+m-2)/2$.

\begin{lemma}\label{lem:BK}
The Bloch--Kato Selmer group of $V_{g,h}^{\psi}$ is given by
\[
\Sel(K,V_{g,h}^{\psi})\cong
\begin{cases}
\Sel_{\bal}(K,V_{g,h}^{\psi}) &\textrm{if $l-m<k<l+m$,}\\[0.2em]
\Sel_{\mathcal F}(K,V_{g,h}^{\psi}) &\textrm{if $k\geq l+m$.}
\end{cases}
\]
\end{lemma}

\begin{proof}
Note that by Shapiro's lemma $H^1(K,V_{g,h}^\psi)\cong H^1(\mathbb Q, V_{fgh})$, where $f=\theta_\psi$ is the theta series of $\psi$ and $V_{fgh}$ is the specialization of the big Galois representation $\mathbb V_{\hf\hg\hh}^\dag$ in (\ref{eq:big-SD}) to weights $(k,l,m)$.
One immediately checks that the Hodge--Tate weights of the $G_{\mathbb{Q}_p}$-subrepresentation $\mathscr{F}^2V_{fgh}\subset V_{fgh}$ (resp. $V_{fgh}^f\subset V_{fgh}$) are all $<0$ (with the $p$-adic cyclotomic character $\epsilon_{\rm cyc}$ having Hodge--Tate weight $-1$) if and only if $l-m<k<l+m$ (resp. $k\geq l+m$). The result follows.
\end{proof}


Here we collect a set of hypotheses for our later reference. For any nonzero $m\in\Z$, ${\rm prime}(m)$ denotes the set of primes that divide $m$, and ${\rm prime}^c(m)$ its complement.

\begin{hypotheses}\label{hyp:h1-h9}
\hfill
\begin{itemize}
\item[{\rm (h1)}] $g$ and $h$ are ordinary at $p$, non-Eisenstein, and $p$-distinguished.
\item[{\rm (h2)}] $p$ splits in $K$,
\item[{\rm (h3)}] $p$ does not divide the class number of $K$,
\item[{\rm (h4)}] $\psi_\frk{P}\vert_{G_{K(p^\infty)^\circ}}\neq \psi_\frk{P}^c\vert_{G_{K(p^\infty)^\circ}}$ modulo $p$,
\item[{\rm (h5)}] neither $g$ nor $h$ are of CM type,
\item[{\rm (h6)}] $g$ is not Galois-conjugate to a twist of $h$.
\item[{\rm (h7)}] ${\rm prime}(D)\cap{\rm prime}^c(N_g)\neq\emptyset$ and ${\rm prime}(D)\cap{\rm prime}^c(N_h)\neq\emptyset$,
\item[{\rm (h8)}] $\mathfrak{P}$ is a good prime in the sense of Definition~\ref{def:good}.
\end{itemize}
\end{hypotheses}



\subsection{Proof of Theorem~\ref{thmB}}

Let $\kapinftyone\in H^1_{\rm Iw}(K[p^\infty],T_{g,h}^\psi)$ be the Iwasawa cohomology class of conductor $n=1$ from Theorem~\ref{principal:wild}, and set
\begin{equation}\label{eq:base-class}
\kap=\kapone\in H^1(K,T_{g,h}^\psi),
\end{equation}
where $\kapone={\rm pr}_{K}(\kapinftyone)$.

If $l-m<k<l+m$, the next result recovers Theorem~\ref{thmB} in the Introduction. Note however, that the result does not require these inequalities to hold.

\begin{theorem}\label{thm:A}
Assume hypotheses (h1)--(h8). Then the following implication holds:
\[
\kap\neq 0\quad\Longrightarrow\quad{\rm dim}_E\,\Sel_{\bal}(K,\Vrep)=1.
\]
In particular, if $l-m<k<l+m$ and $\kap\neq 0$ then the Bloch--Kato Selmer group $\Sel(K,\Vrep)$ is one-dimensional.
\end{theorem}

\begin{proof}
By Proposition~\ref{prop:in-Selinfty}, the classes $\kapn:={\rm pr}_{K[n]}(\kapinftyn)$ land in $\Sel_{\bal}(K[n],T_{g,h}^\psi)$, and by Theorem~\ref{principal:wild} they form an anticyclotomic Euler system for $V_{g,h}^\psi$. Therefore, the result follows from Theorem~\ref{primer-res} and Proposition~\ref{prop:suff}.
\end{proof}

\begin{remark}
If $k=2$ and $l=m\geq 2$, working with the classes $\kapn$ from Theorem~\ref{principal:tame}, rather than those from Theorem~\ref{principal:wild} as above, hypotheses (h2)-(h3) in Theorem~\ref{thm:A} can be replaced by the assumption that there exists an element $\gamma \in G_K$ satisfying the conditions in Remark~\ref{rem:gamma}. Further, (h1) and (h4) can be relaxed as discussed in Remark~\ref{rem:HW}.
\end{remark}



\subsection{Proof of Theorem~\ref{thmC}}

Recall that $\theta_\psi\in S_k(N_\psi,\chi\varepsilon_K)$ is the theta series attached to $\psi$, and put $N=\lcm(N_\psi,N_g,N_h)$.

The next theorem, establishing cases of the Bloch--Kato conjecture for $V_{g,h}^\psi$ 
in analytic rank zero, recovers Theorem~\ref{thmC} in the Introduction.

\begin{theorem}\label{thm:thmB}
Assume hypotheses (h1)--(h8), and in addition that:
\begin{itemize}
\item $\varepsilon_\ell(\theta_\psi,g,h)=+1$ for all primes $\ell\mid N$,
\item ${\rm gcd}(N_\psi,N_g,N_h)$ is squarefree.
\end{itemize}
If $k\geq l+m$ then the following implication holds:
\[
L(\Vrep,0)\neq 0\quad\Longrightarrow\quad \Sel(K,\Vrep)=0.
\]
\end{theorem}

\begin{proof}
We continue to denote by $\kap$ the image of the class in (\ref{eq:base-class}) under the isomorphism
\[
H^1(K,V_{g,h}^{\psi})\cong H^1(\mathbb Q,V_{fgh})
\]
coming from Shapiro's lemma. If $k\geq l+m$, the central value $L(\Vrep,0)$ is in the range of interpolation of the triple product $p$-adic $L$-function of Theorem~\ref{thm:hsieh}, and so by Proposition~\ref{perrin} and Theorem~\ref{rec-law} its non-vanishing implies that the image of $\kap$ under the natural map
\[
{\rm res}_p:\Sel_{\bal}(\mathbb Q,V_{fgh})\longrightarrow H^1(\mathbb Q_p,V_f^{gh})
\]
is nonzero. In particular, $\kap\neq 0$, and therefore by Theorem~\ref{thm:A} the balanced Selmer group $\Sel_{\bal}(K,V_{g,h}^{\psi})=\Sel_{\bal}(\mathbb Q,V_{fgh})$ is one-dimensional.

From the exact sequence
\begin{align*}
0\longrightarrow\Sel_{\mathcal F\cap +}(\mathbb Q,V_{fgh})\longrightarrow\Sel_{\bal}(\mathbb Q,V_{fgh})&\overset{{\rm res}_p}\longrightarrow H^1(\mathbb Q_p,V_f^{gh})\\
&\longrightarrow\Sel_{\mathcal F\cup +}(\mathbb Q,V_{fgh})^\vee\longrightarrow\Sel_{\bal}(\mathbb Q,V_{fgh})^\vee\longrightarrow 0
\end{align*}
coming from global duality (adopting notations similar to those in Theorem~\ref{thm:equiv}), we thus see that $\Sel_{\mathcal F\cap +}(\mathbb Q,V_{fgh})=0$ and that $\Sel_{\mathcal F\cup +}(\mathbb Q,V_{fgh})=\Sel_{\bal}(\mathbb Q,V_{fgh})$.
Together with the exact sequence
\[
\Sel_{\mathcal F\cup +}(\mathbb Q,V_{fgh})\xrightarrow{{\rm res}_p}H^1(\mathbb Q_p,V_f^{gh})\longrightarrow\Sel_{\mathcal F}(\mathbb Q,V_{fgh})^\vee\longrightarrow\Sel_{\mathcal F\cap +}(\mathbb Q,V_{fgh})^\vee\longrightarrow 0,
\]
it follows that $\Sel_{\mathcal F}(\mathbb Q,V_{fgh})=0$, and combined with Lemma~\ref{lem:BK} this concludes the proof.
\end{proof}

Refining the proof of Theorem~\ref{thm:thmB}, we can further bound the size of the Bloch--Kato Selmer group for the discrete module $A_{g,h}^\psi=\Vrep/\Trep$ in terms of $L$-values. For the statement, let $\hf$ be the Hida family associated to $\psi$ as in $\S\ref{subsec:Iw}$, so that $\hf_k$ is the ordinary $p$-stabilization of $\theta_\psi$, and, keeping with the notations in Theorem~\ref{thm:hsieh}, put $\alpha_k=\psi(\overline{\frk{p}})$ and $\beta_k=\psi(\frk{p})$. Let also $\varepsilon_\ell(\theta_\psi,g,h)=\varepsilon_\ell(V_{f g h})$ denote the epsilon factor associated to $V_{fgh}\vert_{G_{\Q_\ell}}$, where $f=\theta_\psi$.

\begin{theorem}\label{thm:size-BK}
Assume hypotheses (h1)--(h8), and in addition that:
\begin{itemize}
\item $\varepsilon_\ell(\theta_\psi,g,h)=+1$ for all primes $\ell\mid N$,
\item ${\rm gcd}(N_\psi,N_g,N_h)$ is squarefree,
\item $H^1(\bb{Q}_p, T_{f}^{gh})$ is torsion-free,
\item $H^1_{\cl{L}}(\bb{Q}_p,T_{fgh})$ is torsion-free for $\cl{L}\in\{\bal,\cl{F},\cl{F}\cap +,\cl{F}\cup +\}$.
\end{itemize}
If $k\geq l+m$ and $L(\Vrep,0)\neq 0$ then the $\mathcal{O}$-module $\Sel_{\cF}(K,\Arep)$ is finite and
\[
{\rm length}_{\cO}(\Sel_{\cF}(K,\Arep))\leq 2\,v_\mathfrak{P}\left(\frac{(l-2)!(m-2)!}{(k-c-1)!}\cdot\frac{\mathcal{E}_1(\hf_k)}{\mathcal{E}(\hf_k,g,h)}\cdot \mathscr{L}_p^\xi(\underline{\breve{\hf}},\underline{\breve{g}},\underline{\breve{h}})(k)\right),
\]
where $\mathcal{E}_1(\hf_k)=\bigl(1-\frac{\beta_k}{p\alpha_k}\bigr)$, $\mathcal{E}(\hf_k,g,h)=\bigl(1-\frac{\alpha_k\alpha_{g}\alpha_{h}}{p^{c}}\bigr)\bigl(1-\frac{\beta_k\beta_{g}\alpha_{h}}{p^{c}}\bigr)\bigl(1-\frac{\beta_k\alpha_{g}\beta_{h}}{p^{c}}\bigr)\bigl(1-\frac{\beta_k\beta_{g}\beta_{h}}{p^{c}}\bigr)$, and $c=(k+l+m-2)/2$.
\end{theorem}

\begin{proof}
As in the proof of Theorem~\ref{thm:thmB}, if $k\geq l+m$ and  $L(\Vrep,0)\neq 0$ then the class $\kap$ is nonzero. Since by Theorem~\ref{principal:wild} this is the bottom class of an anticyclotomic Euler system for $V_{gh}^\psi$, from Theorem~\ref{segon-res} and
Proposition~\ref{prop:suff} we deduce that $\Sel_{\bal}(K,\Arep)$ has $\mathcal{O}$-corank one, with
\begin{equation}\label{eq:bound-bal}
{\rm length}_{\mathcal{O}}(\Sel_{\bal}(K,\Arep)_{/{\rm div}})\leq 2\;{\rm length}_{\mathcal O}\biggl(\frac{\Sel_{\bal}(K,T_{g,h}^\psi)}{\mathcal O \cdot\kap}\biggr).
\end{equation}
By the exact sequence (\ref{PT-1}) specialized to weight $k$,  
it follows that $\Sel_{\cF\cup +}(K,A_{g,h}^\psi)$ has also $\mathcal{O}$-corank one. Thus both $\Sel_{\bal}(K,T_{g,h}^\psi)\subset\Sel_{\cF\cup+}(K,T_{g,h}^\psi)$ have $\cO$-rank one, and therefore
\begin{equation}\label{eq:bal-+}
\Sel_{\bal}(K,T_{g,h}^\psi)=\Sel_{\cF\cup+}(K,T_{g,h}^\psi),
\end{equation}
since their quotient is $\cO$-torsion free. Moreover, letting $\pi\in\mathcal{O}$ be a uniformizer, as in the proof of Lemma~\ref{rangs} we find that
\[
\Sel_{\cF\cup+}(K,A_{g,h}^\psi)[\pi^i]\cong E/\cO[\pi^i]\oplus\Sel_{\cF\cap+}(K,A_{g,h}^\psi)[\pi^i]
\]
for all $i$, and hence ${\rm length}_{\cO}({\rm Sel}_{\cF\cup+}(K,A_{g,h}^\psi)_{/{\rm div}})={\rm length}_{\cO}({\rm Sel}_{\cF\cap +}(K,A_{g,h}^\psi))$.

The finiteness of $\Sel_{\cF}(K,A_{g,h}^\psi)$ with the stated bound on its $\cO$-length thus follows from (\ref{eq:bound-bal}) by the same argument as in the proof of Theorem~\ref{thm:equiv}, 
noting that by 
Theorem~\ref{rec-law} and the same calculation as in \cite[\S{8.5}]{BSV} (see esp. the equality following [\emph{op.\,cit.}, (189)]) the map
\[
\xi_k\cdot\left\langle{\rm exp}^*_p(-),\eta_{\breve{f}}\otimes\omega_{\breve{g}}\otimes\omega_{\breve{h}}\right\rangle
\]
where $f=\theta_\psi$ and $\xi_k$ is the weight $k$ specialization of the congruence ideal generator $\xi\in\Lambda_{\hf}$, gives an isomorphism $H^1(\Q_p,T_{f}^{gh})\rightarrow\cO$ taking $\kap$ to
\[
\frac{(l-2)!\cdot(m-2)!}{(k-c-1)!}\cdot\frac{\mathcal{E}_0(\hf_k)\cdot\mathcal{E}_1(\hf_k)}{\mathcal{E}(\hf_k,g,h)}\cdot \mathscr{L}_p^\xi(\underline{\breve{\hf}},\underline{\breve{g}},\underline{\breve{h}})(k),
\]
where $\mathcal{E}_0(\hf_k)=\bigl(1-\frac{\beta_k}{\alpha_k}\bigr)$ is a $p$-adic unit.

More precisely, under the freeness assumption in the statement, the weight $k$ specializations of (\ref{PT-1}) and (\ref{PT-2}) yield the exact sequences
\begin{equation}\label{eq:pt1}
0\longrightarrow{\rm coker}({\rm res}_p)\longrightarrow\Sel_{\cF\cup+}(K,A_{g,h}^\psi)^\vee\longrightarrow\Sel_{\bal}(K,A_{g,h}^\psi)^\vee\longrightarrow 0,
\end{equation}
\begin{equation}\label{eq:pt2}
0\longrightarrow{\rm coker}({\rm res}_p)\longrightarrow\Sel_{\cF}(K,A_{g,h}^\psi)^\vee\longrightarrow\Sel_{\cF\cap +}(K,A_{g,h}^\psi)^\vee\nonumber\longrightarrow 0,
\end{equation}
where the two terms ${\rm coker}({\rm res}_p)$ are equal in light of (\ref{eq:bal-+}). Thus we find
\begin{align*}
{\rm lt}_{\cO}(\Sel_{\cF}(K,A_{g,h}^\psi))={\rm lt}_{\cO}(\Sel_{\cF}(K,A_{g,h}^\psi)^\vee)&={\rm lt}_{\cO}(\Sel_{\cF\cap+}(K,A_{g,h}^\psi)^\vee)+{\rm lt}_{\cO}({\rm coker}({\rm res}_p))\\
&={\rm lt}_{\cO}((\Sel_{\cF\cup+}(K,A_{g,h}^\psi)_{/{\rm div}})^\vee)+{\rm lt}_{\cO}({\rm coker}({\rm res}_p))\\
&={\rm lt}_{\cO}((\Sel_{\bal}(K,A_{g,h}^\psi)_{/{\rm div}})^\vee)+2\,{\rm lt}_{\cO}({\rm coker}({\rm res}_p))\\
&={\rm lt}_{\cO}(\Sel_{\bal}(K,A_{g,h}^\psi)_{/{\rm div}})+2\,{\rm lt}_{\cO}({\rm coker}({\rm res}_p)),
\end{align*}
where the third equality follows from (\ref{eq:pt1}) and Lemma~\ref{lem:lts} below, concluding the proof.
\end{proof}

\begin{lemma}\label{lem:lts}
Let $0\rightarrow A\xrightarrow{j} B\rightarrow C\rightarrow 0$ be an exact sequence of finitely generated $\cO$-modules, and assume that $A$ is finite. Then $B_{\rm tors}/j(A)\cong C_{\rm tors}$.

In particular, if $B', C'$ are cofinitely generated $\cO$-modules and we have an exact sequence $0\rightarrow A\xrightarrow{j} (B')^\vee\rightarrow (C')^\vee\rightarrow 0$ with $A$ finite, then
\[
(B'_{/{\rm div}})^\vee/j(A)\cong (C'_{/{\rm div}})^\vee,
\]
and so ${\rm lt}_{\cO}((B'_{/{\rm div}})^\vee)={\rm lt}_{\cO}(A)+{\rm lt}_{\cO}((C'_{/{\rm div}})^\vee)$.
\end{lemma}

\begin{proof}
Writing $B\cong\cO^r\oplus B_{\rm tors}$, $C\cong\cO^s\oplus C_{\rm tors}$ we have, by the finiteness of $A$, $r=s$ and $j(A)\subset B_{\rm tors}$, so
\[
\cO^r\oplus C_{\rm tors}\cong C\cong B/j(A)\cong\cO^r\oplus(B_{\rm tors}/j(A)),
\]
which implies the result.
\end{proof}

\begin{remark}
The condition that $H^1(\bb{Q}_p,T_{f}^{gh})$ is torsion-free is equivalent to the vanishing of $H^0(\bb{Q}_p, A_{f}^{gh})$, which is satisfied if $k+2\neq l+m$ modulo $2(p-1)$ or if $\chi_f(p)\alpha_g\alpha_h/\alpha_k\neq 1$ modulo $p$. Similarly, the last condition in the statement of Theorem~\ref{thm:size-BK} can be recast in terms of the vanishing of the corresponding $0$-th cohomology groups.
\end{remark}

\begin{remark}
By Theorem~\ref{thm:hsieh}, the non-vanishing of $L(\Vrep,0)$ implies that $\mathscr{L}_p^\xi(\underline{\breve{\hf}},\underline{\breve{g}},\underline{\breve{h}})(x)\neq 0$, so the upper bound provided by Theorem~\ref{thm:size-BK} is non-trivial. Moreover, by the interpolation formula in Theorem~\ref{thm:hsieh}, this upper bound can be expressed in terms of the central $L$-value $L(\Vrep,0)$, thus giving a result towards the Tamagawa number conjecture of \cite{BK}.
\end{remark}

\subsection{Proof of Theorem~\ref{thmD}}

As before, let $\hf$ be the Hida family attached to $\psi$ as in $\S\ref{subsec:Iw}$. Let $\kapinftyone$ be the $\Lambda$-adic class of conductor $n=1$ constructed in Theorem~\ref{principal:wild}, and set
\[
\kapinfty:=\kapinftyone\in H^1_{\rm Iw}(K_\infty,T_{g,h}^\psi).
\]
As noted before the proof of Proposition~\ref{prop:in-Selinfty}, under the Shapiro isomorphism
\[
H^1(\Q,\mathbb V_{\hf gh}^\dag)\cong H^1(K,\Lambda_{\cl{O}}(\bm{\kappa}_{\text{ac}}^{-1})\hat{\otimes}_{\cl{O}}\Trep)\cong H^1_{\rm Iw}(K_\infty,\Trep),
\]
the balanced Selmer group $\Sel_{\rm bal}(\Q,\mathbb V_{\hf gh}^\dag)$ of $\S\ref{subsec:IMC}$ is identified with the Greenberg Selmer group $\Sel_{\rm Gr}(K_\infty,\Trep)$ of $\S\ref{sec:ES}$ attached to $G_{K_v}$-invariant subspaces $\mathcal{F}_v^+(V_{g,h}^\psi)\subset V_{g,h}^\psi$ in (\ref{eq:bal-Sh}) at the primes $v\mid p$. Moreover, under this isomorphism, the class $\kappa(\hf,g,h)$ in $\S\ref{subsec:ERL}$
corresponds to the class $\kapinfty$.

The next result, establishing one of the divisibilities predicted by the Iwasawa main conjectures from $\S\ref{subsec:IMC}$, recovers Theorem~\ref{thmD} in the Introduction.

\begin{theorem}
Assume hypotheses (h1)--(h8), and in addition that:
\begin{itemize}
\item $\varepsilon_\ell(\theta_\psi,g,h)=+1$ for all primes $\ell\mid N$,
\item ${\rm gcd}(N_\psi,N_g,N_h)$ is squarefree.
\end{itemize}
If $\kappa(\hf,g,h)$ is not $\Lambda_{\hf}$-torsion, then the following hold:
\begin{enumerate}
\item[{\rm (a)}] The modules $\Sel_{\bal}(\mathbb V_{\hf gh}^{\dag})$ and $X_{\bal}(\mathbb A_{\hf gh}^{\dag})$ have both $\Lambda_{\hf}$-rank one, and
    \[
    \Char_{\Lambda_{\hf}}(X_{\bal}(\mathbb A_{\hf gh}^{\dag})_{\tors}) \supset
    \Char_{\Lambda_{\hf}}\left(\frac{\Sel_{\bal}(\mathbb V_{\hf gh}^{\dag})}{\Lambda_{\hf} \cdot\kappa(\hf,g,h)} \right)^2.
    \]
\item[{\rm (b)}] The modules $\Sel_{\mathcal F}(\mathbb V_{\hf gh}^{\dag})$ and $X_{\mathcal F}(\mathbb V_{\hf gh}^{\dag})$ are both $\Lambda_{\hf}$-torsion, and
    \[
    \Char_{\Lambda_{\hf}}(X_{\mathcal F}(\mathbb A_{\hf gh}^{\dag})) \supset (L_p(\hf,g,h))
    \]
in $\Lambda_{\hf}\otimes_{\mathbb Z_p}\mathbb Q_p$.
\end{enumerate}
\end{theorem}

\begin{proof}
The non-triviality assumption on $\kappa(\hf,g,h)$ implies that $\kapinfty$ is not $\Lambda_{\hf}$-torsion.
Since by Theorem~\ref{principal:wild} the class $\kapinfty$ is the bottom class of a $\Lambda$-adic Euler system for $V_{g,h}^\psi$, part (a) follows from Theorem~\ref{segon-res} and Proposition~\ref{prop:suff}. By Theorem~\ref{thm:equiv}, part (b) of the theorem follows from part (a), so this concludes the proof.
\end{proof}

\bibliographystyle{amsalpha}
\bibliography{Schoen-refs}

\providecommand{\bysame}{\leavevmode\hbox to3em{\hrulefill}\thinspace}
\providecommand{\MR}{\relax\ifhmode\unskip\space\fi MR }
\providecommand{\MRhref}[2]{%
  \href{http://www.ams.org/mathscinet-getitem?mr=#1}{#2}
}
\providecommand{\href}[2]{#2}
\begin{thebibliography}{BDR15b}

\bibitem[ACR22]{ACR2}
Raul Alonso, Francesc Castella, and \'{O}scar Rivero, \emph{An anticyclotomic
  {E}uler system for adjoint modular {G}alois representations}, preprint, {\tt
  arXiv:2204.07658} (2022).

\bibitem[AH06]{AH}
Adebisi Agboola and Benjamin Howard, \emph{Anticyclotomic {I}wasawa theory of
  {CM} elliptic curves}, Ann. Inst. Fourier (Grenoble) \textbf{56} (2006),
  no.~4, 1001--1048.

\bibitem[AS86a]{AS}
Avner Ash and Glenn Stevens, \emph{Cohomology of arithmetic groups and
  congruences between systems of {H}ecke eigenvalues}, J. Reine Angew. Math.
  \textbf{365} (1986), 192--220.

\bibitem[AS86b]{AS2}
\bysame, \emph{Modular forms in characteristic {$l$} and special values of
  their {$L$}-functions}, Duke Math. J. \textbf{53} (1986), no.~3, 849--868.

\bibitem[BDR15a]{BDR1}
Massimo Bertolini, Henri Darmon, and Victor Rotger, \emph{Beilinson-{F}lach
  elements and {E}uler systems {I}: {S}yntomic regulators and {$p$}-adic
  {R}ankin {$L$}-series}, J. Algebraic Geom. \textbf{24} (2015), no.~2,
  355--378.

\bibitem[BDR15b]{BDR2}
\bysame, \emph{Beilinson-{F}lach elements and {E}uler systems {II}: the
  {B}irch-{S}winnerton-{D}yer conjecture for {H}asse-{W}eil-{A}rtin
  {$L$}-series}, J. Algebraic Geom. \textbf{24} (2015), no.~3, 569--604.

\bibitem[Bei84]{beilinson}
A.~A. Beilinson, \emph{Higher regulators and values of {$L$}-functions},
  Current problems in mathematics, {V}ol. 24, Itogi Nauki i Tekhniki, Akad.
  Nauk SSSR, Vsesoyuz. Inst. Nauchn. i Tekhn. Inform., Moscow, 1984,
  pp.~181--238.

\bibitem[BK90]{BK}
Spencer Bloch and Kazuya Kato, \emph{{$L$}-functions and {T}amagawa numbers of
  motives}, The {G}rothendieck {F}estschrift, {V}ol.\ {I}, Progr. Math.,
  vol.~86, Birkh\"auser Boston, Boston, MA, 1990, pp.~333--400.

\bibitem[BL18]{BL1}
K\^{a}z\i{m} B\"{u}y\"{u}kboduk and Antonio Lei, \emph{Anticyclotomic
  {$p$}-ordinary {I}wasawa theory of elliptic modular forms}, Forum Math.
  \textbf{30} (2018), no.~4, 887--913.

\bibitem[BL21]{BL-IMRN}
\bysame, \emph{Iwasawa {T}heory of {E}lliptic {M}odular {F}orms {O}ver
  {I}maginary {Q}uadratic {F}ields at {N}on-ordinary {P}rimes}, Int. Math. Res.
  Not. IMRN (2021), no.~14, 10654--10730.

\bibitem[BLLV19]{BKLV}
K\^{a}z\i{m} B\"{u}y\"{u}kboduk, Antonio Lei, David Loeffler, and Guhan Venkat,
  \emph{Iwasawa theory for {R}ankin-{S}elberg products of {$p$}-nonordinary
  eigenforms}, Algebra Number Theory \textbf{13} (2019), no.~4, 901--941.

\bibitem[BSV22]{BSV}
Massimo Bertolini, Marco~Adamo Seveso, and Rodolfo Venerucci, \emph{Reciprocity
  laws for balanced diagonal classes}, Ast\'{e}risque (2022), no.~434, 77--174.

\bibitem[Car86]{carayol}
Henri Carayol, \emph{Sur les repr\'{e}sentations {$l$}-adiques associ\'{e}es
  aux formes modulaires de {H}ilbert}, Ann. Sci. \'{E}cole Norm. Sup. (4)
  \textbf{19} (1986), no.~3, 409--468.

\bibitem[Cas17]{Cas}
Francesc Castella, \emph{{$p$}-adic heights of {H}eegner points and
  {B}eilinson-{F}lach classes}, J. Lond. Math. Soc. (2) \textbf{96} (2017),
  no.~1, 156--180.

\bibitem[CD23]{CaDo}
Francesc Castella and Kim~Tuan Do, \emph{Diagonal cycles and anticyclotomic
  {I}wasawa theory of modular forms}, preprint, {\tt arXiv:2303.06751} (2023).

\bibitem[Do22]{Do-PhD}
Kim~Tuan Do, \emph{Construction of an anticyclotomic {E}uler system with
  applications}, Ph.D. thesis, Princeton University, 2022.

\bibitem[DR14]{DR1}
Henri Darmon and Victor Rotger, \emph{Diagonal cycles and {E}uler systems {I}:
  {A} {$p$}-adic {G}ross-{Z}agier formula}, Ann. Sci. \'Ec. Norm. Sup\'er. (4)
  \textbf{47} (2014), no.~4, 779--832.

\bibitem[DR17]{DR2}
\bysame, \emph{Diagonal cycles and {E}uler systems {II}: {T}he {B}irch and
  {S}winnerton-{D}yer conjecture for {H}asse-{W}eil-{A}rtin {$L$}-functions},
  J. Amer. Math. Soc. \textbf{30} (2017), no.~3, 601--672.

\bibitem[DR22]{DR3}
\bysame, \emph{{$p$}-adic families of diagonal cycles}, Ast\'{e}risque (2022),
  no.~434, 29--75.

\bibitem[Fla92]{flach-sym}
Matthias Flach, \emph{A finiteness theorem for the symmetric square of an
  elliptic curve}, Invent. Math. \textbf{109} (1992), no.~2, 307--327.

\bibitem[GK92]{gross-kudla}
Benedict~H. Gross and Stephen~S. Kudla, \emph{Heights and the central critical
  values of triple product {$L$}-functions}, Compositio Math. \textbf{81}
  (1992), no.~2, 143--209.

\bibitem[Gre94]{Gr94}
Ralph Greenberg, \emph{Iwasawa theory and {$p$}-adic deformations of motives},
  Motives ({S}eattle, {WA}, 1991), Proc. Sympos. Pure Math., vol.~55, Amer.
  Math. Soc., Providence, RI, 1994, pp.~193--223.

\bibitem[GS93]{GrSt}
Ralph Greenberg and Glenn Stevens, \emph{{$p$}-adic {$L$}-functions and
  {$p$}-adic periods of modular forms}, Invent. Math. \textbf{111} (1993),
  no.~2, 407--447.

\bibitem[GS95]{gross-schoen}
B.~H. Gross and C.~Schoen, \emph{The modified diagonal cycle on the triple
  product of a pointed curve}, Ann. Inst. Fourier (Grenoble) \textbf{45}
  (1995), no.~3, 649--679.

\bibitem[GS20]{GS}
Matthew Greenberg and Marco~Adamo Seveso, \emph{Triple product {$p$}-adic
  {$L$}-functions for balanced weights}, Math. Ann. \textbf{376} (2020),
  no.~1-2, 103--176.

\bibitem[HK91]{harris-kudla}
Michael Harris and Stephen~S. Kudla, \emph{The central critical value of a
  triple product {$L$}-function}, Ann. of Math. (2) \textbf{133} (1991), no.~3,
  605--672.

\bibitem[How04]{How}
Benjamin Howard, \emph{The {H}eegner point {K}olyvagin system}, Compos. Math.
  \textbf{140} (2004), no.~6, 1439--1472.

\bibitem[How07]{How2}
\bysame, \emph{Variation of {H}eegner points in {H}ida families}, Invent. Math.
  \textbf{167} (2007), no.~1, 91--128.

\bibitem[Hsi21]{Hs}
Ming-Lun Hsieh, \emph{Hida families and $p$-adic triple product
  {$L$}-functions}, American Journal of Mathematics \textbf{143} (2021), no.~2,
  411--532.

\bibitem[HT01]{harris-tilouine}
Michael Harris and Jacques Tilouine, \emph{{$p$}-adic measures and square roots
  of special values of triple product {$L$}-functions}, Math. Ann. \textbf{320}
  (2001), no.~1, 127--147.

\bibitem[HY]{HY}
Ming-Lun Hsieh and Shunsuke Yamana, \emph{Derivatives of cyclotomic triple
  product {$L$}-functions and $p$-adic heights of diagonal cycles}, in
  preparation.

\bibitem[JNS]{JNS}
Dimitar Jetchev, Jan Nekov{\'a}{\v{r}}, and Christopher Skinner, preprint.

\bibitem[Kat04]{Kato}
Kazuya Kato, \emph{{$p$}-adic {H}odge theory and values of zeta functions of
  modular forms}, Ast\'erisque (2004), no.~295, ix, 117--290, Cohomologies
  $p$-adiques et applications arithm{\'e}tiques. III.

\bibitem[KLZ17]{KLZ}
Guido Kings, David Loeffler, and Sarah~Livia Zerbes, \emph{Rankin-{E}isenstein
  classes and explicit reciprocity laws}, Camb. J. Math. \textbf{5} (2017),
  no.~1, 1--122.

\bibitem[KLZ20]{KLZ0}
\bysame, \emph{Rankin-{E}isenstein classes for modular forms}, Amer. J. Math.
  \textbf{142} (2020), no.~1, 79--138.

\bibitem[LLZ14]{LLZ0}
Antonio Lei, David Loeffler, and Sarah~Livia Zerbes, \emph{Euler systems for
  {R}ankin-{S}elberg convolutions of modular forms}, Ann. of Math. (2)
  \textbf{180} (2014), no.~2, 653--771.

\bibitem[LLZ15]{LLZ}
\bysame, \emph{Euler systems for modular forms over imaginary quadratic
  fields}, Compos. Math. \textbf{151} (2015), no.~9, 1585--1625.

\bibitem[Loe17]{Loe}
David Loeffler, \emph{Images of adelic {G}alois representations for modular
  forms}, Glasg. Math. J. \textbf{59} (2017), no.~1, 11--25.

\bibitem[LZ14]{LZ2}
David Loeffler and Sarah~Livia Zerbes, \emph{Iwasawa theory and {$p$}-adic
  {$L$}-functions over {$\Bbb{Z}_p^2$}-extensions}, Int. J. Number Theory
  \textbf{10} (2014), no.~8, 2045--2095.

\bibitem[MR04]{MR}
Barry Mazur and Karl Rubin, \emph{Kolyvagin systems}, Mem. Amer. Math. Soc.
  \textbf{168} (2004), no.~799, viii+96.

\bibitem[Nek93]{nekovar-ht}
Jan Nekov\'{a}\v{r}, \emph{On {$p$}-adic height pairings}, S\'{e}minaire de
  {T}h\'{e}orie des {N}ombres, {P}aris, 1990--91, Progr. Math., vol. 108,
  Birkh\"{a}user Boston, Boston, MA, 1993, pp.~127--202.

\bibitem[NN16]{NN16}
Jan Nekov{\'a}{\v{r}} and Wies{\l{}}awa Nizio{\l{}}, \emph{Syntomic cohomology
  and {$p$}-adic regulators for varieties over {$p$}-adic fields}, Algebra
  Number Theory \textbf{10} (16), no.~8, 1695--1790.

\bibitem[PR87]{PR}
Bernadette Perrin-Riou, \emph{Fonctions {$L$} {$p$}-adiques, th\'eorie
  d'{I}wasawa et points de {H}eegner}, Bull. Soc. Math. France \textbf{115}
  (1987), no.~4, 399--456.

\bibitem[PR00]{PR-book}
\bysame, \emph{{$p$}-adic {$L$}-functions and {$p$}-adic representations},
  SMF/AMS Texts and Monographs, vol.~3, American Mathematical Society,
  Providence, RI, 2000, Translated from the 1995 French original by Leila
  Schneps and revised by the author.

\bibitem[PW11]{PW}
Robert Pollack and Tom Weston, \emph{On anticyclotomic {$\mu$}-invariants of
  modular forms}, Compos. Math. \textbf{147} (2011), no.~5, 1353--1381.

\bibitem[Rub00]{Rub}
Karl Rubin, \emph{Euler systems}, Annals of Mathematics Studies, vol. 147,
  Princeton University Press, Princeton, NJ, 2000, Hermann Weyl Lectures. The
  Institute for Advanced Study.

\bibitem[Wan20]{wan}
Xin Wan, \emph{Heegner point {K}olyvagin system and {I}wasawa main conjecture},
  Acta Math. Sin. \textbf{37} (2020), no.~1, 104--120.

\bibitem[YZZ]{YZZ}
Xinyi Yuan, Shouwu Zhang, and Wei Zhang, \emph{Triple product {$L$}-series and
  {G}ross--{K}udla--{S}choen cycles}, preprint.

\end{thebibliography}

\end{document}